\documentclass[twoside]{article}
\usepackage[english]{babel}
\usepackage{amssymb,amsmath,amsthm}
\usepackage{color}
\newcommand{\ds}{\displaystyle}

\def\nt{\noindent}
\newtheorem{de}{Definition}[section]
\newtheorem{theor}{Theorem}[section]
\newtheorem{pr}{Proposition}[section]
\newtheorem{lm}{Lemma}[section]

\begin{document}
\begin{center}
{\Large \textbf{On Description of Isomorphism Classes of \\ filiform Leibniz
algebras in dimensions 7 and 8}}\\

{ $^1$Isamiddin S. Rakhimov and $^2$Munther A. Hassan}

$ \ ^{{1},{2}}$Institute for Mathematical Research (INSPEM) $\&$

$^1$Department of Mathematics, FS, Universiti Putra Malaysia \\ 43400, UPM,
Serdang, Selangor Darul Ehsan, Malaysia,\\[0pt]

$^{1}$risamiddin@gmail.com \  $^{2}$munther\_abd@yahoo.com
\end{center}

\author{ Isamiddin S. Rakhimov$^{1}$, Munther A. Hassan $^{2},$ $}
\title{\mathbf{On low-dimensional filiform Leibniz
algebras \\ and their invariants}}

\begin{abstract} The paper concerns the classification problem of a subclass of complex filiform
Leibniz algebras in dimensions 7 and 8. This subclass arises from
naturally graded filiform Lie algebras. We give a complete list of
isomorphism classes of algebras including Lie case. In parametric families cases, the
corresponding orbit functions (invariants) are given. In discrete
orbits case, we show representatives of the orbits.
\end{abstract}



\thispagestyle{empty}

\section{Introduction}

 \qquad  Leibniz algebras were introduced by J. -L. Loday
\cite{L},\cite{LP}. A skew-symmetric Leibniz algebra is a Lie
algebra. The main motivation of J. -L. Loday to introduce this
class of algebras was the search of an ``obstruction'' to the periodicity in algebraic
$K-$theory. Besides this purely algebraic motivation, some
relationships with classical geometry, non-commutative geometry
and physics have been recently discovered. The present paper deals
with the low-dimensional case of subclass of filiform Leibniz
algebras. This subclass, arises from naturally graded filiform Lie
algebras.
\begin{de} An algebra L over a field K is called a Leibniz algebra, if its
bilinear operation $[\cdot,\cdot]$ satisfies the following Leibniz
identity:
 $\ \ds [x,[y,z]]=[[x,y],z]-[[x,z],y],$\ \ \ \ for
\ any $ x,y,z\in L. $
\end{de}

Onward, all algebras are assumed to be over the fields of complex
numbers $\mathbb{C}.$

Let $L$ be a Leibniz algebra. We put:
\begin{equation*}
L^{1}=L,\ L^{k+1}=[L^{k},L],\ k\geq 1.
\end{equation*}
\begin{de}
A Leibniz algebra L is said to be nilpotent, if there exists $s\in
\mathbb{N},$ such that
\begin{equation*}L^{1}\supset  L^{2}\supset ...\supset L^{s}=\{0\}.\end{equation*}
\end{de}
\begin{de}
A Leibniz algebra $L$ is said to be filiform, if $ \ds dimL^{i}=n-i,$ where $%
n=dimL,$ and $2\leq i\leq n.$
\end{de}

We denote by $Leib_n$ the set of all $n-$dimensional filiform
Leibniz algebras.

  The following theorem from \cite{GO} splits the set of
 fixed dimension filiform Liebniz algebras in to three disjoint
 subsets.
\begin{theor}\label{T1}
Any $(n+1)-$dimensional complex filiform Leibniz algebra L admits
a basis $\{e_0,e_1,...,e_n\}$ called adapted, such that the table
of multiplication of L has one of the following forms, where non
defined products are zero:
\end{theor}
$\bigskip FLeib_{n+1}=\left\{
\begin{array}{ll}
\lbrack e_{0},e_{0}]=e_{2}, &  \\[1mm]
\lbrack e_{i},e_{0}]=e_{i+1}, & \ \ \ \ \ \ 1\leq i\leq {n-1}, \\[1mm]
\lbrack e_{0},e_{1}]=\alpha _{3}e_{3}+\alpha _{4}e_{4}+...+\alpha
_{n-1}e_{n-1}+\theta e_{n}, &  \\[1mm]
\lbrack e_{j},e_{1}]=\alpha _{3}e_{j+2}+\alpha
_{4}e_{j+3}+...+\alpha
_{n+1-j}e_{n}, & \ \ \ \ \ \ 1\leq j\leq {n-2}, \\
\qquad \alpha _{3},\alpha _{4},...,\alpha _{n},\theta \in
\mathbb{C}. &
\end{array}%
\right. $ \\[1mm]

$ SLeib_{n+1}=\left\{
\begin{array}{ll}
\lbrack e_{0},e_{0}]=e_{2}, &  \\[1mm]
\lbrack e_{i},e_{0}]=e_{i+1}, & \  \qquad \ \ \ 2\leq i\leq {n-1}, \\[1mm]
\lbrack e_{0},e_{1}]=\beta _{3}e_{3}+\beta _{4}e_{4}+...+\beta
_{n}e_{n}, &
\\[1mm]
\lbrack e_{1},e_{1}]=\gamma e_{n}, &  \\[1mm]
\lbrack e_{j},e_{1}]=\beta _{3}e_{j+2}+\beta _{4}e_{j+3}+...+\beta
_{n+1-j}e_{n}, & \  \qquad \quad 2\leq j\leq {n-2}, \\
\qquad \beta _{3},\beta _{4},...,\beta _{n},\gamma \in \mathbb{C}.
&
\end{array}%
\right. $

\bigskip

$ TLeib_{n+1}=\left\{
\begin{array}{lll}
\lbrack e_{i},e_{0}]=e_{i+1}, \qquad  \qquad  \qquad \qquad  \qquad  \qquad \qquad  \qquad  \quad 1\leq i\leq {n-1},  \\[1mm]
\lbrack e_{0},e_{i}]=-e_{i+1}, \qquad \qquad  \qquad  \qquad \qquad  \qquad  \qquad \qquad  \   2\leq i\leq {n-1}, \\[1mm]
\lbrack e_{0},e_{0}]=b_{0,0}e_{n},  \\[1mm]
\lbrack e_{0},e_{1}]=-e_{2}+b_{0,1}e_{n}, \\[1mm]
\lbrack e_{1},e_{1}]=b_{1,1}e_{n},  \\[1mm]
\lbrack e_{i},e_{j}]=a_{i,j}^{1}e_{i+j+1}+\dots
+a_{i,j}^{n-(i+j+1)}e_{n-1}+b_{i,j}e_{n}, \quad   1\leq i<j\leq {n-1},   \\[1mm]
\lbrack e_{i},e_{j}]=-[e_{j},e_{i}], \qquad  \qquad  \qquad \qquad  \qquad  \qquad \qquad  \quad \ \  1\leq i<j\leq n-1,  \\[1mm]
\lbrack e_{i},e_{n-i}]=-[e_{n-i},e_{i}]=(-1)^{i}b_{i,n-i} e_{n}, \\[1mm]
\emph{\rm {where} $a_{i,j}^k, b_{i,j} \in \mathbb{C},$\
$b_{i,n-i}=b$ whenever $1\leq i\leq {n-1},$ and $b=0$ for even
$n.$}
\end{array}%
\right. $

It should be mentioned that in the theorem above the structure constants $\alpha _{3},\alpha _{4},...,\alpha _{n},\theta$ and $\beta _{3},\beta _{4},...,\beta _{n},\gamma$ in the first two cases are free, however, in the third case there are relations among $a_{i,j}^k, b_{i,j}$ that are different in each fixed dimension (see Lemma 2.1 and Lemma 2.2). According to the theorem each class has the so-called adapted base change sending an adapted basis to adapted and they can be studied autonomously. The classes
$FLeib_{n},SLeib_{n}$ in low dimensional cases, have been considered in
\cite{RS1}, \cite{RS2}. The general methods of classification for
$Leib_{n}$ has been given in \cite{BR1}, \cite{BR2} and \cite{OR}.
This paper deals with the classification problem of
low-dimensional cases of $TLeib_{n}.$

Note that the class $TLeib_{n}$ contains all $n-$dimensional filiform Lie algebras.

The outline of the paper is as follows. Section 1 is an
introduction to a subclass of Leibniz algebras that we are going
to investigate. Section 2 presents the main results of the paper
consisting of a complete classification of a subclass of low
dimensional filiform Leibniz algebras. Here, for 7- and
8-dimensional cases we give complete classification. For
parametric family cases the corresponding invariant functions are
presented.

\begin{de}
Let $\{e_0,e_1,...,e_n\}$ be an adapted basis of $L$ from ${TLeib}_{n+1}.$ Then a nonsingular linear transformation
$f:L\rightarrow L$ is said to be adapted if the basis
 $ \{f(e_0),f(e_1),...,f(e_n)\} $ is adapted.
\end{de}

The subgroup of $GL_{n+1}$ consisting of all adapted transformations is denoted by $G_{ad}.$ The following proposition specifies
elements of $G_{ad}.$
\begin{pr}
Any adapted transformation $f$ in ${TLeib}_{n+1}$ can be represented as follows:\\
$\ds f(e_0)=e_{0}^{\prime }=\sum_{i=0}^{n}A_{i}e_{i},\ \ \ \
f(e_1)=e_{1}^{\prime }=\sum_{i=1}^{n}B_{i}e_{i}, \ \
f(e_i)=e_{i}^{\prime }=[f(e_{i-1}),f(e_0)], \ \ \ \ \ \ \ \ \
2\leq i \leq n, $

 $A_0, A_i, B_j,\ (i,j=1,...,n)$ {\em are complex
numbers and } $A_0\,B_1(A_0+A_1b)\neq0.$

\begin{proof}\emph{}

Since a filiform Leibniz algebra is 2-generated (see Theorem 1.1.)
it is sufficient to consider the adapted action of $f$ on the
generators $e_0, e_1:$

$$\ds f(e_0)=e_{0}^{\prime }=\sum_{i=0}^{n}A_{i}e_{i},\ \
f(e_1)=e_{1}^{\prime }=\sum_{i=0}^{n}B_{i}e_{i}.$$ \\ Then $\ds
f(e_i)=[f(e_{i-1}),f(e_0)]=A^{i-2}_0(A_1B_0-A_0B_1)e_i+\sum_{j=i+1}^{n}(*)e_j,\
\ \ 2\leq i \leq n.$ \\ Note that $A_0\neq0,\
(A_1B_0-A_0B_1)\neq0,$ otherwise $f(e_n)=0.$ The condition
$A_0\,B_1(A_0+A_1b)\neq0$ appears naturally since $f$ is not
singular.

 Let now consider $ \ds
[f(e_1),f(e_2)]=B_0\,(A_1\,B_0-A_0\,B_1)\,e_3+\sum_{j=4}^n
(*)e_j.$ Then for the basis $\{f(e_0),f(e_1),...,f(e_n)\}$ to be
adapted $B_0(A_1B_{0}-A_{0}B_1)=0.$ But according to the
observation above
  $(A_1B_0-A_0B_1)\neq0.$ Therefore $B_0=0.$
\end{proof}
\end{pr}

In $G_{ad}$ we specify the following transformations called
elementary:
\begin{eqnarray*}
 \tau( a,b,c)&=&\left\{\begin{array}{lll} \tau(e_0)=a\,e_0+b\,e_1, &
\\[1mm]
\tau(e_1)=c\,e_1, &  \quad a\,c\neq0,\\[1mm]
\tau(e_{i+1})=[\tau(e_i), \tau(e_0)], &  \quad 1\leq i\leq n-1,  \\[1mm]
\end{array} \right.\\[1mm]
\sigma(a,k)&=&\left\{\begin{array}{lll} \sigma(e_0)=e_0+a\,e_k, &
 \quad 2\leq k \leq n, &
\\[1mm]
\sigma(e_1)=e_1, & \\[1mm]
\sigma(e_{i+1})=[\sigma(e_i), \sigma(e_0)], & \quad 1\leq i\leq n-1,  \\[1mm]
\end{array} \right.\\[1mm]
\phi(c,k)&=&\left\{\begin{array}{lll}\phi(e_0)=e_0, &
\\[1mm]
\phi(e_1)=e_1+c\,e_k, & \quad 2\leq k \leq n,
\\[1mm]
\phi(e_{i+1})=[\phi(e_i), \phi(e_0)], & \quad 1\leq i\leq n-1,
\end{array}\right.
\end{eqnarray*}
where $a,b,c \in \mathbb{C}$.

\begin{pr}\label{p} Let $L$ be an algebra from ${{TLeib}}_{n+1}$, then any adapted transformation $f$ can be represented
as the composition:
$$f=\phi(B_n,n)\circ\phi(B_{n-1},n-1)\circ...\circ\phi(B_2,2)\circ\sigma(A_n,n)\circ\sigma(A_{n-1},n-1)\circ...\circ\sigma(A_2,2)\circ\tau(A_0,A_1,B_1)
.$$

\end{pr}
\begin{proof} The proof is straightforward.
\end{proof}

\begin{pr}\label{p1} The transformations\\
$\ds
g=\phi(B_n,n)\circ\phi(B_{n-1},n-1)\circ\phi(B_{n-2},n-2)\circ\sigma(A_n,n)\circ\sigma(A_{n-1},n-1)\circ\sigma(A_{n-2},n-2)\circ\sigma(A_{n-3},{n-3})
,$ if $n$ even, and\\
$g=\phi(B_n,n)\circ\phi(B_{n-1},n-1)\circ\sigma(A_n,n)\circ\sigma(A_{n-1},n-1)\circ\sigma(A_{n-2},{n-2})
,$ for odd $n$ \\ does not change the structure constants of algebras from
$TLeib_{n+1}.$
\begin{proof} Let us prove the assertion when $n$ is even.
\begin{itemize}
\item Consider the transformation,
$$\sigma(A_{n-3},{n-3})=\left\{\begin{array}{lll} \sigma(e_0)=e_0+A_{n-3}\,e_{n-3}, &
  &
\\[1mm]
\sigma(e_1)=e_1, & \\[1mm]
\sigma(e_{i+1})=[\sigma(e_i), \sigma(e_0)], & \quad 1\leq i\leq
n-1.\end{array}\right.$$

 To show that the transformation
$\sigma(A_{n-3},n-3),$ does not change the structure constants,
note that $
\sigma\left(e_2\right)=e_2+(*)\,A_{n-3}e_{n-1}+(**)A_{n-3}e_{n},$
$\sigma\left(e_3\right)=e_3+(\star)\,A_{n-3}e_{n},$ and
$\sigma\left(e_i\right)=e_i,\ \forall\ i \geq 4,$ and by a simple
computation one can see that the transformation
$\sigma(A_{n-3},n-3)$ does not change the structure constants.

\item The transformation $\sigma(A_{n-2},n-2),$ does not change
the structure constants, because
$\sigma(e_2)=e_2-A_{n-2}\,(*)\,e_n,$ and $\sigma(e_i)=e_i,$ $
\forall\ i\geq 3.$

\item For the transformation $\sigma(A_{n-1},n-1),$ it is enough show that $\sigma(e_2)=e_2,$
note that
$\sigma(e_2)=\left[\sigma(e_1),\sigma(e_0)\right]=e_2+A_{n-1}[e_{n-1},e_1],$
since $n$ is even then $[e_{n-1},e_1]=0,$ and hence
$\sigma(e_2)=e_2$\\
\end{itemize}
It is easy to
see that $\sigma(A_n,n),$ does not change the structure constants.
\newpage
Analogously, we check that $\phi(B_k,k)$ does not change the
structure constants, when $n-2\leq k\leq n.$

Consider the transformation\\ \hspace*{-5pt}
\begin{center} {$\phi(B_k,k)=\left\{\begin{array}{lll}\phi(e_0)=e_0,&
\\[1mm]
\phi(e_1)=e_1+B_k\,e_k, & \quad n-2\leq k \leq n,
\\[1mm]
\phi(e_{i+1})=[\phi(e_i), \phi(e_0)], & \quad 1\leq i\leq n-1,
\end{array}\right.$}
\end{center}
\begin{itemize}
\item \quad If $k=n-2,$ then
$$\phi(e_2)=e_2+B_{n-2}e_{n-1},\ \phi(e_3)=e_3+B_{n-2}e_{n},\  \phi(e_i)=e_i,\ \mbox{where}\ i\geq4.$$
A simple computation shows that $\phi(B_{n-2},{n-2})$
  does not change the structure constants.

  Note that
  $\left[\phi(e_1),\phi(e_2)\right]=\left[e_1+B_{n-2}e_{n-2},e_2+B_{n-2}e_{n-1}\right]$\\
  $$=\left[e_1,e_2\right]
  +B_{n-2}\left[e_1,e_{n-1}\right]+B_{n-2}\left[e_{n-2},e_2\right]+B^2_{n-2}\left[e_{n-2},e_{n-1}\right]
=\left[e_1,e_2\right],$$ \\(here $
\left[e_1,e_{n-1}\right]=\left[e_{n-2},e_2\right]=0,$ since $n$ is
even).

\item \quad  If $k=n-1,$ then we get
$\phi(e_2)=e_2+B_{n-1}e_{n},\ \phi(e_i)=e_i,$ where $i\geq3.$

Consider the bracket
$$\left[\phi(e_0),\phi(e_1)\right]=-\phi(e_2)+b'_{0,1}\,e'_n, $$ and
then
$$\left[e_0,e_1+B_{n-1}e_{n-1}\right]=-e_2-B_{n-1}e_{n}+b'_{0,1}\,e_n,
$$ implies that
$$-e_2+b_{0,1}\,e_n-B_{n-1}e_{n}=-e_2-B_{n-1}e_{n}+b'_{0,1}\,e_n, $$
therefore  $b'_{0,1}=b_{0,1}.$

The chain of equalities
\begin{eqnarray*}\left[\phi(e_1),\phi(e_1)\right]&=&b'_{1,1}\,e'_n,\\[1mm]
\left[e_1+B_{n-1}e_{n-1},e_1+B_{n-1}e_{n-1}\right]&=&b'_{1,1}\,e_n,\\[1mm]
[e_1,e_1]+B_{n-1}[e_1,e_{n-1}]+B_{n-1}[e_{n-1},e_1]&=&b'_{1,1}\,e_n,\
{\rm{show \ that}}\ \ \ \ b'_{1,1}=b_{1,1}. \end{eqnarray*}

One easily can see that
$$\left[\phi(e_1),\phi(e_2)\right]=\left[e_1+B_{n-1}e_{n-1},e_2+B_{n-1}e_{n}\right]=[e_1,e_2]+B_{n-1}[e_1,e_{n}]+B_{n-1}[e_{n-1},e_2]=[e_1,e_2].
$$
\item \quad  If $k=n,$ it is obvious.
\end{itemize}

\end{proof}
\end{pr}

The following lemma from \cite{RH1} keeps track the behavior some of the structure constants under the adapted base change.

\begin{lm}\label{L1}   Let $\{e_0,e_1,...,e_n\}\longrightarrow \{e_0',e_1',...,e_n'\}$ be an adapted base change, $b_{0,0},b_{0,1},b_{1,1},...$ and $b^\prime_{0,0},b^\prime_{0,1},b^\prime_{1,1},...$ be the respective structure constants. Then for
$b^\prime_{0,0},b^\prime_{0,1}$\ and \ $b^\prime_{1,1}$ one has

$ \ds
b_{0,0}^\prime=\frac{A_0^2b_{0,0}+A_0A_1b_{0,1}+A_1^2b_{1,1}}{A_0^{n-2}B_1(A_0+A_1\,b)},\
\
b_{0,1}^\prime=\frac{A_0b_{0,1}+2A_1b_{1,1}}{A_0^{n-2}(A_0+A_1\,b)},\
\ \ b_{1,1}^\prime=\frac{B_1b_{1,1}}{A_0^{n-2}(A_0+A_1\,b)}.
$

\end{lm}

The next sections deal with the classification problem of ${TLeib}_n$ in dimensions 7 and 8.
Here, to classify algebras from ${TLeib}_{7}$ and  ${TLeib}_{8}$ we represent them as
a disjoint union of their subsets. Some of these subsets turn out to be  single
orbits, and the others contain infinitely many orbits. In the last
case, we give invariant functions to discern the orbits.

\section{ The description of ${TLeib}_{n}$,\ $n=7,8.$}
\subsection{Isomorphism criterion for ${TLeib}_{7}$}

\bigskip  Any algebra $L$ from ${TLeib}_{7}$ can be represented as one dimensional central extension ($C(L)=<e_6>$)
 of 6-dimensional filiform Lie algebra with adapted basis $\{e_0,e_1,...,e_5\}$ (see Theorem \ref{T1})
  and on the adapted basis $\{e_0,e_1,...,e_6\}$ the class ${TLeib}_{7}$ can be
represented as follows:

${TLeib}_{7}=\left\{
\begin{array}{lll}
\lbrack e_{i},e_{0}]=e_{i+1}, & 1\leq i\leq {5}, &  \\[1mm]
\lbrack e_{0},e_{i}]=-e_{i+1}, & 2\leq i\leq {5}, &  \\[1mm]
\lbrack e_{0},e_{0}]=b_{0,0}e_{6}, &  &  \\[1mm]
\lbrack e_{0},e_{1}]=-e_{2}+b_{0,1}e_{6}, &  &  \\[1mm]
\lbrack e_{1},e_{1}]=b _{1,1}e_{6}, &  &  \\[1mm]
\lbrack e_{1},e_{2}]=-[e_2,e_1]=a_{1,4}\,e_4+a_{1,5}\,e_5+b_{1,2}e_{6}, \\[1mm]
\lbrack e_{1},e_{3}]=-[e_{3},e_{1}]=a_{1,4}\,e_5+b_{1,3}e_{6}, \\[1mm]
 [e_{1},e_{4}]=-[e_{4},e_{1}]=-a_{2,5}\,e_5+b_{1,4}e_{6}, \\[1mm]
 [e_{2},e_{3}]=-[e_{3},e_{2}]=a_{2,5}\,e_5+b_{2,3}e_{6},\\[1mm]
 [e_{1},e_{5}]=-[e_{5},e_{1}]=b_{1,5}e_{6}, \\[1mm]
 [e_{2},e_{4}]=-[e_{4},e_{2}]=b_{2,4}e_{6}.\\[1mm]
\end{array}%
\right. $

The next lemma specifies the set of structure constants of algebras from ${TLeib}_{7}.$

\begin{lm} The structure constants of algebras from ${TLeib}_{7}$ satisfy the following constraints:
 $$1.\  \ds b_{1,3}=a_{1,5}, \ \  2.\ \ds b_{1,4}=a_{1,4}-b_{2,3}, \ \ \mbox{and}\ \ 3.\  \ds b_{1,5}=b_{2,4}=a_{2,5}=0.$$


\begin{proof}
The relations easily can be found by applying the Leibniz identity to
the triples of the basis vectors $\{e_0,e_1,e_2\},\
\{e_0,e_1,e_3\}$  $\{e_0,e_2,e_3\},\{e_0,e_1,e_4\},$ and $\{e_1,e_2,e_3\}. $
\end{proof}
\end{lm}
Further unifying the above table of multiplication we rewrite it via new parameters $c_{0,0},c_{0,1},c_{1,1},c_{1,2},c_{1,3},c_{1,4},c_{2,3}$ as follows:\\

${TLeib}_{7}=\left\{
\begin{array}{lll}
\lbrack e_{i},e_{0}]=e_{i+1}, & 1\leq i\leq {5}, &  \\[1mm]
\lbrack e_{0},e_{i}]=-e_{i+1}, & 2\leq i\leq {5}, &  \\[1mm]
\lbrack e_{0},e_{0}]=c_{0,0}e_{6}, &  &  \\[1mm]
\lbrack e_{0},e_{1}]=-e_{2}+c_{0,1}e_{6}, &  &  \\[1mm]
\lbrack e_{1},e_{1}]=c _{1,1}e_{6}, &  &  \\[1mm]
\lbrack e_{1},e_{2}]=-[e_2,e_1]=c_{1,2}\,e_4+c_{1,3}\,e_5+c_{1,4}e_{6}, \\[1mm]
\lbrack e_{1},e_{3}]=-[e_{3},e_{1}]=c_{1,2}\,e_5+c_{1,3}e_{6}, \\[1mm]
 [e_{1},e_{4}]=-[e_{4},e_{1}]=(c_{1,2}-c_{2,3})\,e_{6}, \\[1mm]
 [e_{2},e_{3}]=-[e_{3},e_{2}]=c_{2,3}\,e_{6}.
\end{array}%
\right. $

An algebra from ${TLeib}_{7}$ with parameters $c_{0,0},c_{0,1},c_{1,1},c_{1,2},c_{1,3},c_{1,4},c_{2,3}$ is denoted by $L(C),$ where $C=(c_{0,0},c_{0,1},c_{1,1},c_{1,2},c_{1,3},c_{1,4},c_{2,3}).$

The next theorem represents the action of the adapted base change to the parameters $c_{0,0},c_{0,1},c_{1,1},c_{1,2},c_{1,3},c_{1,4},c_{2,3}$ of an algebra from ${TLeib}_{7}.$

\begin{theor}\label{T2}\emph{(}Isomorphism criterion for ${TLeib}_{7}$\emph{)}

 Two filiform Leibniz algebras $L(C)$ and $L(C'),$ where
$C=(c_{0,0},c_{0,1},c_{1,1},c_{1,2},c_{1,3},c_{1,4},c_{2,3})$ and $C'=(c'_{0,0},c'_{0,1},c'_{1,1},c'_{1,2},c'_{1,3},c'_{1,4},c'_{2,3}),$
from ${TLeib}_7$ are isomorphic if and only if there exist $
A_0,A_1,B_1,B_2,B_3\in \mathbb{C}:$ such that $A_0B_1\neq 0$ and
the following equalities hold:

\begin{eqnarray}c_{0,0}^\prime&=&{\frac {{A_{{0}}}^{2}c_{{0,0}}+A_{{0}}A_{{1}}c_{{0,1}}+{A_{{1}}}^{2}c_
{{1,1}}}{{A_{{0}}}^{5}B_{{1}}}} ,
\\
c_{0,1}^\prime&=&{\frac
{2\,A_{{1}}c_{{1,1}}+A_{{0}}c_{{0,1}}}{{A_{{0}}}^{5}}}
,\\
c_{1,1}^\prime&=&\frac{B_1c_{1,1}}{{A_{{0}}}^{5}},\\
c_{1,2}^{\prime }&=&{\frac {B_{{1}}c_{{1,1}}}{{A_{{0}}}^{5}}},\\
c_{1,3}^{\prime }&=&{\frac {B_{{1}} \left(
A_{{0}}c_{{1,3}}+2\,A_{{1}}\,{c_{{1,2}}}^{2}
 \right) }{{A_{{0}}}^{4}}}
,\\
c_{1,4}^{\prime
}&=&\frac{{A_0}^2\,{B_1}^2c_{1,4}+{A_0}^2({B_2}^2-2B_1\,B_3)\,c_{2,3}+A_1\,{B_1}^2(5\,c_{1,2}-c_{2,3})\,(A_0\,c_{1,3}-B_1\,{c_{1,2}}^2)}{{A_{{0}}}^{6}B_{{1}}},\\
c_{2,3}^{\prime }&=& {\frac {B_{{1}}c_{{2,3}}}{{A_{{0}}}^{2}}}.
\end{eqnarray}

\begin{proof}``If'' part.\
 The equations (1)--(3) occur due to Lemma \ref{L1} (remind that in this case $n$ is even therefore $b=0$).

Notice that according  to the Proposition \ref{p} and \ref{p1} the
adapted transformation in ${TLeib}_{7}$ can be taken in the form

$$ \left\{\begin{array}{lll} e_0'=f(e_0)=A_0\,e_0 + A_1\,e_1, &
\\[1mm]
e_1'=f(e_1)=B_1\,e_1+B_2\,e_2+B_3\,e_3, & \\[1mm]
e_{i+1}'=f(e_{i+1})=[f(e_i), f(e_0)], & \quad 1\leq i\leq n-1,  \\[1mm]
\end{array} \right.$$
where $A_0\,B_1\neq0$ or more precisely as follows:
\begin{eqnarray}\label{bc.dim07}
e_{0}^{\prime }&=&A_{0}e_{0}+A_{1}e_{1},\nonumber \\[1mm]
e_{1}^{\prime
}&=&B_{1}\,e_{1}+B_{2}e_{2}+B_{3}e_{3},\nonumber\\[1mm]
e_{2}^{\prime
}&=&A_{0}B_1\,e_{2}+A_{0}B_2e_{3}+\left(A_{{0}}B_{{3}}-A_{{1}}B_{{2}}c_{{1,2}}\right)e_4-A_{{1}}
\left( B_{{2}}c_{{1,3}}+B_{{3}}c_{{1,2}} \right) e_5 \nonumber\\[1mm]
&&+A_{{1}} \left( B_{{1}}c_{{1,1}}-B_{{2}}c_{{1,4}}-B_{{3}}c_{{
1,3}} \right)
e_6,\nonumber \\[1mm]
e_{3}^{\prime
}&=&A_{0}^2B_1\,e_{3}+\left({A^{2}_{{0}}}B_{{2}}-A_{{0}}{A_{{1}}}B_1c_{{1,2}}\right)e_4+({A^{2}_{{0}}}B_{{3}}-2\,A_{{0}}A_{{1}}B
_{{2}}c_{{1,2}}-A_{{0}}A_{{1}}B_{{1}}c_{{1,3}} )e_5 \nonumber
\\&&+(-A_{{1}} \left( -A_{{1}}B_{{2}}{c^{2}_{{1,2}}}+A_{{1}}B_{{2}}c_{{1,2}}
c_{{2,3}}+A_{{0}}B_{{1}}c_{{1,4}}+2\,A_{{0}}B_{{2}}c_{{1,3}}+2\,A_{{0}
}B_{{3}}c_{{1,2}}-A_{{0}}B_{{3}}c_{{2,3}} \right)
)e_6,\\[1mm]
e_{4}^{\prime
}&=&A_{0}^3B_1\,e_{4}+({A^{3}_{{0}}}B_{{2}}-2\,{A^{2}_{{0}}}A_{{1}}B_{{1}}c_{{1,2}})e_5+\nonumber
\\&&({A^{3}_
{{0}}}B_{{3}}-2\,A_{{1}}{A^{2}_{{0}}}B_{{1}}c_{{1,3}}+A_{{0}}{A^{2}_{{1}}}B_{{1}}{b^{2}
_{{1,2}}}-A_{{0}}{A_{{1}}}^{2}B_{{1}}c_{{1,2}}c_{{2,3}}-3\,A_{{1}}
{A_{{0}}}^{2}B_{{2}}c_{{1,2}}+{A^{2}_{{0}}}A_{{1}}B_{{2}}c_{{2,3}}
)e_6,\nonumber\\[1mm]
e_{5}^{\prime
}&=&A^4_0\,B_1\,e_5+({A^{4}_{{0}}}B_{{2}}-3\,A_{{1}}{A^{3}_{{0}}}B_{{1}}c_{{1,2}}+A_{{1}}{A^{3}_{{0}}}B_{{1}}b_
{{2,3}}
)e_6,\nonumber\\[1mm]
e_{6}^{\prime }&=& A^5_0\,B_1\,e_6.\nonumber
\end{eqnarray}

Consider
\begin{eqnarray*}
[e_{1}^{\prime }, e_{2}^{\prime }]&=&{A_{{0}}\,B_{{1}}}^{2} \left(
c_{{1,2}}e_{{4}}+c_{{1,3}}e_{{5}}+c_{{1,4} }e_{{6}} \right)
+A_{{0}}B_{{1}}B_{{2}} \left( c_{{1,2}}e_{{5}}+c_{{1, 3}}e_{{6}}
\right) +B_{{1}} ( A_{{2}}B_{{1}}c_{{1,2}}-A_{{1}}B_{{
2}}c_{{1,2}}+\\[1mm]
&& A_{{0}}B_{{3}} ) \left( c_{{1,2}}-c_{{2,3}}
 \right) e_{{6}}+A_{{0}}\,{B_{{2}}}^{2}c_{{2,3}}e_{{6}}-A_{{0}}B_
{{1}}B_{{3}}c_{{2,3}}e_{{6}}=c^\prime_{1,2}e_{4}^{\prime
}+c^\prime_{1,3}e_{5}^{\prime }+c_{1,4}^\prime e_{6}^{\prime }.
\end{eqnarray*}

Now bearing in mind (8) and equating the coefficients
of $e_4, e_5$ and $e_6$ we get the equalities (4), (5) and (6), respectively. The last
equality follows from $$\ds [f(e_2),
f(e_3)]=c^\prime_{2,3} f(e_6)\ \Longrightarrow  \
{A_{{0}}}^{3}{B_{{1}}}^{2}c_{{2,3}}e_{{6}}=c'_{{2,3}}{A_{{0}}
}^{5}B_{{1}}e_{{6}} \Longrightarrow  \ c_{2,3}^{\prime }={\frac
{B_1\,c_{2,3}}{{A^2_0}}}. $$

``Only if'' part. Let the
equalities (1)--(7) hold. Then the above base change is adapted and
it transforms $L(C)$ to $L(C').$

Indeed,
\begin{eqnarray*}
[e'_0,e'_0]&=&\left[A_{0}e_{0}+A_{1}e_{1},\ A_{0}e_{0}+A_{1}e_{1}\right]\\[1mm]
&=&{A^{2}_{{0}}}[e_0,e_0]+A_{0}A_{1}[e_0,e_1]+A_{0}A_{1}[e_1,e_0]+A^2_{1}[e_1,e_1]\\[1mm]
&=& \left( {A^{2}_{{0}}}c_{0,0}+A_{{0}}A_{{1}}c_{{0,1}}+{A^{2}_{{
1}}}c_{{1,1}} \right)
e_{{6}}=c'_{0,0}{A^{5}_{{0}}}B_{{1}}e_6=c'_{0,0}e'_6.\\[1mm]
&&\\[1mm]
[e'_0,e'_1]&=&\left[A_{0}e_{0}+A_{1}e_{1},\ B_{1}e_{1}+B_{2}e_{2}+B_3e_3\right]\\[1mm]
&=&A_{{0}}B_{{1}} \left( -e_{{2}}+c_{{0,1}}e_{{6}} \right)
-A_{{0}}B_{{2}
}e_{{3}}-A_{{0}}B_{{3}}e_{{4}}+A_{{1}}B_{{1}}c_{{1,1}}e_{{6}}+A_{{1}}B
_{{2}} ( c_{{1,2}}e_{{4}}+c_{{1,3}}e_{{5}}+c_{{1,4}}e_{{6}}
 )+\\[1mm]
 && A_{{1}}B_{{3}} ( c_{{1,2}}e_{{5}}+c_{{1,3}}e_{{6}}
 )
\\[1mm]
 &=&-(A_{0}B_1\,e_{2}+A_{0}B_2e_{3}+\left(A_{{0}}B_{{3}}-A_{{1}}B_{{2}}c_{{1,2}}\right)e_4-A_{{1}}
\left( B_{{2}}c_{{1,3}}+B_{{3}}c_{{1,2}} \right) e_5+\\[1mm]
&& A_{{1}} \left( B_{{1}}c_{{1,1}}-B_{{2}}c_{{1,4}}-B_{{3}}c_{{
1,3}} \right)
e_6)+B_{{1}} \left( A_{{0}}c_{{0,1}}+2\,A_{{1}}c_{{1,1}} \right)e_{{6}}\\[1mm]
 &=& -e'_2+A^5_0 B_1c'_{0,1}e_5= -e'_2+c'_{0,1}e'_5.
\end{eqnarray*}

\begin{equation*}
[e'_1,e'_1]=\left[B_{1}e_{1}+B_{2}e_{2}+B_3e_3,B_{1}e_{1}+B_{2}e_{2}+B_3e_3\right]=\left[B_{{1}}e_1,B_{{1}}e_1\right]=B^2_1c_{1,1}e_6=A^5_0B_1c'_{1,1}e_6=c'_{1,1}e'_6.\\[1mm]
\end{equation*}

The brackets $[e_{1}',e_{2}'], [e_{1}',e_{3}'], [e_{1}',e_{4}']$ and $[e_{2}',e_{3}']$ can be gotten similarly.
\end{proof}
\end{theor}

The next section deals with the classification problem of ${TLeib}_7$.

\subsubsection{ Isomorphism classes in $TLeib_{7}$}
In this subsection we give a list of
all algebras from $TLeib_{7}.$

\nt Represent $TLeib_{7}$ as a union of the following subsets: \
\

$U_{7}^{1}=\{L(C)\in TLeib_{7}\ :c_{2,3}\neq 0,\ c_{1,1} \neq
0\};$

$U_{7}^{2}=\{L(C)\in TLeib_{7}\ : c_{2,3}\neq 0,\ c_{1,1}=0,\
c_{0,1}\neq0\};$

$U_{7}^{3}=\{L(C)\in TLeib_{7}\ : c_{2,3}\neq 0,\
c_{1,1}=c_{0,1}=0,\ c_{1,2}\neq0,\ c_{0,0}\neq0\};$

$U_{7}^{4}=\{L(C)\in TLeib_{7}\ : c_{2,3}\neq 0,\
c_{1,1}=c_{0,1}=0,\ c_{1,2}\neq0,\ c_{0,0}=0\};$

$U_{7}^{5}=\{L(C)\in TLeib_{7}\ : c_{2,3}\neq 0,\
c_{1,1}=c_{0,1}=c_{1,2}=0,\ c_{1,3}\neq0\};$

$U_{7}^{6}=\{L(C)\in TLeib_{7}\ : c_{2,3}\neq 0,\
c_{1,1}=c_{0,1}=c_{1,2}=c_{1,3}=0,\ c_{0,0}\neq0\};$

$U_{7}^{7}=\{L(C)\in TLeib_{7}\ : c_{2,3}\neq 0,\
c_{1,1}=c_{0,1}=c_{1,2}=c_{1,3}=c_{0,0}=0\};$

$U_{7}^{8}=\{L(C)\in TLeib_{7}\ :c_{2,3}=0,\ c_{1,2}\neq0,\
c_{1,1}\neq0\};$

$U_{7}^{9}=\{L(C)\in TLeib_{7}\ :c_{2,3}=0,\ c_{1,2}\neq0,\
c_{1,1}=0,\ c_{0,1}\neq0\};$

$U_{7}^{10}=\{L(C)\in TLeib_{7}\ :c_{2,3}=0,\ c_{1,2}\neq0,\
c_{1,1}=c_{0,1}=0,\ c_{0,0}\neq0\};$

$U_{7}^{11}=\{L(C)\in TLeib_{7}\ :c_{2,3}=0,\ c_{1,2}\neq0,\
c_{1,1}=c_{0,1}=c_{0,0}=0,\ 4\,c_{{1,4}}c_{{1,2}}-5\,c_{{1,3}}^{2}
\neq0\};$

$U_{7}^{12}=\{L(C)\in TLeib_{7}\ :c_{2,3}=0,\ c_{1,2}\neq0,\
c_{1,1}=c_{0,1}=c_{0,0}=4\,c_{{1,4}}c_{{1,2}}-5\,c_{{1,3}}^{2}
=0\};$

$U_{7}^{13}=\{L(C)\in TLeib_{7}\ :c_{2,3}=c_{1,2}=0,\
c_{1,1}\neq0,\ c_{1,4}\neq0\};$

$U_{7}^{14}=\{L(C)\in TLeib_{7}\ :c_{2,3}=c_{1,2}=0,\
c_{1,1}\neq0,\ c_{1,4}=0,\ c_{1,3}\neq0\};$

$U_{7}^{15}=\{L(C)\in TLeib_{7}\ :c_{2,3}=c_{1,2}=0,\
c_{1,1}\neq0,\ c_{1,4}=c_{1,3}=0,\ 4\,\
c_{{0,0}}c_{{1,1}}-c_{{0,1}}^{2}\neq0\};$

$U_{7}^{16}=\{L(C)\in TLeib_{7}\ :c_{2,3}=c_{1,2}=0,\
c_{1,1}\neq0,\
c_{1,4}=c_{1,3}=4\,c_{{0,0}}c_{{1,1}}-c_{{0,1}}^{2}=0\};$

$U_{7}^{17}=\{L(C)\in TLeib_{7}\ :c_{2,3}=c_{1,2}=c_{1,1}=0,\
c_{0,1}\neq0,\ c_{1,4}\neq0\};$

$U_{7}^{18}=\{L(C)\in TLeib_{7}\ :c_{2,3}=c_{1,2}=c_{1,1}=0,\
c_{0,1}\neq0,\ c_{1,4}=0,\ c_{1,3}\neq0\};$

$U_{7}^{19}=\{L(C)\in TLeib_{7}\ :c_{2,3}=c_{1,2}=c_{1,1}=0,\
c_{0,1}\neq0,\ c_{1,4}=c_{1,3}=0\};$

$U_{7}^{20}=\{L(C)\in TLeib_{7}\
:c_{2,3}=c_{1,2}=c_{1,1}=c_{0,1}=0,\ c_{0,0}\neq0,\
c_{1,4}\neq0\};$

$U_{7}^{21}=\{L(C)\in TLeib_{7}\
:c_{2,3}=c_{1,2}=c_{1,1}=c_{0,1}=0,\ c_{0,0}\neq0,\ c_{1,4}=0,\
c_{1,3}\neq0\};$

$U_{7}^{22}=\{L(C)\in TLeib_{7}\
:c_{2,3}=c_{1,2}=c_{1,1}=c_{0,1}=0,\ c_{0,0}\neq0,\
c_{1,4}=c_{1,3}=0\};$

$U_{7}^{23}=\{L(C)\in TLeib_{7}\
:c_{2,3}=c_{1,2}=c_{1,1}=c_{0,1}=c_{0,0}=0,\ c_{1,4}\neq0,\
c_{1,3}\neq0\};$

$U_{7}^{24}=\{L(C)\in TLeib_{7}\
:c_{2,3}=c_{1,2}=c_{1,1}=c_{0,1}=c_{0,0}=0,\ c_{1,4}\neq0,\
c_{1,3}=0\};$

$U_{7}^{25}=\{L(C)\in TLeib_{7}\
:c_{2,3}=c_{1,2}=c_{1,1}=c_{0,1}=c_{0,0}=c_{1,4}=0,\
c_{1,3}\neq0\};$

$U_{7}^{26}=\{L(C)\in TLeib_{7}\
:c_{2,3}=c_{1,2}=c_{1,1}=c_{0,1}=c_{0,0}=c_{1,4}=c_{1,3}=0\}.$

\begin{pr}\label{p2}\emph{}
\begin{enumerate}
\item Two algebras $L(C)$ and $L(C')$ from $U_{7}^{1}$ are
isomorphic, if and only if \\ $$ \ds \left(\frac
{c^\prime_{2,3}}{c'_{{1,1}}}\right)^8\,(4\,c'_{{0,0}}c'_{{1,1}}-{c'}_{{0,1}}^{2})=\left(\frac
{c_{2,3}}{c_{{1,1}}}\right)^8\,(4\,c_{{0,0}}c_{{1,1}}-c_{{0,1}}^{2})
 ,$$\ $$ {\frac {c^\prime_{{1,2}}}{c^\prime_{2,3}}}={\frac {c_{{1,2}}}{c_{2,3}}} \ \ \mbox{and}\ \
 {\frac { \left( c'_{{1,3}}c'_{{1,1}}-c'_{{0,1}}{{c'}^{2}_{{1,2}}} \right)
^{3}}{{{c'}^{2}_{{2,3}}}{{c'}^{4}_{{1,1}}}}} ={\frac { \left(
c_{{1,3}}c_{{1,1}}-c_{{0,1}}{c^{2}_{{1,2}}} \right)
^{3}}{{c^{2}_{{2,3}}}{c^{4}_{{1,1}}}}}. $$
 \item For any $\lambda_1,  \lambda_2,\lambda_3 \in \mathbb{C}, $  there exists $L(C)\in
U_{7}^{1}:$ \\ $$\ds \left(\frac
{c_{2,3}}{c_{{1,1}}}\right)^8\,(4\,c_{{0,0}}c_{{1,1}}-c_{{0,1}}^{2})=\lambda_1,
\ \ {\frac {c_{{1,2}}}{c_{2,3}}}=\lambda_2, \ \ {\frac { \left(
c_{{1,3}}c_{{1,1}}-c_{{0,1}}{c^{2}_{{1,2}}} \right)
^{3}}{{c^{2}_{{2,3}}}{c^{4}_{{1,1}}}}}=\lambda_3.$$

Then orbits in $U_{7}^{1}$ can be parameterized as
$L\left(\lambda_1, 0, 1, \lambda_2, \lambda_3, 0, 1\right), \
\lambda_1, \lambda_2,\lambda_3 \in \mathbb{C} .$
\end{enumerate}

\begin{proof}
\begin{enumerate}

\item\ ``If'' part due to Theorem  \ref{T2} if one substitutes the
expressions for $c_{0,0}',$ $c_{0,1}',$ $c_{1,1}',$ $c_{1,2}',$
$c_{1,3}',$ $c_{1,4}',$ $c_{2,3}':$
$$ \ds \left(\frac
{c^\prime_{2,3}}{c'_{{1,1}}}\right)^8\,(4\,c'_{{0,0}}c'_{{1,1}}-{c'}_{{0,1}}^{2})=\left(\frac
{c_{2,3}}{c_{{1,1}}}\right)^8\,(4\,c_{{0,0}}c_{{1,1}}-c_{{0,1}}^{2}),$$ $$
 \ \ {\frac {c^\prime_{{1,2}}}{c^\prime_{2,3}}}={\frac {c_{{1,2}}}{c_{2,3}}}, \ \
 {\frac { \left( c'_{{1,3}}c'_{{1,1}}-c'_{{0,1}}{{c'}^{2}_{{1,2}}} \right)
^{3}}{{{c'}^{2}_{{2,3}}}{{c'}^{4}_{{1,1}}}}} ={\frac { \left(
c_{{1,3}}c_{{1,1}}-c_{{0,1}}{c^{2}_{{1,2}}} \right)
^{3}}{{c^{2}_{{2,3}}}{c^{4}_{{1,1}}}}}. $$

`` Only if '' part.\ \ Let the equalities  $$ \ds \left(\frac
{c^\prime_{2,3}}{c'_{{1,1}}}\right)^8\,(4\,c'_{{0,0}}c'_{{1,1}}-{c'}_{{0,1}}^{2})=\left(\frac
{c_{2,3}}{c_{{1,1}}}\right)^8\,(4\,c_{{0,0}}c_{{1,1}}-c_{{0,1}}^{2}),$$
 $${\frac {c^\prime_{{1,2}}}{c^\prime_{2,3}}}={\frac
 {c_{{1,2}}}{c_{2,3}}},\
 {\frac { \left( c'_{{1,3}}c'_{{1,1}}-c'_{{0,1}}{{c'}^{2}_{{1,2}}} \right)
^{3}}{{{c'}^{2}_{{2,3}}}{{c'}^{4}_{{1,1}}}}} ={\frac { \left(
c_{{1,3}}c_{{1,1}}-c_{{0,1}}{c^{2}_{{1,2}}} \right)
^{3}}{{c^{2}_{{2,3}}}{c^{4}_{{1,1}}}}} $$ hold.

Consider the base change (\ref{bc.dim07}) in Theorem \ref{T2}
with $$\ds A_1={\frac {-A_{{0}}c_{{0,1}}}{2\,c_{{1,1}}}}, \ \ds
B_1={\frac {{A_{{0}}}^{2}}{c_{{2,3}}}},$$ and $$\ds
B_3=\frac{1}{8\,{A_{{0}}}^{2}{c_{{1,1}}}^{2}{c_{{2,3}}}^{ 2}}
(4\,{c_{{1,1}}}^{2} \left(
{A_{{0}}}^{4}c_{{1,4}}+4\,{B_{{2}}}^{2}{c_{{ 2,3}}}^{3} \right)
+{A_{{0}}}^{4} \left( 5\,c_{{1,2}}-c_{{2,3}}
 \right)  \left( {c_{{0,1}}}^{2}{c_{{1,2}}}^{2}-2\,c_{{0,1}}c_{{1,1}}c
_{{1,3}} \right) ).$$
\\ This changing leads $L(C)$ into
$$\ds L\left(\left(\frac
{c_{2,3}}{c_{{1,1}}}\right)^8\,(4\,c_{{0,0}}c_{{1,1}}-c_{{0,1}}^{2}),0,1,{\frac
{c_{{1,2}}}{c_{2,3}}},{\frac { \left(
c_{{1,3}}c_{{1,1}}-c_{{0,1}}{c^{2}_{{1,2}}} \right)
^{3}}{{c^{2}_{{2,3}}}{c^{4}_{{1,1}}}}},0,1\right).$$ An analogous
base change with ``dash'' transforms $L(C')$ into $$\ds
L\left(\left(\frac
{c'_{2,3}}{c'_{{1,1}}}\right)^8\,(4\,{c'}_{{0,0}}{c'}_{{1,1}}-{c'}_{{0,1}}^{2}),0,1,{\frac
{{c'}_{{1,2}}}{{c'}_{2,3}}},{\frac { \left(
{c'}_{{1,3}}{c'}_{{1,1}}-{c'}_{{0,1}}{{c'}^{2}_{{1,2}}} \right)
^{3}}{{{c'}^{2}_{{2,3}}}{{c'}^{4}_{{1,1}}}}},0,1\right).$$

Since  $$ \ds \left(\frac
{c^\prime_{2,3}}{c'_{{1,1}}}\right)^8\,(4\,c'_{{0,0}}c'_{{1,1}}-{c'}_{{0,1}}^{2})=\left(\frac
{c_{2,3}}{c_{{1,1}}}\right)^8\,(4\,c_{{0,0}}c_{{1,1}}-c_{{0,1}}^{2})
 ,\ \ {\frac {c^\prime_{{1,2}}}{c^\prime_{2,3}}}={\frac {c_{{1,2}}}{c_{2,3}}}$$ \ \ \mbox{and}\ \
$$ {\frac { \left( c'_{{1,3}}c'_{{1,1}}-c'_{{0,1}}{{c'}^{2}_{{1,2}}} \right)
^{3}}{{{c'}^{2}_{{2,3}}}{{c'}^{4}_{{1,1}}}}} ={\frac { \left(
c_{{1,3}}c_{{1,1}}-c_{{0,1}}{c^{2}_{{1,2}}} \right)
^{3}}{{c^{2}_{{2,3}}}{c^{4}_{{1,1}}}}},$$ the algebras are isomorphic.

 \item\qquad Obvious.
\end{enumerate}

\end{proof}

\end{pr}

\begin{pr}\emph{}
\begin{enumerate}
\item Two algebras $L(C)$ and $L(C')$ from $U_{7}^{2}$ are
isomorphic, if and only if  $$ \ds
 {\frac {c'_{{1,2}}}{c'_{{2,3}}}}
 ={\frac {c_{{1,2}}}{c_{{2,3}}}},\ \ {\frac { \left( c'_{{1,3}}c'_{{0,1}}-2\,c'_{{0,0}}{{c'}^{2}_{{1,2}}}
 \right) ^{4}}{{{c'}^{4}_{{2,3}}}{{c'}^{5}_{{0,1}}}}}
={\frac { \left( c_{{1,3}}c_{{0,1}}-2\,c_{{0,0}}{c^{2}_{{1,2}}}
 \right) ^{4}}{{c^{4}_{{2,3}}}{c^{5}_{{0,1}}}}}.$$
 \item For any $ \lambda_1, \lambda_2\in \mathbb{C},$ there exists $L(C)\in U_{7}^{2}:$ $$\ds{\frac
{c_{{1,2}}}{c_{{2,3}}}}=\lambda_1,\ {\frac { \left(
c_{{1,3}}c_{{0,1}}-2\,c_{{0,0}}{c^{2}_{{1,2}}}
 \right) ^{4}}{{c^{4}_{{2,3}}}{c^{5}_{{0,1}}}}}=\lambda_2.$$
\end{enumerate}
The orbits in  $U_{7}^{2}$ can be parameterized as $L\left(0, 1,
0,\lambda_1, \lambda_2, 0,1\right), \ \lambda_1,\lambda_2 \in
\mathbb{C}$.

\begin{proof} The proof is similar that of Proposition \ref{p2}, where we put $$\ds A_1=-{\frac
{A_{{0}}c_{{0,0}}}{c_{{0,1}}}},\ B_1={\frac
{{A_{{0}}}^{2}}{c_{{2,3}}}},$$ and \\
$$ \ds
B_3=\frac{({c_{{0,1}}}^{2} \left(
{A_{{0}}}^{4}c_{{1,4}}+{B_{{2}}}^{2}{c_{{2,3}}} ^{3}
\right)+{A_{{0}}}^{4}(c_{{0,0}} \left(
c_{{0,0}}{c_{{1,2}}}^{2}-c_{{1,3}}c_{{0,1}} \right)
 \left( 5\,c_{{1,2}}-c_{{2,3}} \right)
) )}{2\,{A_{{0}}}^{2}{c_{{0,1}}}^{2}{c_{{2,3}}}^{2}}.$$
\end{proof}

\end{pr}

Since the proving of the next coming Propositions 2.3 -- 2.13 are similar those of Propositions 2.1 and 2.2 we decided to omit the details of them. ``If'' parts of them follow from Theorem 1.2, for ``Only if'' part we just give the respective values of the coefficients $A_0, A_1, B_1, B_2$ and $B_3$ in the base change (\ref{bc.dim07}). Note that if no value is given then it is considered as an arbitrary.
\begin{pr}\emph{}
\begin{enumerate}
\item Two algebras $L(C)$ and $L(C')$ from $U_{7}^{3}$ are
isomorphic, if and only if \ $ \ds {\frac
{c'_{{1,2}}}{c'_{{2,3}}}}={\frac {c_{{1,2}}}{c_{{2,3}}}}.$ \item
For any $\lambda\in \mathbb{C^*},$ there exists $L(C)\
\mbox{from}\ U_{7}^{3}:\ \ds{\frac {c_{{1,2}}}{c_{{2,3}}}}=\lambda
.$
\end{enumerate}

Then orbits in  $U_{7}^{3}$ can be parameterized as $L\left(1, 0,
0, \lambda, 0, 0, 1\right),\ \lambda\in \mathbb{C^*}$.

\begin{proof} The respective values of coefficients are: $$\ds A_1={\frac {-A_{{0}}c_{{1,3}}}{2\,{c_{{1,2}}}^{2}}},
\ B_1={\frac {{A_{{0}}}^{2}}{c_{{2,3}}}}$$ and $$ \ds B_3={\frac
{{A_{{0}}}^{4}\,\left(
4\,c_{{1,4}}{c_{{1,2}}}^{2}-5\,{c_{{1,3}}}^{2}c_{{1,2}}+{c_{{1
,3}}}^{2}c_{{2,3}} \right) +4\,{B_{{2}}}^{2}{c_{{2,3}}}^{
3}{c_{{1,2}}}^{2} }{8\,{A_{{0}}}^{2}{c_{{1,2}}}^{2}{
c_{{2,3}}}^{2}}} .$$
\end{proof}

\end{pr}

\begin{pr}\emph{}
\begin{enumerate}
\item Two algebras $L(C)$ and $L(C')$ from $ U_{7}^{4}$ are
isomorphic, if and only if \ $ \ds {\frac
{c'_{{1,2}}}{c'_{{2,3}}}}={\frac {c_{{1,2}}}{c_{{2,3}}}}.$ \item
For any $\lambda\in \mathbb{C^*},$ there exists $L(C)\
\mbox{from}\  U_{7}^{4}:\ \ds{\frac
{c_{{1,2}}}{c_{{2,3}}}}=\lambda .$
\end{enumerate}

Then orbits in  $U_{7}^{4}$ can be parameterized as $L\left(0, 0,
0, \lambda, 0, 0, 1\right),\ \lambda\in \mathbb{C^*}$.

\begin{proof} Here $$\ds A_1={\frac {-A_{{0}}c_{{1,3}}}{2\,{c_{{1,2}}}^{2}}},
\ B_1={\frac {{A_{{0}}}^{2}}{c_{{2,3}}}}\ \ \mbox{and} \ \ \ds
B_3={\frac {{A_{{0}}}^{4}\,\left(
4\,c_{{1,4}}{c_{{1,2}}}^{2}-5\,{c_{{1,3}}}^{2}c_{{1,2}}+{c_{{1
,3}}}^{2}c_{{2,3}} \right) +4\,{B_{{2}}}^{2}{c_{{2,3}}}^{
3}{c_{{1,2}}}^{2} }{8\,{A_{{0}}}^{2}{c_{{1,2}}}^{2}{
c_{{2,3}}}^{2}}} .$$
\end{proof}

\end{pr}

\begin{pr}\emph{}
\begin{enumerate}
\item Two algebras $L(C)$ and $L(C')$ from $U_{7}^{5}$ are
isomorphic, if and only if \ $$ \ds {\frac
{{{c'}_{{2,3}}}^{6}{c'}_{{0,0}}}{{{c'}_{{1,3}}}^{5}}}={\frac
{{c_{{2,3}}}^{6}c_{{0,0}}}{{c_{{1,3}}}^{5}}}.
 $$
 \item For any $\lambda\in \mathbb{C},$ there exists $L(C)\in U_{7}^{5}:\ \ds{\frac
{{c_{{2,3}}}^{6}c_{{0,0}}}{{c_{{1,3}}}^{5}}}=\lambda.$
\end{enumerate}

Then orbits in  $U_{7}^{5}$ can be parameterized as
$L\left(\lambda, 0, 0, 0, 1, 0, 1\right), \lambda\in \mathbb{C}$.

\begin{proof}For this case:
 $$\ds A_0={\frac {c_{{1,3}}}{c_{{2,3}}}},
\ B_1={\frac {{c_{{1,3}}}^{2}}{{c_{{2,3}}}^{3}}}\ \ \mbox{and}\ \
\ds B_3={\frac
{{c_{{1,3}}}^{4}c_{{1,4}}+{B_{{2}}}^{2}{c_{{2,3}}}^{7}-A_{{1}}{
c_{{1,3}}}^{4}{c_{{2,3}}}^{2}}{2\,{c_{{1,3}}}^{2}{c_{{2,3}}}^{4}}}
.$$
\end{proof}

\end{pr}

\begin{pr}\emph{}
\begin{enumerate}
\item Two algebras $L(C)$ and $L(C')$ from $U_{7}^{8}$ are
isomorphic, if and only if $$ \ds {\frac { \left(
4\,c'_{{0,0}}{{c'}_{{1,2}}}^{4}-2\,c'_{{1,3}}c'_{{0,1}}{{c'}_{{
1,2}}}^{2}+{{c'}_{{1,3}}}^{2}c'_{{1,1}} \right)
^{3}}{{{c'}_{{1,2}}}^{4}{{c'}_{{ 1,1}}}^{5}}}={\frac { \left(
4\,c_{{0,0}}{c_{{1,2}}}^{4}-2\,c_{{1,3}}c_{{0,1}}{c_{{
1,2}}}^{2}+{c_{{1,3}}}^{2}c_{{1,1}} \right)
^{3}}{{c_{{1,2}}}^{4}{c_{{ 1,1}}}^{5}}},$$
$${\frac { \left( c'_{{0,1}}{{c'}_{{1,2}}}^{2}-c_{{1,3}}c_{{1,1}} \right) ^{
3}}{{{c'}_{{1,2}}}^{2}{{c'}_{{1,1}}}^{4}}}={\frac { \left(
c_{{0,1}}{c_{{1,2}}}^{2}-c_{{1,3}}c_{{1,1}} \right) ^{
3}}{{c_{{1,2}}}^{2}{c_{{1,1}}}^{4}}}\  \ \mbox{and} \ \ {\frac {
\left( 4\,{c'}_{{1,4}}{c'}_{{1,2}}-5\,{{c'}_{{1,3}}}^{2} \right)
^{3}} {{{c'}_{{1,2}}}^{4}{{c'}_{{1,1}}}^{2}}} ={\frac { \left(
4\,c_{{1,4}}c_{{1,2}}-5\,{c_{{1,3}}}^{2} \right) ^{3}}
{{c_{{1,2}}}^{4}{c_{{1,1}}}^{2}}}. $$
 \item For any $\lambda_1, \lambda_2, \lambda_3\in
\mathbb{C},$ there exists $L(C)\in U_{7}^{8}:\ \ \ $

 \nt$${\frac {
\left( 4\,c_{{0,0}}{c_{{1,2}}}^{4}-2\,c_{{1,3}}c_{{0,1}}{c_{{
1,2}}}^{2}+{c_{{1,3}}}^{2}c_{{1,1}} \right)
^{3}}{{c_{{1,2}}}^{4}{c_{{ 1,1}}}^{5}}}=\lambda_1 , \ {\frac {
\left( c_{{0,1}}{c_{{1,2}}}^{2}-c_{{1,3}}c_{{1,1}} \right) ^{
3}}{{c_{{1,2}}}^{2}{c_{{1,1}}}^{4}}}=\lambda_2,\ {\frac { \left(
4\,c_{{1,4}}c_{{1,2}}-5\,{c_{{1,3}}}^{2} \right) ^{3}}
{{c_{{1,2}}}^{4}{c_{{1,1}}}^{2}}}=\lambda_3.
$$
\end{enumerate}

Then orbits in  $U_{7}^{8}$ can be parameterized as
$L\left(\lambda_1,\lambda_2,1,1,0,\lambda_3,0\right), \
\lambda_1,\lambda_2,\lambda_3 \in \mathbb{C}.$

\begin{proof} We put here,$$\ds A_1={\frac {-A_{{0}}c_{{1,3}}}{2\,{c_{{1,2}}}^{2}}},
\ B_1={\frac {{A_{{0}}}^{2}}{c_{{1,2}}}}.$$
\end{proof}

\end{pr}

\begin{pr}\emph{}
\begin{enumerate}
\item Two algebras $L(C)$ and $L(C')$ from $U_{7}^{9}$ are
isomorphic, if and only if
$${\frac { \left( 2\,c'_{{0,0}}{{c'}_{{1,2}}}^{2}-c'_{{1,3}}c'_{{0,1}}
 \right) ^{4}}{{{c'}_{{1,2}}}^{4}{{c'}_{{0,1}}}^{5}}}={\frac { \left( 2\,c_{{0,0}}{b_{{1,2}}}^{2}-c_{{1,3}}c_{{0,1}}
 \right) ^{4}}{{c_{{1,2}}}^{4}{c_{{0,1}}}^{5}}}\ \ \mbox{and}\ \ {\frac { \left( 4\,c'_{{1,4}}c'_{{1,2}}-5\,{{c'}_{{1,3}}}^{2} \right) ^{2}}
{{{c'}_{{1,2}}}^{4}c'_{{0,1}}}}={\frac { \left(
4\,c_{{1,4}}c_{{1,2}}-5\,{c_{{1,3}}}^{2} \right) ^{2}}
{{c_{{1,2}}}^{4}c_{{0,1}}}}.
$$
 \item For any $\lambda_1,\lambda_2 \in \mathbb{C},$ there exists $L(C)\in U_{7}^{9}:$\ $$\ds {\frac { \left( 2\,c_{{0,0}}{b_{{1,2}}}^{2}-c_{{1,3}}c_{{0,1}}
 \right) ^{4}}{{c_{{1,2}}}^{4}{c_{{0,1}}}^{5}}}=\lambda_1 \ \ \mbox{and} \ \ {\frac { \left(
4\,c_{{1,4}}c_{{1,2}}-5\,{c_{{1,3}}}^{2} \right) ^{2}}
{{c_{{1,2}}}^{4}c_{{0,1}}}}=\lambda_2.$$
\end{enumerate}

Then orbits in  $U_{7}^{9}$ can be parameterized as
$L\left(\lambda_1,1,0,1,0,\lambda_2,0\right),\ \lambda_1,\lambda_2
\in \mathbb{C}.$

\begin{proof}For this case we put $$\ds A_1={\frac
{-A_{{0}}c_{{1,3}}}{2\,{c_{{1,2}}}^{2}}},\ {\rm{and}} \ B_1={\frac
{{A_{{0}}}^{2}}{c_{{1,2}}}}.$$
\end{proof}

\end{pr}

\begin{pr}\emph{}
\begin{enumerate}
\item Two algebras $L(C)$ and $L(C')$ from $U_{7}^{10}$ are
isomorphic, if and only if
$${\frac { \left( 4\,c'_{{1,4}}c'_{{1,2}}-5\,{{c'}_{{1,3}}}^{2} \right) ^{5}}
{{{c'}_{{1,2}}}^{12}{{c'}_{{0,0}}}^{2}}}={\frac { \left(
4\,c_{{1,4}}c_{{1,2}}-5\,{c_{{1,3}}}^{2} \right)^{5}}
{{c_{{1,2}}}^{12}{c_{{0,0}}}^{2}}}.
 $$
 \item For any $\lambda\in
\mathbb{C},$ there exists $L(C)\in U_{7}^{10}: \ \ds{\frac {
\left( 4\,c_{{1,4}}c_{{1,2}}-5\,{c_{{1,3}}}^{2} \right)^{5}}
{{c_{{1,2}}}^{12}{c_{{0,0}}}^{2}}}=\lambda.$
\end{enumerate}

Then orbits in  $U_{7}^{10}$ can be parameterized as $L\left(1,0,0
,1,0,\lambda,0\right),\ \lambda\in \mathbb{C}$.

\begin{proof}Here, $$\ds A_1={\frac {-A_{{0}}c_{{1,3}}}{2\,{c_{{1,2}}}^{2}}}\ \ \mbox{and}\
\ B_1={\frac {{A_{{0}}}^{2}}{c_{{1,2}}}}.$$
\end{proof}

\end{pr}

\begin{pr}\emph{}
\begin{enumerate}
\item Two algebras $L(C)$ and $L(C')$ from $U_{7}^{13}$ are
isomorphic, if and only if  $$\ds \left(\frac {{c'_{{1,4}}}
}{{c'_{{1,1}}}}\right)^8 \left(
4\,c'_{{0,0}}c'_{{1,1}}-{{c'}_{{0,1}}}^{2}
 \right)=\left(\frac
{{c_{{1,4}}} }{{c_{{1,1}}}}\right)^8 \left(
4\,c_{{0,0}}c_{{1,1}}-{c_{{0,1}}}^{2}
 \right),\ \ {\frac {c'_{{1,1}}c'_{{1,3}}}{{{c'}_{{1,4}}}^{2}}}={\frac {c_{{1,1}}c_{{1,3}}}{{c_{{1,4}}}^{2}}}.$$
 \item For any $\lambda_1, \lambda_2 \in \mathbb{C},$ there exists $L(C)\in U_{7}^{13}:\ \ \ $
$\ds \left(\frac {{c_{{1,4}}} }{{c_{{1,1}}}}\right)^8 \left(
4\,c_{{0,0}}c_{{1,1}}-{c_{{0,1}}}^{2}
 \right) =\lambda_1, \ \
 {\frac {c_{{1,1}}c_{{1,3}}}{{c_{{1,4}}}^{2}}}=\lambda_2.$
\end{enumerate}

Then orbits in  $U_{7}^{13}$ can be parameterized as
$L\left(\lambda_1,0,1,0,\lambda_2,1,0\right),  \
\lambda_1,\lambda_2 \in \mathbb{C}$.

\begin{proof}We put $$\ds A_0={\frac {c_{{1,1}}}{c_{{1,4}}}},\
 A_1={\frac {-c_{{0,1}}}{2\,c_{{1,4}}}},\ {\rm{ and}}\
 B_1={\frac {{c_{{1,1}}}^{4}}{{c_{{1,4}}}^{5}}}.$$
\end{proof}

\end{pr}

\begin{pr}\emph{}
\begin{enumerate}
\item Two algebras $L(C)$ and $L(C')$ from $U_{7}^{14}$ are
isomorphic, if and only if $$
\left({\frac{{c'}_{{1,3}}}{{c'}_{{1,1}}}}\right)^4 \left(
4\,c'_{{0,0}}c'_{{1,1}}-{{c'}_{{0,1}}}^{2}
 \right) =\left({\frac{{c}_{{1,3}}}{{c}_{{1,1}}}}\right)^4 \left( 4\,c_{{0,0}}c_{{1,1}}-{c_{{0,1}}}^{2}
 \right).
$$
 \item For any $\lambda\in\mathbb{C},$ there exists $L(C)\in U_{7}^{14}:\ \ \ $
$\ds \left({\frac{{c}_{{1,3}}}{{c}_{{1,1}}}}\right)^4 \left(
4\,c_{{0,0}}c_{{1,1}}-{c_{{0,1}}}^{2}
 \right)=\lambda.$
\end{enumerate}

Then orbits in  $U_{7}^{14}$ can be parameterized as
$L\left(\lambda,0,1,0,1,0,0\right), \ \lambda\in \mathbb{C}$.

\begin{proof}Here, we take $$\ds A_1={\frac
{-A_{{0}}c_{{0,1}}}{2\,c_{{1,1}}}},\ {\rm{ and}}\ B_1={\frac
{{A_{{0}}}^{5}}{c_{{1,1}}}}.$$
\end{proof}
\end{pr}

\begin{pr}\emph{}
\begin{enumerate}
\item Two algebras $L(C)$ and $L(C')$ from $U_{7}^{17}$ are
isomorphic, if and only if $ \ds \left(\frac
{{c'_{{1,3}}}}{{c'_{{1,4}}}}\right)^4 c'_{{0,1}}=\left(\frac
{{c_{{1,3}}}}{{c_{{1,4}}}}\right)^4 c_{{0,1}}.$
 \item For any $\lambda \in \mathbb{C},$ there exists $L(C)\in U_{7}^{17}:$
$ \ds \left(\frac {{c_{{1,3}}}}{{c_{{1,4}}}}\right)^4
c_{{0,1}}=\lambda.$
\end{enumerate}

Then orbits in  $U_{7}^{17}$ can be parameterized as
$L\left(0,1,0,0,\lambda,1,0\right), \ \lambda\in \mathbb{C}.$

\begin{proof}Take $$\ds A_1=-{\frac {A_{{0}}c_{{0,0}}}{c_{{0,1}}}},\ {\rm{ and}}\ B_1={\frac {{A_{{0}}}^{4}}{c_{{1,4}}}}.$$
\end{proof}

\end{pr}

\begin{pr}\emph{}
\begin{enumerate}
\item Two algebras $L(C)$ and $L(C')$ from $U_{7}^{20}$ are
isomorphic, if and only if $ \ds {\frac
{c'_{{0,0}}{{c'}_{{1,3}}}^{7}}{{{c'}_{{1,4}}}^{6}}}={\frac
{c_{{0,0}}{c_{{1,3}}}^{7}}{{c_{{1,4}}}^{6}}}.$
 \item For any $\lambda \in
\mathbb{C},$ there exists $ L(C)\in U_{7}^{20}:\ \ds {\frac
{c_{{0,0}}{c_{{1,3}}}^{7}}{{c_{{1,4}}}^{6}}}=\lambda$ .
\end{enumerate}

Then orbits in  $U_{7}^{20}$ can be parameterized as
$L\left(1,0,0,0,\lambda,1,0\right),\ \lambda\in \mathbb{C}$.

\begin{proof}Here, $\ds B_1={\frac {c_{{0,0}}}{{A_{{0}}}^{3}}}.$
\end{proof}
\end{pr}

\begin{pr}\emph{}

The subsets  $ U_{7}^{6},\ U_{7}^{7},\ U_{7}^{11},\ U_{7}^{12},\
U_{7}^{15},\ U_{7}^{16},\ U_{7}^{18},\ U_{7}^{19},\ U_{7}^{21},\
U_{7}^{22},\ U_{7}^{23},\ U_{7}^{24},\ U_{7}^{25}$ and $
U_{7}^{26}$\\ are single orbits with representatives \ $
L(1,0,0,0,0,0,0,1),\ L(0,0,0,0,0,0,0,1),\\
L(0,0,0,1,0,1,0),\ L(0,0,0,1,0,0,0),\ L(1,0,1,0,0,0,0),\
L(0,0,1,0,0,0,0),\ L(0,1,0,0,1,0,0),\\ L(0,1,0,0,0,0,0),\
L(1,0,0,0,1,0,0),\ L(1,0,0,0,0,0,0),\ L(0,0,0,0,1,1,0),\
L(0,0,0,0,0,1,0),\\ L(0,0,0,0,1,0,0)$ and $L(0,0,0,0,0,0,0),$
respectively.

\begin{proof}

To prove it, we give the respective values of $A_0,A_1,B_1,B_2$
and $B_3$ in the base change (\ref{bc.dim07}) leading to the appropriate representatives.

For $ U_{7}^{6}\ \rm{and}\ U_{7}^{7} :$  $$\ds B_{1}={\frac
{{A_{{0}}}^{2}}{c_{{2,3}}}}\ \mbox{and} \ B_3={\frac
{{A_{{0}}}^{4}c_{{1,4}}+{B_{{2}}}^{2}{c_{{2,3}}}^{3}}{2\,{A_{
{0}}}^{2}{c_{{2,3}}}^{2}}} .$$

For $U_{7}^{11} \ \rm{and}\ U_{7}^{12}:$  $$A_1=\ds {\frac
{-A_{{0}}c_{{1,3}}}{2\,{c_{{1,2}}}^{2}}}\ \ \mbox{and} \ \ B_1=\ds {\frac
{{A_{{0}}}^{2}}{c_{{1,2}}}}.$$

For $U_{7}^{15} \ \rm{and}\ U_{7}^{16}:$ $$A_1=\ds {\frac
{-A_{{0}}c_{{0,1}}}{2\,c_{{1,1}}}}\ \ \mbox{and}\ \ B_1=\ds {\frac
{{A_{{0}}}^{5}}{c_{{1,1}}}}.$$

For $U_{7}^{18}:$ $$A_1=\ds-{\frac
{A_{{0}}c_{{0,0}}}{c_{{0,1}}}}\ \ \mbox{and}\ \ B_1=\ds{\frac
{{A_{{0}}}^{3}}{c_{{1,3}}}}.$$

For $U_{7}^{19}:$ $$\ds A_1=-{\frac
{A_{{0}}c_{{0,0}}}{c_{{0,1}}}}.$$

For $U_{7}^{21} \ \rm{and}\ U_{7}^{22}:$  $$B_1=\ds {\frac
{c_{{0,0}}}{{A_{{0}}}^{3}}}.$$

For $U_{7}^{23} \ \rm{and}\ U_{7}^{24}:$  $$B_1=\ds {\frac
{{A_{{0}}}^{4}}{c_{{1,4}}}}.$$

\end{proof}
\end{pr}

Note that the orbits $U_7^6\  \rm{and} \ U_7^7$ can be included in the parametric family of orbits $U_7^3\
\rm{and}\  U_7^4,$ respectively at $\lambda =0.$

\subsection{Isomorphism criterion for $TLeib_8$}

Here, we have $n=7,$ means odd case, and any algebra $L$ from
${TLeib}_{8}$ can be represented as one dimensional central
extension ($C(L)=<e_7>$)  of 7-dimensional filiform Lie algebra
with adapted basis $\{e_0,e_1,...,e_6\}$ (see Theorem \ref{T1})
and on the adapted basis $\{e_0,e_1,...,e_7\}$ the class
${TLeib}_{8}$ can be represented as follows:

$TLeib_{8}=\left\{
\begin{array}{lll}
\lbrack e_{i},e_{0}]=e_{i+1}, & 1\leq i\leq {6}, &  \\[1mm]
\lbrack e_{0},e_{i}]=-e_{i+1}, & 2\leq i\leq {6}, &  \\[1mm]
\lbrack e_{0},e_{0}]=b_{0,0}e_{7}, &  &  \\[1mm]
\lbrack e_{0},e_{1}]=-e_{2}+b_{0,1}e_{7}, &  &  \\[1mm]
\lbrack e_{1},e_{1}]=b_{1,1}e_{7}, &  &  \\[1mm]
\lbrack e_{1},e_{2}]=-[e_2,e_1]=a_{1,4}\,e_4+a_{1,5}\,e_5+a_{1,6}\,e_6+b_{1,2}e_{7}, \\[1mm]
\lbrack e_{1},e_{3}]=-[e_{3},e_{1}]=a_{1,4}\,e_5+a_{1,5}\,e_6+b_{1,3}e_{7}, \\[1mm]
 [e_{1},e_{4}]=-[e_{4},e_{1}]=\left( a_{{1,4}}-a_{{2,6}} \right) e_{{6}}+b_{{1,4}}e_{{7}}, \\[1mm]
 [e_{1},e_{5}]=-[e_{5},e_{1}]=b_{{1,5}}e_{{7}}, \\[1mm]
 [e_{2},e_{3}]=-[e_{3},e_{2}]=a_{{2,6}}e_{{6}}+b_{{2,3}}e_{{7}},\\[1mm]
 [e_{1},e_{5}]=-[e_{5},e_{1}]=b_{1,5}e_{6}, \\[1mm]
 [e_{2},e_{4}]=-[e_{4},e_{2}]=b_{2,4}e_{7},\\[1mm]
[e_{i},e_{7-i}]=-[e_{7-i},e_{i}]=\left(-1\right)^{i}b_{3,4}e_{7},
\ \ \  1\leq i\leq n-1.
\end{array}%
\right. $

The next lemma specifies the set of structure constants of
algebras from  $TLeib_{8}.$


\begin{lm}\emph{}The structure constants of algebras from ${TLeib}_{8}$ satisfy the
following constraints:
 $$ \ds b_{1,3}=a_{1,6}, \ \ \  \ds b_{1,4}=a_{1,5}-b_{2,3},\ \ \  \ds b_{2,4}=a_{2,6},\ \ \ds b_{1,5}=a_{1,4}-2a_{2,6},
 \ \ and \ \ \ds
 b_{{1,6}} \left( a_{{2,6}}+2\,a_{{1,4}} \right)=0.$$

\begin{proof}
These relations come out from sequentially applications of the
Leibniz identity to the triples of the basis vectors
$\{e_0,e_1,e_2\}, \{e_0,e_1,e_3\},\{e_0,e_2,e_3\},\{e_0,e_1,e_4\}$
and
 $\{e_1,e_2,e_3\}$ respectively.
\end{proof}
\end{lm}

Further unifying the $TLeib_{8}$ table of multiplication we rewrite it via parameters

$c_{0,0},$ $c_{0,1},$ $c_{1,1},$ $c_{1,2},$ $c_{1,3},$ $c_{1,4},$ $c_{1,5},$ $c_{2,3},$ $c_{2,4}$ as follows:\\

$TLeib_{8}=\left\{
\begin{array}{lll}
\lbrack e_{i},e_{0}]=e_{i+1}, & 1\leq i\leq {6}, &  \\[1mm]
\lbrack e_{0},e_{i}]=-e_{i+1}, & 2\leq i\leq {6}, &  \\[1mm]
\lbrack e_{0},e_{0}]=c_{0,0}e_{7}, &  &  \\[1mm]
\lbrack e_{0},e_{1}]=-e_{2}+c_{0,1}e_{7}, &  &  \\[1mm]
\lbrack e_{1},e_{1}]=c _{1,1}e_{7}, &  &  \\[1mm]
\lbrack e_{1},e_{2}]=-[e_2,e_1]=c_{1,2}\,e_4+c_{1,3}\,e_5+c_{1,4}e_{6}+c_{1,5}e_{7}, \\[1mm]
\lbrack e_{1},e_{3}]=-[e_{3},e_{1}]=c_{1,2}\,e_5+c_{1,3}e_{6}+c_{1,4}e_{7}, \\[1mm]
[e_{1},e_{4}]=-[e_{4},e_{1}]=(c_{1,2}-c_{2,3})\,e_{6}+(c_{1,3}-c_{2,4})\,e_{7}, \\[1mm]
[e_{1},e_{5}]=-[e_{5},e_{1}]=(c_{1,2}-2c_{2,3})\,e_{7}, \\[1mm]
[e_{2},e_{3}]=-[e_{3},e_{2}]=c_{2,3}\,e_{6}+c_{2,4}\,e_{7},\\[1mm]
[e_{2},e_{4}]=-[e_{4},e_{2}]=c_{2,3}\,e_{7},\\[1mm]
[e_{i},e_{7-i}]=-[e_{7-i},e_{i}]=\left(-1\right)^{i}c_{3,4}e_{7}, \ \ \  1\leq i\leq n-1.\\[1mm]
\end{array}%
\right. $

\nt So an algebra from ${TLeib}_{8}$ with parameters
$c_{0,0},c_{0,1},c_{1,1},c_{1,2},c_{1,3},c_{1,4},c_{1,5},c_{2,3},c_{2,4},c_{3,4}$
  is denoted by $L(C),$ where $C=(c_{0,0},c_{0,1},c_{1,1},c_{1,2},c_{1,3},c_{1,4},c_{1,5},c_{2,3},c_{2,4},c_{3,4}).$

Note that  $L(C) \in TLeib_{8}$ if and only if \ $
 \ds
 c_{{3,4}} \left( c_{{2,3}}+2\,c_{{1,2}} \right)=0.$

Notice that according  to the Proposition \ref{p} and \ref{p1} the
adapted transformation in ${TLeib}_{8}$ can be taken in the form

$$ \left\{\begin{array}{lll} f(e_0)=A_0\,e_0 + A_1\,e_1 + A_2\,e_2 + A_3\,e_3+ A_4\,e_4, &
\\[1mm]
f(e_1)=B_1\,e_1+B_2\,e_2+B_3\,e_3+B_4\,e_4+B_5\,e_5, & \\[1mm]
f(e_{i+1})=[f(e_i), f(e_0)], & \quad 1\leq i\leq n-1,  \\[1mm]
\end{array} \right.$$
where $A_0\,B_1(A_0+A_1c_{3,4})\neq0.$

\nt The next theorem represents the action of the adapted base
change to the parameters
$c_{0,0},c_{0,1},c_{1,1},c_{1,2},c_{1,3},c_{1,4},c_{1,5},c_{2,3},c_{2,4},c_{3,4}$
of an algebra from ${TLeib}_{8}.$


\begin{theor}\emph{(Isomorphism criterion for $TLeib_{8}$)}
 Two filiform Leibniz algebras $L(C)$ and $L(C'),$ where
 $C=(c_{0,0},c_{0,1},c_{1,1},c_{1,2},c_{1,3},c_{1,4},c_{1,5},c_{2,3},c_{2,4},c_{3,4})$
 and
$C'=(c'_{0,0},c'_{0,1},c'_{1,1},c'_{1,2},c'_{1,3},c'_{1,4},c'_{1,5},c'_{2,3},c'_{2,4},
c'_{3,4}),$ from $TLeib_8$ are isomorphic, if and only if  there
exist $ A_0, A_1, A_2, A_3,A_4, B_1, B_2, B_3, B_4, B_5\in
\mathbb{C},$ such that $A_0B_1(A_0+A_1c_{3,4})\neq 0$ and the
following equalities hold:

\begin{eqnarray*}c_{0,0}^\prime&=&{\frac {{A_{{0}}}^{2}c_{{0,0}}+A_{{0}}A_{{1}}c_{{0,1}}+{A_{{1}}}^{2}c_
{{1,1}}}{c_{{1}}{A_{{0}}}^{5} \left( A_{{0}}+A_{{1}}c_{{3,4}}
\right) }},
\\
c_{0,1}^\prime&=&{\frac
{A_{{0}}c_{{0,1}}+2\,A_{{1}}c_{{1,1}}}{{A_{{0}}}^{5} \left( A_{
{0}}+A_{{1}}c_{{3,4}} \right) }}
,\\
c_{1,1}^\prime&=&{\frac {B_{{1}}c_{{1,1}}}{{A_{{0}}}^{5} \left(
A_{{0}}+A_{{1}}c_{{3,4} } \right) }}
,\\
c_{1,2}^{\prime }&=&{\frac {B_{{1}}c_{{1,2}}}{{A_{{0}}}^{2}}},\\
c_{1,3}^{\prime }&=&{\frac {B_{{1}} \left(
A_{{0}}c_{{1,3}}+2\,A_{{1}}{c_{{1,2}}}^{2}
 \right) }{{A_{{0}}}^{4}}},\\
c_{1,4}^{\prime
}&=&\frac{{A_0}^2\,{B_1}^2c_{1,4}+{A_0}^2({B_2}^2-2B_1\,B_3)\,c_{2,3}+A_1\,{B_1}^2(5\,c_{1,2}-c_{2,3})\,(A_0\,c_{1,3}-B_1\,{c_{1,2}}^2)}{{A_{{0}}}^{6}B_{{1}}},\\
c_{1,5}^{\prime}&=&-\frac{1}{{A_{{0 }}}^{8}{B_{{1}}}^{2} \left(
A_{{0}}+c_{{3,4}}A_{{1}} \right) }
({A_{{0}}}^{3}A_{{1}}{B_{{1}}}^{3}c_{{1,3}}c_{{2,4}}-3\,{A_{{0}}}^{3}A_{{2}}{B_{{1
}}}^{2}B_{{2}}c_{{1,2}}c_{{2,3}}-2\,{A_{{0}}}^{3}A_{{2}}{B_{{1}}}^{2}B_{{2}}c_{{3,4
}}c_{{1,3}}+\\
&&{A_
{{0}}}^{3}A_{{1}}B_{{1}}{B_{{2}}}^{2}c_{{3,4}}c_{{1,3}}-3\,A_{{0}}{A_{{
1}}}^{3}{B_{{1}}}^{3}{c_{{1,2}}}^{2}{c_{{2,3}}}^{2}-3\,{A_{{0}}}^{2}{A_{{1}}}^{2}{B_{{1}}}^
{3}c_{{1,3}}{c_{{2,3}}}^{2}-2\,{A_{{0}}}^{4}B_{
{1}}B_{{2}}B_{{4}}c_{{3,4}}-\\
&&6\,{A_{{0}}}^{3}A_{{1}}{B_{{1}}}^{3}c_{{1,2}}c_{{1,4}}-6\,{A_{{0}}}^{2}{A_{{1}}}^{2}{B_{{1}
}}^{3}c_{{1,2}}c_{{3,4}}c_{{1,4}}+4\,{A_{{0}}}^{2}{A_{{1}}}^{2}{B_{{1}
}}^{2}B_{{3}}c_{{1,2}}c_{{3,4}}c_{{2,3}}+\\
&&12
\,{A_{{0}}}^{2}{A_{{1}}}^{2}{B_{{1}}}^{3}c_{{1,2}}c_{{1,3}}c_{{2,3}}-{A_{{0}}}^{2}{A_{{2}}}^{2}{
B_{{1}}}^{3}{c_{{1,2}}}^{2}c_{{3,4}}-2\,{A_{{0}}}^{3}A_{{1}}{B_{
{2}}}^{3}c_{{1,2}}c_{{3,4}}+2\,{A_{{0
}}}^{3}A_{{3}}{B_{{1}}}^{3}c_{{1,3}}c_{{3,4}}+\\
&&6\,{A_{{0}}}^{2}{A_{{1}}
}^{2}B_{{1}}{B_{{2}}}^{2}{c_{{1,2}}}^{2}c_{{3,4}}+2\,{A_{{0}}}^{3}A_{{
2}}B_{{1}}{B_{{2}}}^{2}c_{{1,2}}c_{{3,4}}+2\,{A_{{0}}}^{2}A_{{1}}A_{{2}}{B_{{1}}}^{2}B_{{2}}{c_{{1,2}}}^
{2}c_{{3,4}}-\\
&&2\,{A_{{0}}}^{3}A_{{2}}{B_{{1}}}^{2}B_{{3}}
c_{{1,2}}c_{{3,4}}-14\,{A_{{0}}}^{2}{A_{{1}}}^{2}{B_{{1}}}^{2}B_{{3}}{c_{{1,2}}}^{2
}c_{{3,4}}+2\,{A_{{1}}}^{4}{B_{{1}}}^{3}{c_{{1,2}}
}^{3}c_{{3,4}}c_{{2,3}}-\\
&&2\,{A_{{0}}}^{2}{A_{{1}}}^{2}B_{{1}}{B_{{2}}}^{2}c_{{2,3}}c_{{3,4}
}c_{{1,2}}+3\,{A_{{0}}}^{3}A
_{{1}}{B_{{1}}}^{3}c_{{1,4}}c_{{2,3}}+2\,{A_{{0}}}^{3}A_{{4}}{B_{{1}}}^{3}c_{{1,2}}c_{{3,4}}
+\\
&&4\,{A_{{0}}}^{3}A_{{1}}B_{{1}}B_{{2}}B_{{3}}c_{{1,2}}c_{{3,4}}+9\,{A_{{0}}}^{3}A_{{1}}{B_
{{1}}}^{2}B_{{3}}c_{{1,2}}c_{{2,3}}+3\,{A_{{0}}}^{3}A_{{3}}{B_{{1}}}^{
3}c_{{1,2}}c_{{2,3}}-\\
&&21\,{A_{{0}}}^{2
}{A_{{1}}}^{2}{B_{{1}}}^{3}{c_{{1,2}}}^{2}c_{{1,3}}+{A_{{0}}}^{2}{A_{{1}}}^{2}{B_{{1}}}^{3}{c_{{
1,2}}}^{2}c_{{2,4}}+{A_{{0}}}^{4}B_
{{1}}{B_{{3}}}^{2}c_{{3,4}}-3\,{A_{{0}}}^{3}A_{{1}}{B_{{1}}}^{3}{c_{{1,3}}}^{2}-\\
&&14\,A_{{0}}{A_{{1}}}^{3}{B_{{1}}
}^{3}{c_{{1,2}}}^{4}-25\,{A_{{1}}}^{4
}{B_{{1}}}^{3}{c_{{1,2}}}^{4}c_{{3,4}}-2\,{A_{{0}}}^{3}A_{{3}}{B_{{1}}}^{2}B_{{2}}c
_{{1,2}}c_{{3,4}}-3\,{A_{{0}}}^{3}A_{{1}}B_{{1}}{B_{{2}}}^{2}c_{{2,3}}
c_{{1,2}}+\\
&&3\,{A_{{0}}}^{3}A_{{1}}B_{{1}}{B_{{2}}}^{2}{c_{{2,3}}}^{2}-2
\,{A_{{0}}}^{3}A_{{1}}{B_{{1}}}^{2}B_{{4}}c_{{1,2}}c_{{3,4}}-6\,{A_{{0
}}}^{3}A_{{1}}{B_{{1}}}^{2}B_{{3}}{c_{{2,3}}}^{2}-3\,{A_{{0}}}^{4}B_{{
3}}B_{{1}}B_{{2}}c_{{2,3}}-\\
&&{A_{{0}}}^{4}{B_{{1}}}^{3}c_{{1,5}}+{A_{{0}}}^{4}{B_{{2}
}}^{3}c_{{2,3}}-{A_{{0}}}^{4}B_{{1}
}{B_{{2}}}^{2}c_{{2,4}}+3\,{A_{{0}}}^{4}{B_{{1}}}^{2}B_{{4}}c_{{2,3}}+2\,{A_{{0}}}^{4
}{B_{{1}}}^{2}B_{{5}}c_{{3,4}}+\\
&&2\,{A_{{0}}}^{4}B_{{3}}{B_{{1}}}^{2}c_{{2,4}}-
2\,{A_{{0}}}^{2}{A_{{1}}}^{2}{B_{{1}}}^{3}{c_{{1,3}}}^{2}c_{{3,4}}-32
\,A_{{0}}{A_{{1}}}^{3}{B_{{1}}}^{3}c_{{1,3}}c_{{3,4}}{c_{{1,2}}}^{2}+\\
&&2
\,A_{{0}}{A_{{1}}}^{3}{B_{{1}}}^{3}c_{{1,3}}c_{{3,4}}c_{{1,2}}c_{{2,3}
}+9\,A_{{0}}{A_{{1}}}^{3}{B_{{1}}}^{3}{c_{{1,2}}}^{3}c_{{2,3}}),
\nonumber\\
c_{2,3}^{\prime }&=& {\frac {B_1\,c_{2,3}}{{A^2_0}}},\\[1mm]
c_{2,4}^{\prime }&=&\frac {1}{A_{{0}}^{4}B_{{1}} \left(
A_{{0}}+c_{{3, 4}}A_{{1}} \right)
}(A_{{0}}^{2}B_{{1}}^{2}c_{{2,4}}+ \left( -A_{{0}}^{2}B_{
{2}}^{2}+2\,A_{{0}}^{2}B_{{1}}B_{{3}}-A_{{0}}A_{{1}}B_{{1}}^{2}c_{{1,3}}
-10\,A_{{1}}^{2}B_{{1}}^{2}c_{{1,2}}^{2} \right) c_{{3,4}
}\\
&&+3\,A_{{0}}A_{{1}}B_{{1}}^{2}c_{{2,3}}c_{{1,2}}-3\,A_{{0}}A_{{1}}B_{{1}}^{2}c_
{{2,3}}^{2}),\nonumber\\[1mm]
c_{3,4}^{\prime }&=&{\frac
{B_{{1}}c_{{3,4}}}{A_{{0}}+c_{{3,4}}A_{{1}}}}.
\end{eqnarray*}
\end{theor}
\newpage

The following base change is adapted and it transforms $L(C)$ to
$L(C')$

\begin{eqnarray}\label{bc.dim08}
e_{0}^{\prime }&=&A_0\,e_0 + A_1\,e_1 + A_2\,e_2 + A_3\,e_3+ A_4\,e_4,\nonumber \\[1mm]
e_{1}^{\prime
}&=&B_1\,e_1+B_2\,e_2+B_3\,e_3+B_4\,e_4+B_5\,e_5,\nonumber \\[1mm]
e_{2}^{\prime
}&=&A_{0}B_1\,e_{2}+A_{0}B_2e_{3}+\left(A_{{0}}B_{{3}}+(A_{{2}}B_{{1}}-A_{{1}}B_{{2}})c_{{1,2}}\right)e_4+(A_{{0}}B_{{4}}+(A_{{3}}B_{{1}}-A_{{1}}B_{{3}})c_{{1,2}}+\nonumber\\
&&(A_{{2}}B_{{1}}-A_{{1}}B_{{2}})c_{{1,3}}) e_5+(\left(
A_{{3}}B_{{2}}-A_{{4}}B_{{1}}-A_{{2}}B_{{3}}+A_{{1}}B_{{4}}
 \right) c_{{2,3}}+ \left( A_{{2}}B_{{1}}-A_{{1}}B_{{2}} \right) c_{{1
,4}}+\nonumber \\[1mm]
&&\left( A_{{3}}B_{{1}}-A_{{1}}B_{{3}} \right) c_{{1,3}}+ \left( -
A_{{1}}B_{{4}}+A_{{4}}B_{{1}} \right) c_{{1,2}}+A_{{0}}B_{{5}}
\nonumber)
e_6+(\ast)e_7,\nonumber \\[1mm]
e_{3}^{\prime
}&=&A_{0}^2B_1\,e_{3}+\left({A^{2}_{{0}}}B_{{2}}-A_{{0}}{A_{{1}}}B_1c_{{1,2}}\right)e_4+(A_{{0}}^{2}B_{{3}}
-A_{{1}}A_{{0}}B_{{1}}c_{{1,3}}+A_{{0}} \left(
-2\,A_{{1}}B_{{2}}+A_{{ 2}}B_{{1}} \right) c_{{1,2}})e_5 \nonumber
\\&&+(\left( -A_{{1}}A_{{2}}B_{{1}}+{A_{{1}}}^{2}B_{{2}} \right) {c_{{1,2}}
}^{2}+ \left(  \left( A_{{1}}A_{{2}}B_{{1}}-{A_{{1}}}^{2}B_{{2}}
 \right) c_{{2,3}}+A_{{3}}A_{{0}}B_{{1}}-2\,A_{{1}}A_{{0}}B_{{3}}
 \right) c_{{1,2}}+\\
 && \left( -2\,A_{{1}}A_{{0}}B_{{2}}+A_{{0}}A_{{2}}B_{
{1}} \right) c_{{1,3}}+ \left(
A_{{3}}A_{{0}}B_{{1}}-A_{{2}}A_{{0}}B_{ {2}}+A_{{1}}A_{{0}}B_{{3}}
\right) c_{{2,3}}+A_{{1}}A_{{0}}B_{{1}}c_{{
1,4}}+\nonumber \\[1mm]
&& {A_{{0}}}^{2}B_{{4}}
)e_6+(\ast)e_7,\nonumber \\[1mm]
e_{4}^{\prime
}&=&A_{0}^3B_1\,e_{4}+({A^{3}_{{0}}}B_{{2}}-2\,{A^{2}_{{0}}}A_{{1}}B_{{1}}c_{{1,2}})e_5+A_{{0}}
( {A_{{0}}}^{2}B_{{3}}-2\,A_{{1}}A_{{0}}B_{{1}}c_{{1,3}}-
A_{{0}}A_{{2}}B_{{1}}c_{{2,3}}-
 \nonumber \\&&  3\,A_{{1}}A_{{0}}B_{{2}}c_{{1,2}}+A_{{0
}}A_{{1}}B_{{2}}c_{{2,3}}+A_{{0}}A_{{2}}B_{{1}}c_{{1,2}}+{A_{{1}}}^{2}
B_{{1}}{c_{{1,2}}}^{2}-{A_{{1}}}^{2}B_{{1}}c_{{1,2}}c_{{2,3}} )
e_6+(\ast)e_7,\nonumber \\[1mm]
e_{5}^{\prime
}&=&A^4_0\,B_1\,e_5+({A^{4}_{{0}}}B_{{2}}-3\,A_{{1}}{A^{3}_{{0}}}B_{{1}}c_{{1,2}}+A_{{1}}{A^{3}_{{0}}}B_{{1}}c_
{{2,3}}
)e_6+(\ast)e_7, \nonumber \\[1mm]
e_{6}^{\prime }&=& A^5_0\,B_1\,e_6+(\ast)e_7,\nonumber \\[1mm]
e_{7}^{\prime }&=&B_{{1}}{A_{{0}}}^{5} \left(
A_{{0}}+A_{{1}}c_{{3,4}} \right)e_7.\nonumber
\end{eqnarray}

The next section deals with the applications of the results of the
previous section to the classification problem of $TLeib_8$.

\subsubsection{  Isomorphism classes of $TLeib_{8}$} In this section we give a list of
all algebras from $TLeib_{8};$

\nt Represent $TLeib_{8}\ $as a union of the following subsets: \\

\nt $M_8^1=\{L(C)\in TLeib_{8}\ :c_{3,4}\neq 0\},$\\

$U_{8}^{1}=\{L(C)\in TLeib_{8}\ :c_{1,1}\neq 0\};$

$U_{8}^{2}=\{L(C)\in TLeib_{8}\ :c_{1,1}=0,\ c_{0,1}\neq 0\};$

$U_{8}^{3}=\{L(C)\in TLeib_{8}\ :c_{1,1}=c_{0,1}=0,\ c_{1,2}\neq
0\};$

$U_{8}^{4}=\{L(C)\in TLeib_{8}\ :c_{1,1}=c_{0,1}=c_{1,2}=0,\
c_{1,3}\neq 0,\ c_{1,4}\neq 0\};$

$U_{8}^{5}=\{L(C)\in TLeib_{8}\ :c_{1,1}=c_{0,1}=c_{1,2}=0,\
c_{1,3}\neq 0,\ c_{1,4}=0,\ c_{0,0}\neq 0\};$

$U_{8}^{6}=\{L(C)\in TLeib_{8}\ :c_{1,1}=c_{0,1}=c_{1,2}=0,\
c_{1,3}\neq 0,\ c_{1,4}=c_{0,0}=0\};$

$U_{8}^{7}=\{L(C)\in TLeib_{8}\
:c_{1,1}=c_{0,1}=c_{1,2}=c_{1,3}=0,\ c_{1,4}\neq 0,\ c_{0,0}\neq
0\};$

$U_{8}^{8}=\{L(C)\in TLeib_{8}\
:c_{1,1}=c_{0,1}=c_{1,2}=c_{1,3}=0,\ c_{1,4}\neq 0,\ c_{0,0}=
0\};$

$U_{8}^{9}=\{L(C)\in TLeib_{8}\
:c_{1,1}=c_{0,1}=c_{1,2}=c_{1,3}=c_{1,4}=0,\ c_{0,0}\neq 0\};$

$U_{8}^{10}=\{L(C)\in TLeib_{8}\
:c_{1,1}=c_{0,1}=c_{1,2}=c_{1,3}=c_{1,4}=0,\ c_{0,0}=0\};$\\

\nt $M_8^2=\{L(C)\in TLeib_{8}\ :c_{3,4}=0\}$\\

$U_{8}^{11}=\{L(C)\in TLeib_{8}\ :c_{2,3}\neq 0,\ c_{1,1}\neq
0\};$

$U_{8}^{12}=\{L(C)\in TLeib_{8}\ :c_{2,3}\neq 0,\ c_{1,1}=0,\
c_{0,1}\neq 0\};$

$U_{8}^{13}=\{L(C)\in TLeib_{8}\ :c_{2,3}\neq 0,\
c_{1,1}=c_{0,1}=0,\ c_{1,2}\neq0,\ $

$
2\,c_{{2,4}}c_{{1,2}}^{2}-3\,c_{{1,3}}c_{{2,3}}c_{{1,2}}+3\,c_{{1,3}
}c_{{2,3}}^{2}\neq0\};$

$U_{8}^{14}=\{L(C)\in TLeib_{8}\ :c_{2,3}\neq 0,\
c_{1,1}=c_{0,1}=0,\ c_{1,2}\neq0,$

$
2\,c_{{2,4}}c_{{1,2}}^{2}-3\,c_{{1,3}}c_{{2,3}}c_{{1,2}}+3\,c_{{1,3}
}c_{{2,3}}^{2}=0,\ c_{0,0}\neq0\};$

$U_{8}^{15}=\{L(C)\in TLeib_{8}\ :c_{2,3}\neq 0,\
c_{1,1}=c_{0,1}=0,\ c_{1,2}\neq0,$

$
2\,c_{{2,4}}{c_{{1,2}}}^{2}-3\,c_{{1,3}}c_{{2,3}}c_{{1,2}}+3\,c_{{1,3}
}{c_{{2,3}}}^{2}=0,\ c_{0,0}=0\};$

$U_{8}^{16}=\{L(C)\in TLeib_{8}\ :c_{2,3}\neq 0,\
c_{1,1}=c_{0,1}=c_{1,2}=0,\ c_{1,3}\neq0,\ c_{0,0}\neq0\};$

$U_{8}^{17}=\{L(C)\in TLeib_{8}\ :c_{2,3}\neq 0,\
c_{1,1}=c_{0,1}=c_{1,2}=c_{1,3}=0,\ c_{0,0}\neq0\};$

$U_{8}^{18}=\{L(C)\in TLeib_{8}\ :c_{2,3}\neq 0,\
c_{1,1}=c_{0,1}=c_{1,2}=c_{1,3}=c_{0,0}=0\};$

$U_{8}^{19}=\{L(C)\in TLeib_{8}\ :c_{2,3}=0,\  c_{2,4}\neq0,\
c_{1,2}\neq0\};$

$U_{8}^{20}=\{L(C)\in TLeib_{8}\ :c_{2,3}=0,\ c_{2,4}\neq0,\
c_{1,2}=0,\ c_{1,1}\neq0\};$

$U_{8}^{21}=\{L(C)\in TLeib_{8}\ :c_{2,3}=0,\ c_{2,4}\neq0,\
c_{1,2}=c_{1,1}=0,\ c_{0,1}\neq0\};$

$U_{8}^{22}=\{L(C)\in TLeib_{8}\ :c_{2,3}=0,\ c_{2,4}\neq0,\
c_{1,2}=c_{1,1}=c_{0,1}=0,\ c_{1,4}\neq0\};$

$U_{8}^{23}=\{L(C)\in TLeib_{8}\ :c_{2,3}=0,\ c_{2,4}\neq0,\
c_{1,2}=c_{1,1}=c_{0,1}=c_{1,4}=0,\ c_{0,0}\neq0\};$

$U_{8}^{24}=\{L(C)\in TLeib_{8}\ :c_{2,3}=0,\ c_{2,4}\neq0,\
c_{1,2}=c_{1,1}=c_{0,1}=c_{1,4}=c_{0,0}=0\};$

$U_{8}^{25}=\{L(C)\in TLeib_{8}\ :c_{2,3}=c_{2,4}=0,\
c_{1,2}\neq0,\ 4\,c_{{1,4}}c_{{1,2}}-5\,c_{{1,3}}^{2}\neq0\};$

$U_{8}^{26}=\{L(C)\in TLeib_{8}\ :c_{2,3}=c_{2,4}=0,\
c_{1,2}\neq0,\ 4\,c_{{1,4}}c_{{1,2}}-5\,c_{{1,3}}^{2}=0,\
-7\,c_{{1,3}}^{3}+4\,c_{{1,5}}c_{{1,2}}^{2}\neq0\};$

$U_{8}^{27}=\{L(C)\in TLeib_{8}\ :c_{2,3}=c_{2,4}=0,\
c_{1,2}\neq0,\ 4\,c_{{1,4}}c_{{1,2}}-5\,c_{{1,3}}^{2}=0,\ $

$ -7\,c_{{1,3}}^{3}+4\,c_{{1,5}}c_{{1,2}}^{2}=0,\ c_{1,1}\neq0\};$

$U_{8}^{28}=\{L(C)\in TLeib_{8}\ :c_{2,3}=c_{2,4}=0,\
c_{1,2}\neq0,\ 4\,c_{{1,4}}c_{{1,2}}-5\,c_{{1,3}}^{2}=0,\ $

$ -7\,c_{{1,3}}^{3}+4\,c_{{1,5}}c_{{1,2}}^{2}=0,\ c_{1,1}=0,\
c_{0,1}\neq0\};$

$U_{8}^{29}=\{L(C)\in TLeib_{8}\ :c_{2,3}=c_{2,4}=0,\
c_{1,2}\neq0,\ 4\,c_{{1,4}}c_{{1,2}}-5\,c_{{1,3}}^{2}=0,\ $

$ -7\,c_{{1,3}}^{3}+4\,c_{{1,5}}c_{{1,2}}^{2}=0,\
c_{1,1}=c_{0,1}=0,\ c_{0,0}\neq0\};$

$U_{8}^{30}=\{L(C)\in TLeib_{8}\ :c_{2,3}=c_{2,4}=0,\
c_{1,2}\neq0,\ 4\,c_{{1,4}}c_{{1,2}}-5\,c_{{1,3}}^{2}=0,\ $

$ -7\,c_{{1,3}}^{3}+4\,c_{{1,5}}c_{{1,2}}^{2}=0,\
c_{1,1}=c_{0,1}=c_{0,0}=0\};$

$U_{8}^{31}=\{L(C)\in TLeib_{8}\ :c_{2,3}=c_{2,4}=c_{1,2}=0,\
 c_{1,3}\neq 0,\ c_{1,4}\neq 0\};$

$U_{8}^{32}=\{L(C)\in TLeib_{8}\ :c_{2,3}=c_{2,4}=c_{1,2}=0,\
 c_{1,3}\neq 0,\ c_{1,4}=0,\ c_{1,1}\neq 0\};$

$U_{8}^{33}=\{L(C)\in TLeib_{8}\ :c_{2,3}=c_{2,4}=c_{1,2}=0,\
 c_{1,3}\neq 0,\ c_{1,4}= c_{1,1}=0,\ c_{0,1}\neq 0\};$

$U_{8}^{34}=\{L(C)\in TLeib_{8}\ :c_{2,3}=c_{2,4}=c_{1,2}=0,\
 c_{1,3}\neq 0,\ c_{1,4}= c_{1,1}=c_{0,1}=0,\ c_{0,0}\neq 0\};$

$U_{8}^{35}=\{L(C)\in TLeib_{8}\ :c_{2,3}=c_{2,4}=c_{1,2}=0,\
 c_{1,3}\neq 0,\ c_{1,4}= c_{1,1}=c_{0,1}=c_{0,0}=0\};$

$U_{8}^{36}=\{L(C)\in TLeib_{8}\
:c_{2,3}=c_{2,4}=c_{1,2}=c_{1,3}=0,\
 c_{1,1}\neq 0,\ c_{1,5}\neq0\};$

$U_{8}^{37}=\{L(C)\in TLeib_{8}\
:c_{2,3}=c_{2,4}=c_{1,2}=c_{1,3}=0,\
 c_{1,1}\neq 0,\ c_{1,5}=0,\ c_{1,4}\neq0\};$

$U_{8}^{38}=\{L(C)\in TLeib_{8}\
:c_{2,3}=c_{2,4}=c_{1,2}=c_{1,3}=0,\
 c_{1,1}\neq 0,\ c_{1,5}=c_{1,4}=0,\ $

 $4\,c_{{0,0}}c_{{1,1}}-c_{{0,1}}^{2}\neq0\};$

$U_{8}^{39}=\{L(C)\in TLeib_{8}\
:c_{2,3}=c_{2,4}=c_{1,2}=c_{1,3}=0,c_{1,1}\neq 0,\
c_{1,5}=c_{1,4}=0,\ $

$
 4\,c_{{0,0}}c_{{1,1}}-c_{{0,1}}^{2}=0\};$

$U_{8}^{40}=\{L(C)\in TLeib_{8}\
:c_{2,3}=c_{2,4}=c_{1,2}=c_{1,3}=c_{1,1}=0,\
  c_{0,1}\neq 0,\ c_{1,5}\neq 0\};$

$U_{8}^{41}=\{L(C)\in TLeib_{8}\
:c_{2,3}=c_{2,4}=c_{1,2}=c_{1,3}=c_{1,1}=0,\
  c_{0,1}\neq 0,\ c_{1,5}= 0,\ c_{1,4}\neq0\};$

$U_{8}^{42}=\{L(C)\in TLeib_{8}\
:c_{2,3}=c_{2,4}=c_{1,2}=c_{1,3}=c_{1,1}=0,\
  c_{0,1}\neq 0,\ c_{1,5}=c_{1,4}=0\};$

$U_{8}^{43}=\{L(C)\in TLeib_{8}\
:c_{2,3}=c_{2,4}=c_{1,2}=c_{1,3}=c_{1,1}=c_{0,1}=0,
 c_{1,4}\neq 0,\ c_{1,5}\neq 0\};$

$U_{8}^{44}=\{L(C)\in TLeib_{8}\
:c_{2,3}=c_{2,4}=c_{1,2}=c_{1,3}=c_{1,1}=c_{0,1}=0,\
 c_{1,4}\neq 0,\ c_{1,5}=0,\ c_{0,0}\neq0\};$

$U_{8}^{45}=\{L(C)\in TLeib_{8}\
:c_{2,3}=c_{2,4}=c_{1,2}=c_{1,3}=c_{1,1}=c_{0,1}=0,\
 c_{1,4}\neq 0,\ c_{1,5}=c_{0,0}=0\};$

$U_{8}^{46}=\{L(C)\in TLeib_{8}\
:c_{2,3}=c_{2,4}=c_{1,2}=c_{1,3}=c_{1,1}=c_{0,1}= c_{1,4}=0,\
 c_{1,5}\neq 0,\ c_{0,0}\neq 0\};$

$U_{8}^{47}=\{L(C)\in TLeib_{8}\
:c_{2,3}=c_{2,4}=c_{1,2}=c_{1,3}=c_{1,1}=c_{0,1}= c_{1,4}=0,\
 c_{1,5}\neq 0,\ c_{0,0}=0\};$

$U_{8}^{48}=\{L(C)\in TLeib_{8}\
:c_{2,3}=c_{2,4}=c_{1,2}=c_{1,3}=c_{1,1}=c_{0,1}= c_{1,4}=
c_{1,5}=0,\ c_{0,0}\neq0\};$

$U_{8}^{49}=\{L(C)\in TLeib_{8}\
:c_{2,3}=c_{2,4}=c_{1,2}=c_{1,3}=c_{1,1}=c_{0,1}= c_{1,4}=
c_{1,5}=c_{0,0}=0\}.$

\begin{pr}\emph{}
\begin{enumerate}
\item Two algebras $L(C)$ and $L(C')$ from $U_{8}^{1}$
are isomorphic, if and only if\\
 $ \ds {\frac { \left(
4\,c'_{{0,0}}c'_{{1,1}}-{c'}_{{0,1}}^{2} \right) {c'}_{{3,4
}}^{2}}{ \left( -2\,c'_{{1,1}}+c'_{{0,1}}c'_{{3,4}} \right)
^{2}}}= {\frac { \left( 4\,c_{{0,0}}c_{{1,1}}-c_{{0,1}}^{2}
\right) c_{{3,4 }}^{2}}{ \left( -2\,c_{{1,1}}+c_{{0,1}}c_{{3,4}}
\right) ^{2}}}
 ,\ \ {\frac { \left( -2\,c'_{{1,1}}+c'_{{0,1}}c'_{{3,4}} \right) ^{5}{c'}_{{1,2}
}^{5}}{{c'}_{{3,4}}^{4}{c'}_{{1,1}}^{6}}}={\frac { \left(
-2\,c_{{1,1}}+c_{{0,1}}c_{{3,4}} \right) ^{5}c_{{1,2}
}^{5}}{c_{{3,4}}^{4}c_{{1,1}}^{6}}},\\ {\frac { \left(
-2\,c'_{{1,1}}+c'_{{0,1}}c'_{{3,4}} \right) ^{5} \left( -
c'_{{1,3}}c'_{{1,1}}+{c'}_{{1,2}}^{2}c'_{{0,1}} \right)
^{5}}{{c'}_{{3,4}}^ {3}{c'}_{{1,1}}^{12}}}={\frac { \left(
-2\,c_{{1,1}}+c_{{0,1}}c_{{3,4}} \right) ^{5} \left( -
c_{{1,3}}c_{{1,1}}+c_{{1,2}}^{2}c_{{0,1}} \right) ^{5}}{c_{{3,4}}^
{3}c_{{1,1}}^{12}}},\\ \frac { \left(
-2\,c'_{{1,1}}+c'_{{0,1}}c'_{{3,4}} \right)
^{5}}{{c'}_{{1,1}}^{18}{c'}_{{3,4}}^{7}} (
27\,c'_{{3,4}}{c'}_{{1,2}}^{3}{c'}_{{0,1}}^{2}-18\,c'_{{3,4}}c'_{{1,2}}c'_{{0
,1}}c'_{{1,3}}c'_{{1,1}}+4\,c'_{{3,4}}c'_{{1,4}}{c'}_{{1,1}}^{2}-72\,{c'}_{{1
,2}}^{3}c'_{{0,1}}c'_{{1,1}}-\\
8\,c'_{{1,2}}c'_{{2,4}}{c'}_{{1,1}}^{2}
 ) ^{5}= \frac { \left(
-2\,c_{{1,1}}+c_{{0,1}}c_{{3,4}} \right)
^{5}}{c_{{1,1}}^{18}c_{{3,4}}^{7}} (
27\,c_{{3,4}}c_{{1,2}}^{3}c_{{0,1}}^{2}-18\,c_{{3,4}}c_{{1,2}}c_{{0
,1}}c_{{1,3}}c_{{1,1}}+4\,c_{{3,4}}c_{{1,4}}c_{{1,1}}^{2}\\
-72\,c_{{1
,2}}^{3}c_{{0,1}}c_{{1,1}}-8\,c_{{1,2}}c_{{2,4}}c_{{1,1}}^{2}
 ) ^{5}.
 $
 \item For any $\lambda_1, \lambda_2, \lambda_3, \lambda_4 \in \mathbb{C}, $  there exists $L(C)\in
U_{8}^{1}:$  $ \ds{\frac { \left(
4\,c_{{0,0}}c_{{1,1}}-c_{{0,1}}^{2} \right) c_{{3,4 }}^{2}}{
\left( -2\,c_{{1,1}}+c_{{0,1}}c_{{3,4}} \right) ^{2}}}=\lambda_1,
\\ {\frac { \left( -2\,c'_{{1,1}}+c'_{{0,1}}c'_{{3,4}} \right)
^{5}{c'}_{{1,2}
}^{5}}{{c'}_{{3,4}}^{4}{c'}_{{1,1}}^{6}}}=\lambda_2, \ \ {\frac {
\left( -2\,c_{{1,1}}+c_{{0,1}}c_{{3,4}} \right) ^{5} \left( -
c_{{1,3}}c_{{1,1}}+c_{{1,2}}^{2}c_{{0,1}} \right) ^{5}}{c_{{3,4}}^
{3}c_{{1,1}}^{12}}}=\lambda_3,\\
\frac { \left( -2\,c_{{1,1}}+c_{{0,1}}c_{{3,4}} \right)
^{5}}{c_{{1,1}}^{18}c_{{3,4}}^{7}} (
27\,c_{{3,4}}c_{{1,2}}^{3}c_{{0,1}}^{2}-18\,c_{{3,4}}c_{{1,2}}c_{{0
,1}}c_{{1,3}}c_{{1,1}}+4\,c_{{3,4}}c_{{1,4}}c_{{1,1}}^{2}
-72\,c_{{1
,2}}^{3}c_{{0,1}}c_{{1,1}}\\
-8\,c_{{1,2}}c_{{2,4}}c_{{1,1}}^{2}
 ) ^{5}=\lambda_4.$

Then orbits in $U_{8}^{1}$ can be parameterized as
$L\left(\lambda_1, 0, 1, \lambda_2, \lambda_3, \lambda_4, 0,
-32\lambda_2,0,1\right),$ $\lambda_1, \lambda_2, \lambda_3,
\lambda_4 \in \mathbb{C}.$
\end{enumerate}

\begin{proof}See proof of Proposition \ref{p2}, where we put $\ds A_1={\frac {-A_{{0}}c_{{0,1}}}{2\,c_{{1,1}}}},\ B_1={\frac {-A_{{0}} \left( -2\,c_{{1,1}}+c_{{0,1}}c_{{3,4}} \right)
}{2\,c_{{3,4}}c_{{1,1}}}}
,$ and \\
$ \ds B_3=-\frac{1}{2\,A_{{0}}{c_{{ 1,1}}}^{3}{c_{{3,4}}}^{2}
\left( -2\,c_{{1,1}}+c_{{0,1}}c_{{3,4}}
 \right) }(-8\,{A_{{0}}}^{2}c_{{2,4}}{c_{{1,1}}}^{4}+8\,{A_{{0}}}^{
2}c_{{2,4}}c_{{0,1}}c_{{3,4}}{c_{{1,1}}}^{3}-2\,{A_{{0}}}^{2}c_{{2,4}}
{c_{{0,1}}}^{2}{c_{{3,4}}}^{2}{c_{{1,1}}}^{2}-72\,{A_{{0
}}}^{2}{c_{{1,2}}}^{2}c_{{0,1}}{c_{{1,1}}}^{3}+92\,{A_{{0}}}^{2}{c_{{1,2}}}^{2}{c_{{0
,1}}}^{2}c_{{3,4}}{c_{{1,1}}}^{2}-38\,{A_{{0}}}^{2}{c_{{1,2}}}^{2}{c_{
{0,1}}}^{3}{c_{{3,4}}}^{2}c_{{1,1}}-4\,{A_{{0}}}^{2}c_{{0,1}}c_{{1,3}}
c_{{3,4}}{c_{{1,1}}}^{3}+\\
4\,{A_{{0}}}^{2}{c_{{0,1}}}^{2}c_{{1,3}}{c_{{
3,4}}}^{2}{c_{{1,1}}}^{2}-{A_{{0}}}^{2}{c_{{0,1}}}^{3}c_{{1,3}}{c_{{3,
4}}}^{3}c_{{1,1}}+8\,{B_{{2}}}^{2}{c_{{3,4}}}^{3}{c_{{1,1}}}^{4}+5\,{A_{{0}}}^{2}{c
_{{1,2}}}^{2}{c_{{0,1}}}^{4}{c_{{3,4}}}^{3}) ,$ in base change
(\ref{bc.dim08}).
\end{proof}

\end{pr}

\begin{pr}\emph{}
\begin{enumerate}
\item Two algebras $L(C)$ and $L(C')$ from $U_{8}^{2}$
are isomorphic, if and only if\\
 $$ \ds {\frac { \left( -c'_{{0,1}}+c'_{{0,0}}c'_{{3,4}} \right)
^{6}{c'}_{{1,2}}^ {5}}{{c'}_{{3,4}}^{5}{c'}_{{0,1}}^{7}}}={\frac {
\left( -c_{{0,1}}+c_{{0,0}}c_{{3,4}} \right) ^{6}c_{{1,2}}^
{5}}{c_{{3,4}}^{5}c_{{0,1}}^{7}}}
 ,$$ $$ {\frac { \left( -c'_{{0,1}}+c'_{{0,0}}c'_{{3,4}} \right) ^{7} \left( -c'_{
{1,3}}c'_{{0,1}}+2\,{c'}_{{1,2}}^{2}c'_{{0,0}} \right)
^{5}}{{c'}_{{3,4}}^ {5}{c'}_{{0,1}}^{14}}}={\frac { \left(
-c_{{0,1}}+c_{{0,0}}c_{{3,4}} \right) ^{7} \left( -c_{
{1,3}}c_{{0,1}}+2\,c_{{1,2}}^{2}c_{{0,0}} \right) ^{5}}{c_{{3,4}}^
{5}c_{{0,1}}^{14}}} ,$$ $ \ds \frac { \left(
-c'_{{0,1}}+c'_{{0,0}}c'_{{3,4}} \right)
^{8}}{{c'}_{{0,1}}^{21}{c'}_{{3,4}}^{10}} ( 27\,
c'_{{3,4}}{c'}_{{1,2}}^{3}{c'}_{{0,0}}^{2}-9\,c'_{{3,4}}c'_{{1,2}}c'_{{0,0}}
c'_{{1,3}}c'_{{0,1}}+c'_{{3,4}}c'_{{1,4}}{c'}_{{0,1}}^{2}-36\,{c'}_{{1,2}}^{
3}c'_{{0,0}}c'_{{0,1}}-\\
2\,c'_{{1,2}}c'_{{2,4}}{c'}_{{0,1}}^{2} ) ^{5}=\frac { \left(
-c_{{0,1}}+c_{{0,0}}c_{{3,4}} \right)
^{8}}{c_{{0,1}}^{21}c_{{3,4}}^{10}} ( 27\,
c_{{3,4}}c_{{1,2}}^{3}c_{{0,0}}^{2}-9\,c_{{3,4}}c_{{1,2}}c_{{0,0}}
c_{{1,3}}c_{{0,1}}+c_{{3,4}}c_{{1,4}}c_{{0,1}}^{2}-\\
36\,c_{{1,2}}^{
3}c_{{0,0}}c_{{0,1}}-2\,c_{{1,2}}c_{{2,4}}c_{{0,1}}^{2} ) ^{5}.
 $
 \item For any $\lambda_1, \lambda_2, \lambda_3 \in \mathbb{C}, $  there exists $L(C)\in
U_{8}^{2}:$  $ \ds {\frac { \left( -c_{{0,1}}+c_{{0,0}}c_{{3,4}}
\right) ^{6}c_{{1,2}}^
{5}}{c_{{3,4}}^{5}c_{{0,1}}^{7}}}=\lambda_1,
\\ {\frac { \left(
-c_{{0,1}}+c_{{0,0}}c_{{3,4}} \right) ^{7} \left( -c_{
{1,3}}c_{{0,1}}+2\,c_{{1,2}}^{2}c_{{0,0}} \right) ^{5}}{c_{{3,4}}^
{5}c_{{0,1}}^{14}}}=\lambda_2, \ \ \frac { \left(
-c_{{0,1}}+c_{{0,0}}c_{{3,4}} \right)
^{8}}{c_{{0,1}}^{21}c_{{3,4}}^{10}} ( 27\,
c_{{3,4}}c_{{1,2}}^{3}c_{{0,0}}^{2}-\\
9\,c_{{3,4}}c_{{1,2}}c_{{0,0}}
c_{{1,3}}c_{{0,1}}+c_{{3,4}}c_{{1,4}}c_{{0,1}}^{2}-
36\,c_{{1,2}}^{
3}c_{{0,0}}c_{{0,1}}-2\,c_{{1,2}}c_{{2,4}}c_{{0,1}}^{2} )
^{5}=\lambda_3.$

Then orbits in $U_{8}^{2}$ can be parameterized as $L\left(0, 1,
0, \lambda_1, \lambda_2, \lambda_3, 0, -32\lambda_1,0,1\right), \
\lambda_1, \lambda_2, \lambda_3 \in \mathbb{C}.$
\end{enumerate}

\begin{proof}The respective value of coefficients in the base change (\ref{bc.dim08}) are: $$ A_1=-{\frac
{A_{{0}}c_{{0,0}}}{c_{{0,1}}}},\ \ B_1=-{\frac {A_{{0}} \left(
-c_{{0,1}}+c_{{0,0}}c_{{3,4}} \right) }{c_{{3, 4}}c_{{0,1}}}},$$
and \\ $ \ds B_3=-\frac{1}{2\,A_{{0}}{c_{{0,1}}}^{3}
{c_{{3,4}}}^{2} \left( -c_{{0,1}}+c_{{0,0}}c_{{3,4}} \right)
}(-{A_{{0}}}^{2}c_{{2,4}}{c_{{0,1}}}^{4}+2\,{A_{{0}}}^{2}c
_{{2,4}}c_{{0,0}}c_{{3,4}}{c_{{0,1}}}^{3}-{A_{{0}}}^{2}c_{{2,4}}{c_{{0
,0}}}^{2}{c_{{3,4}}}^{2}{c_{{0,1}}}^{2}-18\,{A_{{0}}}^{
2}{c_{{1,2}}}^{2}c_{{0,0}}{c_{{0,1}}}^{3}+46\,{A_{{0}}}^{2}{c_{{1,2}}}^{2}{c_{{0,0}}}
^{2}c_{{3,4}}{c_{{0,1}}}^{2}-38\,{A_{{0}}}^{2}{c_{{1,2}}}^{2}{c_{{0,0}
}}^{3}{c_{{3,4}}}^{2}c_{{0,1}}-
{A_{{0}}}^{2}c_{{0,0}}c_{{1,3}}c_{{3,4} }{c_{{0,1}}}^{3}+\\
2\,{A_{{0}}}^{2}{c_{{0,0}}}^{2}c_{{1,3}}{c_{{3,4}}}^{
2}{c_{{0,1}}}^{2}-{A_{{0}}}^{2}{c_{{0,0}}}^{3}c_{{1,3}}{c_{{3,4}}}^{3}
c_{{0,1}}+{B_{{2}}}^{2}{c_{{3,4}}}^{3}{c_{{0,1}}}^{4}+10\,{A_{{0}}}^{2}{c_{{1,2}}}^
{2}{c_{{0,0}}}^{4}{c_{{3,4}}}^{3})
 .$
\end{proof}
\end{pr}
\newpage

\begin{pr}\emph{}
\begin{enumerate}
\item Two algebras $L(C)$ and $L(C')$ from $U_{8}^{3}$
are isomorphic, if and only if\\
 $$ \frac {c'_{{0,0}}{c'}_{{1,2}}^{9}{c'}_{{3,4}}^{6}}{ \left( -2\,{c'}_{{1,2}
}^{2}+c'_{{1,3}}c'_{{3,4}} \right) ^{7}}=\frac
{c_{{0,0}}c_{{1,2}}^{9}c_{{3,4}}^{6}}{ \left( -2\,c_{{1,2}
}^{2}+c_{{1,3}}c_{{3,4}} \right) ^{7}}
 ,$$ $$\frac {c'_{{3,4}}}{ \left( -2\,{b}_{{1,2}}^{2}+c'_{{1,3}}c'_{{3,4}} \right) ^{2}
} \left( 9\,c'_{{3,4}}{c'}_{{1,3}}^{2}+4\,c'_{{3,4}}c'_{{1
,4}}c'_{{1,2}}-72\,{c'}_{{1,2}}^{2}c'_{{1,3}}-8\,{c'}_{{1,2}}^{2}c'_{{2,4}}
 \right)=$$ $$ \frac {c_{{3,4}}}{ \left( -2\,c_{{1,2}}^{2}+c_{{1,3}}c_{{3,4}} \right) ^{2}
} \\ \left( 9\,c_{{3,4}}c_{{1,3}}^{2}+4\,c_{{3,4}}c_{{1
,4}}c_{{1,2}}-72\,c_{{1,2}}^{2}c_{{1,3}}-8\,c_{{1,2}}^{2}c_{{2,4}}
 \right)  .
 $$
 \item For any $\lambda_1, \lambda_2\in \mathbb{C} , $  there exists $L(C)\in
U_{8}^{3}:$  $ \ds \frac {c_{{0,0}}c_{{1,2}}^{9}c_{{3,4}}^{6}}{
\left( -2\,c_{{1,2} }^{2}+c_{{1,3}}c_{{3,4}} \right)
^{7}}=\lambda_1,$ $$
 \frac {c_{{3,4}}}{ \left( -2\,c_{{1,2}}^{2}+c_{{1,3}}c_{{3,4}} \right) ^{2}
}  \left( 9\,c_{{3,4}}c_{{1,3}}^{2}+4\,c_{{3,4}}c_{{1
,4}}c_{{1,2}}-72\,c_{{1,2}}^{2}c_{{1,3}}-8\,c_{{1,2}}^{2}c_{{2,4}}
 \right)  =\lambda_2.$$

Then orbits in $U_{8}^{3}$ can be parameterized as
$L\left(\lambda_1,0,0,1,0, \lambda_2,0,-2,0,1\right), \ \lambda_1,
\lambda_2\in \mathbb{C}.$
\end{enumerate}

\begin{proof}Here, $$ A_0={\frac {2\,{c_{{1,2}}}^{2}
-c_{{1,3}}c_{{3,4}}}{2\,c_{{3,4}}c_{{1,2}} }},\
 A_1=-{\frac { \left( 2\,{c_{{1,2}}}^{2}-c_{{1,3}}c_{{3,4}} \right) c_
{{1,3}}}{4\,c_{{3,4}}{c_{{1,2}}}^{3}}} ,\ B_1={\frac {
\left( 2\,{c_{{1,2}}}^{2}-c_{{1,3}}c_{{3,4}} \right) ^{2
}}{4\,{c_{{3,4}}}^{2}{c_{{1,2}}}^{3}}} ,$$ and \\ $ \ds
B_3=\frac{1}{8\,{c_{{3,4}}}^{3}{c_{{1,2}}}^{5} \left(
2\,{c_{{1,2}}}^{2} -c_{{1,3}}c_{{3,4}} \right) ^{2}}
(-16\,c_{{2,4}}{c_{{1,2}}}^{10}+32\,c_{{2,4}}{c_{{1,2}}}^{
8}c_{{1,3}}c_{{3,4}}-24\,c_{{2,4}}{c_{{1,2}}}^{6}{c_{{1,3}}}^{2}{c_{{3
,4}}}^{2}+8\,c_{{2,4}}{c_{{1,3}}}^{3}{c_{{3,4}}}^{3}{c_{{1,2}}}^{4}-c_
{{2,4}}{c_{{1,3}}}^{4}{c_{{3,4}}}^{4}{c_{{1,2}}}^{2}-144\,c_{{1,3}}{c_
{{1,2}}}^{10}+320\,{c_{{1,3}}}^{2}{c_{{1,2}}}^{8}c_{{3,4}}-280\,{c_{{1
,3}}}^{3}{c_{{1,2}}}^{6}{c_{{3,4}}}^{2}+120\,{c_{{1,3}}}^{4}{c_{{3,4}}
}^{3}{c_{{1,2}}}^{4}-25\,{c_{{1,3}}}^{5}{c_{{3,4}}}^{4}{c_{{1,2}}}^{2}
+2\,{c_{{1,3}}}^{6}{c_{{3,4}}}^{5}+16\,{B_{{2}}}^{2}{c_{{3,4}}}^{5}{c_
{{1,2}}}^{8}) .$
\end{proof}
\end{pr}

\begin{pr}\emph{}
\begin{enumerate}
\item Two algebras $L(C)$ and $L(C')$ from $U_{8}^{4}$ are
isomorphic, if and only if\ \
 $ \ds {\frac {{c'}_{{1,3}}^{11}c'_{{0,0}}}{{c'}_{{1,4}}^{9}c'_{{3,4}}}}={\frac {c_{{1,3}}^{11}c_{{0,0}}}{c_{{1,4}}^{9}c_{{3,4}}}}
  .
 $
 \item For any $\lambda\in \mathbb{C}, $  there exists $L(C)\in
U_{8}^{4}:$  $ \ds{\frac
{c_{{1,3}}^{11}c_{{0,0}}}{c_{{1,4}}^{9}c_{{3,4}}}}=\lambda.$

Then orbits in $U_{8}^{4}$ can be parameterized as
$L\left(\lambda,0,0,0,1,1,0,0,0,1\right), \ \lambda\in
\mathbb{C}.$
\end{enumerate}

\begin{proof}For this case, we put in the base change (\ref{bc.dim08}) the following coefficients, $$ A_0={\frac {c_{{1,4}}}{c_{{1,3}}}},\ B_1={\frac
{c_{{1,4}}+A_{{1}}c_{{3,4}}c_{{1,3}}}{c_{{3,4}}c_{{1,3}}}},$$ and\\ $
\ds B_3=\frac{1}{2\,c_{{1,3}}{c_{{1,4}}}^{2}{c_{{3,4}}}^{2} \left(
c_{{1,4}}+A_{ {1}}c_{{3,4}}c_{{1,3}} \right)
}(-{c_{{1,4}}}^{4}c_{{2,4}}-2\,A_{{
1}}{c_{{1,4}}}^{3}c_{{2,4}}c_{{3,4}}c_{{1,3}}-{A_{{1}}}^{2}{c_{{1,4}}}^{2}c_{{2,4}}{c_{{3,4}}}
^{2}{c_{{1,3}}}^{2}+18\,A_{{1}}{c_{{1,2}}}^{2}{c_{{1,4}}}^{3}c_{{1,3}}
+46\,{A_{{1}}}^{2}{c_{{1,2}}}^{2}{c_{{1,4}}}^{2}c_{{3,4}}{c_{{1,3}}}^{
2}+38\,{A_{{1}}}^{3}{c_{{1,2}}}^{2}c_{{1,4}}{c_{{3,4}}}^{2}{c_{{1,3}}}
^{3}+A_{{1}}{c_{{1,4}}}^{3}c_{{3,4}}{c_{{1,3}}}^{2}+2\,{A_{{1}}}^{2}{c
_{{1,4}}}^{2}{c_{{3,4}}}^{2}{c_{{1,3}}}^{3}+{A_{{1}}}^{3}c_{{1,4}}{c_{
{3,4}}}^{3}{c_{{1,3}}}^{4}+{B_{{2}}}^{2}{c_{{1,4}}}^{2}{c_{{3,4}}}^{3}
{c_{{1,3}}}^{2}+10\,{A_{{1}}}^{4}{c_{{1,2}}}^{2}{c_{{3,4}}}^{3}{c_{{1,
3}}}^{4}).$
\end{proof}
\end{pr}

\begin{pr}\emph{}
\begin{enumerate}
\item Two algebras $L(C)$ and $L(C')$ from $U_{8}^{11}$
are isomorphic, if and only if\\
 $$ {\frac { \left( 4\,c'_{{0,0}}c'_{{1,1}}-{c'}_{{0,1}}^{2} \right) ^{2}{c'}_{
{2,3}}^{5}}{{c'}_{{1,1}}^{5}}}= {\frac { \left(
4\,c_{{0,0}}c_{{1,1}}-c_{{0,1}}^{2} \right) ^{2}c_{
{2,3}}^{5}}{c_{{1,1}}^{5}}},$$ $${\frac
{c'_{{1,2}}}{c'_{{2,3}}}}={\frac {c_{{1,2}}}{c_{{2,3}}}},\ \
{\frac { \left( -c'_{{1,3}}c'_{{1,1}}+{c'}_{{1,2}}^{2}c'_{{0,1}}
\right) ^ {4}}{{c'}_{{2,3}}^{3}{c'}_{{1,1}}^{5}}}={\frac { \left(
-c_{{1,3}}c_{{1,1}}+c_{{1,2}}^{2}c_{{0,1}} \right) ^
{4}}{c_{{2,3}}^{3}c_{{1,1}}^{5}}},$$ $$ {\frac { \left(
2\,c'_{{2,4}}c'_{{1,1}}-3\,c'_{{2,3}}c'_{{0,1}}c'_{{1,2}}+3
\,{c'}_{{2,3}}^{2}c'_{{0,1}} \right)
^{4}}{{c'}_{{2,3}}^{3}{c'}_{{1,1}}^{5 }}}={\frac { \left(
2\,c_{{2,4}}c_{{1,1}}-3\,c_{{2,3}}c_{{0,1}}c_{{1,2}}+3
\,c_{{2,3}}^{2}c_{{0,1}} \right) ^{4}}{c_{{2,3}}^{3}c_{{1,1}}^{5
}}}
 .
 $$
 \item For any $\lambda_1, \lambda_2, \lambda_3, \lambda_4 \in \mathbb{C}, $  there exists $L(C)\in
U_{8}^{11}:$  $$ {\frac { \left(
4\,c_{{0,0}}c_{{1,1}}-c_{{0,1}}^{2} \right) ^{2}c_{
{2,3}}^{5}}{c_{{1,1}}^{5}}}=\lambda_1, \ \ {\frac
{c_{{1,2}}}{c_{{2,3}}}}=\lambda_2, $$ $$  {\frac { \left(
-c_{{1,3}}c_{{1,1}}+c_{{1,2}}^{2}c_{{0,1}} \right) ^
{4}}{c_{{2,3}}^{3}c_{{1,1}}^{5}}}=\lambda_3,\ {\frac { \left(
2\,c_{{2,4}}c_{{1,1}}-3\,c_{{2,3}}c_{{0,1}}c_{{1,2}}+3
\,c_{{2,3}}^{2}c_{{0,1}} \right) ^{4}}{c_{{2,3}}^{3}c_{{1,1}}^{5
}}}=\lambda_4.$$

Then orbits in $U_{8}^{11}$ can be parameterized as
$L\left(\lambda_1, 0, 1, \lambda_2, \lambda_3,0,0,1, \lambda_4,
0\right),\ \ \lambda_1, \lambda_2, \lambda_3, \lambda_4 \in
\mathbb{C}.$
\end{enumerate}

\begin{proof}For this case we put, $\ds A_1={\frac
{-A_{{0}}c_{{0,1}}}{2\,c_{{1,1}}}},$
$\ds B_1={\frac {{A_{{0}}}^{2}}{c_{{2,3}}}},$ $\ds
B_3=\frac{1}{8\,{A_{{0}}}^{2}{c_{{2,3}}}^{2}{c_{{1,1}}}^ {2}}
(4\,{B_{{2}}}^{2}{c_{{2,3}}}^{3}{c_{{1,1}}}^{2}+4\,{A_{{0}
}}^{4}c_{{1,4}}{c_{{1,1}}}^{2}+2\,{A_{{0}}}^{4}c_{{0,1}}c_{{1,3}}c_{{2
,3}}c_{{1,1}}-10\,{A_{{0}}}^{4}c_{{1,2}}c_{{0,1}}c_{{1,3}}c_{{1,1}}+5
\,{A_{{0}}}^{4}{c_{{1,2}}}^{3}{c_{{0,1}}}^{2}-{A_{{0}}}^{4}{c_{{1,2}}}
^{2}{c_{{0,1}}}^{2}c_{{2,3}}) $ and $ \ds
B_4=\frac{1}{48\,{A_{{0}}}^{4}{c_{{2,3}}}^{3}{c_{{1,1}} }^{3}}
(8\,{B_{{2}}}^{3}{c_{{2,3}}}^{5}{c_{{1,1}}}^{3}-20\,{A_{{0
}}}^{6}c_{{2,4}}{c_{{1,2}}}^{3}{c_{{0,1}}}^{2}c_{{1,1}}+16\,{A_{{0}}}^
{6}c_{{1,5}}c_{{2,3}}{c_{{1,1}}}^{3}-16\,{A_{{0}}}^{6}c_{{2,4}}c_{{1,4
}}{c_{{1,1}}}^{3}-24\,{A_{{0}}}^{6}c_{{0,1}}{c_{{1,3}}}^{2}c_{{2,3}}{c
_{{1,1}}}^{2}+17\,{A_{{0}}}^{6}{c_{{0,1}}}^{3}{c_{{1,2}}}^{4}c_{{2,3}}
+12\,{B_{{2}}}^{2}{c_{{2,3}}}^{4}{A_{{0}}}^{2}c_{{0,1}}c_{{1,2}}{c_{{1
,1}}}^{2}-48\,A_{{3}}{A_{{0}}}^{5}c_{{1,2}}{c_{{2,3}}}^{2}{c_{{1,1}}}^
{3}+48\,B_{{2}}A_{{2}}{A_{{0}}}^{3}c_{{1,2}}{c_{{2,3}}}^{3}{c_{{1,1}}}
^{3}-21\,{A_{{0}}}^{6}{c_{{0,1}}}^{3}{c_{{1,2}}}^{3}{c_{{2,3}}}^{2}-12
\,{A_{{0}}}^{6}c_{{1,2}}c_{{0,1}}c_{{1,4}}c_{{2,3}}{c_{{1,1}}}^{2}-6\,
{A_{{0}}}^{6}{c_{{1,2}}}^{2}{c_{{0,1}}}^{2}c_{{1,3}}c_{{2,3}}c_{{1,1}}
+30\,{A_{{0}}}^{6}c_{{1,2}}{c_{{0,1}}}^{2}c_{{1,3}}{c_{{2,3}}}^{2}c_{{
1,1}}+40\,{A_{{0}}}^{6}c_{{2,4}}c_{{1,2}}c_{{0,1}}c_{{1,3}}{c_{{1,1}}}
^{2}+30\,B_{{2}}{A_{{0}}}^{4}{c_{{1,2}}}^{3}{c_{{0,1}}}^{2}{c_{{2,3}}}
^{2}c_{{1,1}}+24\,B_{{2}}{A_{{0}}}^{4}c_{{1,4}}{c_{{2,3}}}^{2}{c_{{1,1
}}}^{3}+12\,B_{{2}}{A_{{0}}}^{4}c_{{0,1}}c_{{1,3}}{c_{{2,3}}}^{3}{c_{{
1,1}}}^{2}-60\,B_{{2}}{A_{{0}}}^{4}c_{{1,2}}c_{{0,1}}c_{{1,3}}{c_{{2,3
}}}^{2}{c_{{1,1}}}^{2}-6\,B_{{2}}{A_{{0}}}^{4}{c_{{1,2}}}^{2}{c_{{0,1}
}}^{2}{c_{{2,3}}}^{3}c_{{1,1}}) .$
\end{proof}

\end{pr}

\begin{pr}\emph{}
\begin{enumerate}
\item Two algebras $L(C)$ and $L(C')$ from $U_{8}^{12}$
are isomorphic, if and only if\\
 $$ {\frac
{c'_{{1,2}}}{c'_{{2,3}}}}={\frac {c_{{1,2}}}{c_{{2,3}}}},\ \
{\frac { \left(
-c'_{{1,3}}c'_{{0,1}}+2\,{c'}_{{1,2}}^{2}c'_{{0,0}}
 \right) ^{5}}{{c'}_{{2,3}}^{5}{c'}_{{0,1}}^{6}}}={\frac { \left( -c_{{1,3}}c_{{0,1}}+2\,c_{{1,2}}^{2}c_{{0,0}}
 \right) ^{5}}{c_{{2,3}}^{5}c_{{0,1}}^{6}}},$$ $$ {\frac { \left( c'_{{2,4}}c'_{{0,1}}-3\,c'_{{2,3}}c'_{{0,0}}c'_{{1,2}}+3\,{c'}_{{2,3}}^{2}c'_{{0,0}} \right)
^{5}}{{c'}_{{2,3}}^{5}{c'}_{{0,1}}^{6}}}={\frac { \left(
c_{{2,4}}c_{{0,1}}-3\,c_{{2,3}}c_{{0,0}}c_{{1,2}}+3\,c_{{2,3}}^{2}c_{{0,0}}
\right) ^{5}}{c_{{2,3}}^{5}c_{{0,1}}^{6}}}.
 $$
 \item For any $\lambda_1, \lambda_2, \lambda_3 \in \mathbb{C},  $  there exists $L(C)\in
U_{8}^{12}:$  $$  {\frac {c_{{1,2}}}{c_{{2,3}}}}=\lambda_1, \ \
{\frac { \left( -c_{{1,3}}c_{{0,1}}+2\,c_{{1,2}}^{2}c_{{0,0}}
 \right) ^{5}}{c_{{2,3}}^{5}c_{{0,1}}^{6}}}=\lambda_2, \ \ {\frac { \left(
c_{{2,4}}c_{{0,1}}-3\,c_{{2,3}}c_{{0,0}}c_{{1,2}}+3\,c_{{2,3}}^{2}c_{{0,0}}
\right) ^{5}}{c_{{2,3}}^{5}c_{{0,1}}^{6}}}=\lambda_3.$$

Then orbits in $U_{8}^{12}$ can be parameterized as $L\left(0, 1,
0, \lambda_1, \lambda_2,0,0,1, \lambda_3, 0\right), \ \lambda_1,
\lambda_2, \lambda_3 \in \mathbb{C}.$
\end{enumerate}
\begin{proof}Here, $\ds A_1=-{\frac {A_{{0}}c_{{0,0}}}{c_{{0,1}}}},$
$\ds B_1={\frac {{A_{{0}}}^{2}}{c_{{2,3}}}},$ $\ds
B_3=\frac{1}{2\,{A_{{0}}}^{2}{c_{{2,3}}}^{2}{c_{{0,1}}}^{2}}
({B_{{2}}}^{2}{c_{{2,3}}}^{3}{c_{{0,1}}}^{2}+{A_{{0}}}^{4}
c_{{1,4}}{c_{{0,1}}}^{2}+{A_{{0}}}^{4}c_{{0,0}}c_{{1,3}}c_{{2,3}}c_{{0
,1}}-5\,{A_{{0}}}^{4}c_{{1,2}}c_{{0,0}}c_{{1,3}}c_{{0,1}}+5\,{A_{{0}}}
^{4}{c_{{1,2}}}^{3}{c_{{0,0}}}^{2}-{A_{{0}}}^{4}{c_{{1,2}}}^{2}{c_{{0,0
}}}^{2}c_{{2,3}})
 $ and \\ $ \ds
B_4=\frac{1}{6\,{A_{{0}}}^{4}{c_{{2,3}}}^{3}{c_{{0,1}}}^{3}}
({B_{{2}}}^{3}{c_{{2,3}}}^{5}{c_{{0,1}}}^{3}-10\,{A_{{0}}}
^{6}c_{{2,4}}{c_{{1,2}}}^{3}{c_{{0,0}}}^{2}c_{{0,1}}+2\,{A_{{0}}}^{6}c
_{{1,5}}c_{{2,3}}{c_{{0,1}}}^{3}-2\,{A_{{0}}}^{6}c_{{2,4}}c_{{1,4}}{c_
{{0,1}}}^{3}-6\,{A_{{0}}}^{6}c_{{0,0}}{c_{{1,3}}}^{2}c_{{2,3}}{c_{{0,1
}}}^{2}+17\,{A_{{0}}}^{6}{c_{{0,0}}}^{3}{c_{{1,2}}}^{4}c_{{2,3}}+3\,{B
_{{2}}}^{2}{c_{{2,3}}}^{4}{A_{{0}}}^{2}c_{{0,0}}c_{{1,2}}{c_{{0,1}}}^{
2}-6\,A_{{3}}{A_{{0}}}^{5}c_{{1,2}}{c_{{2,3}}}^{2}{c_{{0,1}}}^{3}+\\ 6\,B
_{{2}}\,A_{{2}}{A_{{0}}}^{3}c_{{1,2}}{c_{{2,3}}}^{3}{c_{{0,1}}}^{3}-21\,
{A_{{0}}}^{6}{c_{{0,0}}}^{3}{c_{{1,2}}}^{3}{c_{{2,3}}}^{2}-3\,{A_{{0}}
}^{6}c_{{1,2}}c_{{0,0}}c_{{1,4}}c_{{2,3}}{c_{{0,1}}}^{2}-3\,{A_{{0}}}^
{6}{c_{{1,2}}}^{2}{c_{{0,0}}}^{2}c_{{1,3}}c_{{2,3}}c_{{0,1}}+15\,{A_{{0
}}}^{6}c_{{1,2}}{c_{{0,0}}}^{2}c_{{1,3}}{c_{{2,3}}}^{2}c_{{0,1}}+10\,{
A_{{0}}}^{6}c_{{2,4}}c_{{1,2}}c_{{0,0}}c_{{1,3}}{c_{{0,1}}}^{2}+15\,B_
{{2}}{A_{{0}}}^{4}{c_{{1,2}}}^{3}{c_{{0,0}}}^{2}{c_{{2,3}}}^{2}c_{{0,1
}}+3\,B_{{2}}{A_{{0}}}^{4}c_{{1,4}}{c_{{2,3}}}^{2}{c_{{0,1}}}^{3}+3\,B
_{{2}}{A_{{0}}}^{4}c_{{0,0}}c_{{1,3}}{c_{{2,3}}}^{3}{c_{{0,1}}}^{2}-15
\,B_{{2}}{A_{{0}}}^{4}c_{{1,2}}c_{{0,0}}c_{{1,3}}{c_{{2,3}}}^{2}{c_{{0
,1}}}^{2}-3\,B_{{2}}{A_{{0}}}^{4}{c_{{1,2}}}^{2}{c_{{0,0}}}^{2}{c_{{2,
3}}}^{3}c_{{0,1}}) .$
\end{proof}

\end{pr}

\begin{pr}\emph{}
\begin{enumerate}
\item Two algebras $L(C)$ and $L(C')$ from $U_{8}^{13}$
are isomorphic, if and only if\\
 $$ \ds {\frac {{c'}_{{2,3}}^{7}{c'}_{{1,2}}^{12}c'_{{0,0}}}{ \left( -3\,c'_{{1,3}
}c'_{{2,3}}c'_{{1,2}}+3\,c'_{{1,3}}{c'}_{{2,3}}^{2}+2\,{c'}_{{1,2}}^{2}c'_{{
2,4}} \right) ^{6}}}={\frac
{c_{{2,3}}^{7}c_{{1,2}}^{12}c_{{0,0}}}{ \left( -3\,c_{{1,3}
}c_{{2,3}}c_{{1,2}}+3\,c_{{1,3}}c_{{2,3}}^{2}+2\,c_{{1,2}}^{2}c_{{
2,4}} \right) ^{6}}}
 ,\ \   {\frac {c'_{{1,2}}}{c'_{{2,3}}}}={\frac {c_{{1,2}}}{c_{{2,3}}}}.
 $$
 \item For any $\lambda_1\in \mathbb{C} \ \rm{and} \ \lambda_2\in \mathbb{C^*} , $  there exists $L(C)\in
U_{8}^{13}:$  $$ \ds {\frac
{c_{{2,3}}^{7}c_{{1,2}}^{12}c_{{0,0}}}{ \left( -3\,c_{{1,3}
}c_{{2,3}}c_{{1,2}}+3\,c_{{1,3}}c_{{2,3}}^{2}+2\,c_{{1,2}}^{2}c_{{
2,4}} \right) ^{6}}}=\lambda_1, \ \ {\frac
{c_{{1,2}}}{c_{{2,3}}}}=\lambda_2.$$

Then orbits in $U_{8}^{13}$ can be parameterized as
$L\left(\lambda_1,0,0,\lambda_2,0,0,0,1,1,0\right), \ \lambda_1\in
\mathbb{C} \ \rm{and} \ \lambda_2\in \mathbb{C^*}.$
\end{enumerate}

\begin{proof}Here, we take $\ds A_{{0}}={\frac {-3\,c_{{1,3}}c_{{2,3}}c_{{1,2}}+3\,c_{{1,3}}{c_{{
2,3}}}^{2}+2\,{c_{{1,2}}}^{2}c_{{2,4}}}{2\,c_{{2,3}}{c_{{1,2}}}^{2}}}
,$ $\ds A_1={\frac {-A_{{0}}c_{{1,3}}}{2\,{c_{{1,2}}}^{2}}},$ $\ds
B_1={\frac {{A_{{0}}}^{2}}{c_{{2,3}}}},$\\ $\ds
B_3=\frac{1}{{A_{{0}}}^{2}{c_{{2,3}}}^{2}{ c_{{1,2}}}^{2}}
(4\,{B_{{2}}}^{2}{c_{{2,3}}}^{3}{c_{{1,2}}}^{2}+4\,{A_{{0}
}}^{4}c_{{1,4}}{c_{{1,2}}}^{2}+{A_{{0}}}^{4}{c_{{1,3}}}^{2}c_{{2,3}}-5
\,{A_{{0}}}^{4}{c_{{1,3}}}^{2}c_{{1,2}})
 ,$ and\\ $
\ds B_4=-\frac{1}{48\,{A_{{0
}}}^{4}{c_{{2,3}}}^{3}{c_{{1,2}}}^{3}}
(-8\,{B_{{2}}}^{3}{c_{{2,3}}}^{5}{c_{{1,2}}}^{3}-20\,{A_
{{0}}}^{6}c_{{2,4}}{c_{{1,3}}}^{2}{c_{{1,2}}}^{2}-16\,{A_{{0}}}^{6}c_{
{1,5}}c_{{2,3}}{c_{{1,2}}}^{3}+16\,{A_{{0}}}^{6}c_{{2,4}}c_{{1,4}}{c_{
{1,2}}}^{3}+13\,{A_{{0}}}^{6}{c_{{1,3}}}^{3}c_{{2,3}}c_{{1,2}}-12\,{B_
{{2}}}^{2}{c_{{2,3}}}^{4}{A_{{0}}}^{2}c_{{1,3}}{c_{{1,2}}}^{2}+48\,A_{
{3}}{A_{{0}}}^{5}{c_{{1,2}}}^{4}{c_{{2,3}}}^{2}-48\,B_{{2}}A_{{2}}{A_{
{0}}}^{3}{c_{{1,2}}}^{4}{c_{{2,3}}}^{3}-9\,{A_{{0}}}^{6}{c_{{1,3}}}^{3
}{c_{{2,3}}}^{2}+12\,{A_{{0}}}^{6}c_{{1,3}}c_{{1,4}}c_{{2,3}}{c_{{1,2}
}}^{2}+30\,B_{{2}}{A_{{0}}}^{4}{c_{{1,3}}}^{2}{c_{{2,3}}}^{2}{c_{{1,2}
}}^{2}-24\,B_{{2}}{A_{{0}}}^{4}c_{{1,4}}{c_{{2,3}}}^{2}{c_{{1,2}}}^{3}
-6\,B_{{2}}{A_{{0}}}^{4}{c_{{1,3}}}^{2}{c_{{2,3}}}^{3}c_{{1,2}}).$
\end{proof}

\end{pr}

\begin{pr}\emph{}
\begin{enumerate}
\item Two algebras $L(C)$ and $L(C')$ from $U_{8}^{14}$ are
isomorphic, if and only if\ \
 $ \ds  {\frac {c'_{{1,2}}}{c'_{{2,3}}}}={\frac {c_{{1,2}}}{c_{{2,3}}}}.
 $
 \item For any $\lambda\in \mathbb{C^*}, $  there exists $L(C)\in
U_{8}^{14}:$ \ \  $ \ds{\frac {c_{{1,2}}}{c_{{2,3}}}}=\lambda.$

Then orbits in $U_{8}^{14}$ can be parameterized as
$L\left(1,0,0,\lambda,0,0,0,1,0,0\right) , \ \lambda\in
\mathbb{C^*}.$
\end{enumerate}
\begin{proof}Here,
 $\ds A_1=-{\frac {A_{{0}}c_{{1,3}}}{2\,{c_{{1,2}}}^{2}}},$ $\ds
B_1={\frac {{A_{{0}}}^{2}}{c_{{2,3}}}},$ $\ds B_3={\frac
{4\,{B_{{2}}}^{2}{c_{{2,3}}}^{3}{c_{{1,2}}}^{2}+4\,{A_{{0}
}}^{4}c_{{1,4}}{c_{{1,2}}}^{2}+{A_{{0}}}^{4}{c_{{1,3}}}^{2}c_{{2,3}}-5
\,{A_{{0}}}^{4}{c_{{1,3}}}^{2}c_{{1,2}}}{8\,{A_{{0}}}^{2}{c_{{2,3}}}^{2}{
c_{{1,2}}}^{2}}}
 ,$ and $
\ds B_4=\frac{-1}{48\,{A_{{0
}}}^{4}{c_{{2,3}}}^{3}{c_{{1,2}}}^{3}}(-8\,{B_{{2}}}^{3}{c_{{2,3}}}^{5}{c_{{1,2}}}^{3}-20\,{A_
{{0}}}^{6}c_{{2,4}}{c_{{1,3}}}^{2}{c_{{1,2}}}^{2}-16\,{A_{{0}}}^{6}c_{
{1,5}}c_{{2,3}}{c_{{1,2}}}^{3}+16\,{A_{{0}}}^{6}c_{{2,4}}c_{{1,4}}{c_{
{1,2}}}^{3}+13\,{A_{{0}}}^{6}{c_{{1,3}}}^{3}c_{{2,3}}c_{{1,2}}-12\,{B_
{{2}}}^{2}{c_{{2,3}}}^{4}{A_{{0}}}^{2}c_{{1,3}}{c_{{1,2}}}^{2}+48\,A_{
{3}}{A_{{0}}}^{5}{c_{{1,2}}}^{4}{c_{{2,3}}}^{2}-48\,B_{{2}}A_{{2}}{A_{
{0}}}^{3}{c_{{1,2}}}^{4}{c_{{2,3}}}^{3}-9\,{A_{{0}}}^{6}{c_{{1,3}}}^{3
}{c_{{2,3}}}^{2}+12\,{A_{{0}}}^{6}c_{{1,3}}c_{{1,4}}c_{{2,3}}{c_{{1,2}
}}^{2}+30\,B_{{2}}{A_{{0}}}^{4}{c_{{1,3}}}^{2}{c_{{2,3}}}^{2}{c_{{1,2}
}}^{2}-24\,B_{{2}}{A_{{0}}}^{4}c_{{1,4}}{c_{{2,3}}}^{2}{c_{{1,2}}}^{3}
-6\,B_{{2}}{A_{{0}}}^{4}{c_{{1,3}}}^{2}{c_{{2,3}}}^{3}c_{{1,2}})
.$
\end{proof}
\end{pr}

\begin{pr}\emph{}
\begin{enumerate}
\item Two algebras $L(C)$ and $L(C')$ from $ U_{8}^{15}$ are
isomorphic, if and only if\ \
 $ \ds  {\frac {c'_{{1,2}}}{c'_{{2,3}}}}={\frac {c_{{1,2}}}{c_{{2,3}}}}.
 $
 \item For any $\lambda\in \mathbb{C^*}, $  there exists $L(C)\in
 U_{8}^{15}:$ \ \  $ \ds{\frac
{c_{{1,2}}}{c_{{2,3}}}}=\lambda.$

Then orbits in $ U_{8}^{15}$ can be parameterized as
$L\left(0,0,0,\lambda,0,0,0,1,0,0\right),\ \lambda\in
\mathbb{C^*}.$
\end{enumerate}
\begin{proof}Put, $\ds A_1=-{\frac {A_{{0}}c_{{1,3}}}{2\,{c_{{1,2}}}^{2}}},$ $\ds
B_1={\frac {{A_{{0}}}^{2}}{c_{{2,3}}}},$ $\ds B_3={\frac
{4\,{B_{{2}}}^{2}{c_{{2,3}}}^{3}{c_{{1,2}}}^{2}+4\,{A_{{0}
}}^{4}c_{{1,4}}{c_{{1,2}}}^{2}+{A_{{0}}}^{4}{c_{{1,3}}}^{2}c_{{2,3}}-5
\,{A_{{0}}}^{4}{c_{{1,3}}}^{2}c_{{1,2}}}{8\,{A_{{0}}}^{2}{c_{{2,3}}}^{2}{
c_{{1,2}}}^{2}}}
 ,$ and $
\ds B_4=\frac{-1}{48\,{A_{{0
}}}^{4}{c_{{2,3}}}^{3}{c_{{1,2}}}^{3}}(-8\,{B_{{2}}}^{3}{c_{{2,3}}}^{5}{c_{{1,2}}}^{3}-20\,{A_
{{0}}}^{6}c_{{2,4}}{c_{{1,3}}}^{2}{c_{{1,2}}}^{2}-16\,{A_{{0}}}^{6}c_{
{1,5}}c_{{2,3}}{c_{{1,2}}}^{3}+16\,{A_{{0}}}^{6}c_{{2,4}}c_{{1,4}}{c_{
{1,2}}}^{3}+13\,{A_{{0}}}^{6}{c_{{1,3}}}^{3}c_{{2,3}}c_{{1,2}}-12\,{B_
{{2}}}^{2}{c_{{2,3}}}^{4}{A_{{0}}}^{2}c_{{1,3}}{c_{{1,2}}}^{2}+48\,A_{
{3}}{A_{{0}}}^{5}{c_{{1,2}}}^{4}{c_{{2,3}}}^{2}-48\,B_{{2}}A_{{2}}{A_{
{0}}}^{3}{c_{{1,2}}}^{4}{c_{{2,3}}}^{3}-9\,{A_{{0}}}^{6}{c_{{1,3}}}^{3
}{c_{{2,3}}}^{2}+12\,{A_{{0}}}^{6}c_{{1,3}}c_{{1,4}}c_{{2,3}}{c_{{1,2}
}}^{2}+30\,B_{{2}}{A_{{0}}}^{4}{c_{{1,3}}}^{2}{c_{{2,3}}}^{2}{c_{{1,2}
}}^{2}-24\,B_{{2}}{A_{{0}}}^{4}c_{{1,4}}{c_{{2,3}}}^{2}{c_{{1,2}}}^{3}
-6\,B_{{2}}{A_{{0}}}^{4}{c_{{1,3}}}^{2}{c_{{2,3}}}^{3}c_{{1,2}})
.$
\end{proof}
\end{pr}

\begin{pr}\emph{}
\begin{enumerate}
\item Two algebras $L(C)$ and $L(C')$ from $U_{8}^{16}$ are
isomorphic, if and only if $ \ds {\frac
{{c'_{{2,3}}}^{7}c'_{{0,0}}}{{c'_{{1,3}}}^{6}}}={\frac
{{c_{{2,3}}}^{7}c_{{0,0}}}{{c_{{1,3}}}^{6}}}.
 $
 \item For any $\lambda_1\in \mathbb{C}, $  there exists $L(C)\in
U_{8}^{16}:$  $ \ds {\frac
{{c_{{2,3}}}^{7}c_{{0,0}}}{{c_{{1,3}}}^{6}}}=\lambda_1.$

Then orbits in $U_{8}^{16}$ can be parameterized as
$L\left(\lambda_1,0,0,0,1,0,0,1,0,0\right), \ \lambda_1\in
\mathbb{C}.$
\end{enumerate}

\begin{proof}For this case we put
 $\ds A_{{0}}={\frac {c_{{1,3}}}{c_{{2,3}}}}
,$ $\ds A_1={\frac {c_{{1,3}}c_{{2,4}}}{3\,{c_{{2,3}}}^{3}}},$
$\ds B_1={\frac {{c_{{1,3}}}^{2}}{{c_{{2,3}}}^{3}}},$ $\ds
B_3={\frac
{3\,{B_{{2}}}^{2}{c_{{2,3}}}^{8}+3\,{c_{{1,3}}}^{4}c_{{1,4
}}c_{{2,3}}-{c_{{1,3}}}^{5}c_{{2,4}}}{6\,{c_{{1,3}}}^{2}{c_{{2,3}}}^{5}}}
 ,$ and $
\ds B_4=\frac{1}{6\,{c_{{1,3}}}^{4}{c_{{2, 3}}}^{6}}
({B_{{2}}}^{3}{c_{{2,3}}}^{12}+2\,{c_{{1,3}}}^{6}c_{{1,5}}
{c_{{2,3}}}^{2}-2\,{c_{{1,3}}}^{6}c_{{2,4}}c_{{1,4}}c_{{2,3}}+2\,{c_{{
1,3}}}^{8}c_{{2,4}}+3\,B_{{2}}{c_{{1,3}}}^{4}c_{{1,4}}{c_{{2,3}}}^{5}-
B_{{2}}{c_{{1,3}}}^{5}c_{{2,4}}{c_{{2,3}}}^{4}) .$
\end{proof}

\end{pr}

\begin{pr}\emph{}
\begin{enumerate}
\item Two algebras $L(C)$ and $L(C')$ from $U_{8}^{19}$
are isomorphic, if and only if\\
 $${\frac { \left(
4\,c'_{{0,0}}{c'}_{{1,2}}^{4}-2\,c'_{{1,3}}c'_{{0,1}}{c'}_{{
1,2}}^{2}+{c'}_{{1,3}}^{2}c'_{{1,1}} \right)
{c'}_{{1,2}}^{3}}{{c'}_{{2,4} }^{6}}} ={\frac { \left(
4\,c_{{0,0}}c_{{1,2}}^{4}-2\,c_{{1,3}}c_{{0,1}}c_{{
1,2}}^{2}+c_{{1,3}}^{2}c_{{1,1}} \right) c_{{1,2}}^{3}}{c_{{2,4}
}^{6}}} ,$$ $$
 {\frac { \left(
-c'_{{1,3}}c'_{{1,1}}+c'_{{0,1}}{c'}_{{1,2}}^{2} \right)
{c'}_{{1,2}}^{3}}{{c'}_{{2,4}}^{5}}} ={\frac { \left(
-c_{{1,3}}c_{{1,1}}+c_{{0,1}}c_{{1,2}}^{2} \right)
c_{{1,2}}^{3}}{c_{{2,4}}^{5}}} ,\ \ {\frac
{{c'}_{{1,2}}^{3}c'_{{1,1}}}{{c'}_{{2,4}}^{4}}}={\frac
{c_{{1,2}}^{3}c_{{1,1}}}{c_{{2,4}}^{4}}},$$ $$  {\frac
{4\,c'_{{1,4}}c'_{{1,2}}-5\,{c'}_{{1,3}}^{2}}{{c'}_{{2,4}}^{2}}}={\frac
{4\,c_{{1,4}}c_{{1,2}}-5\,c_{{1,3}}^{2}}{c_{{2,4}}^{2}}} .
 $$
 \item For any $\lambda_1, \lambda_2, \lambda_3, \lambda_4 \in \mathbb{C}, $  there exists $L(C)\in
U_{8}^{19}:$ \ $ \ds {\frac { \left(
4\,c_{{0,0}}c_{{1,2}}^{4}-2\,c_{{1,3}}c_{{0,1}}c_{{
1,2}}^{2}+c_{{1,3}}^{2}c_{{1,1}} \right) c_{{1,2}}^{3}}{c_{{2,4}
}^{6}}} =\lambda_1,$ $$ {\frac { \left(
-c_{{1,3}}c_{{1,1}}+c_{{0,1}}c_{{1,2}}^{2} \right)
c_{{1,2}}^{3}}{c_{{2,4}}^{5}}}=\lambda_2, \ \ {\frac
{c_{{1,2}}^{3}c_{{1,1}}}{c_{{2,4}}^{4}}}=\lambda_3,\ \ {\frac
{4\,c_{{1,4}}c_{{1,2}}-5\,c_{{1,3}}^{2}}{c_{{2,4}}^{2}}}=\lambda_4.$$

Then orbits in $U_{8}^{19}$ can be parameterized as
$L\left(\lambda_1, \lambda_2, \lambda_3,1,0,\lambda_4,
0,0,1,0\right),\ \ \lambda_1, \lambda_2, \lambda_3, \lambda_4 \in
\mathbb{C}.$
\end{enumerate}
\begin{proof}Put
 $\ds A_0={\frac {c_{{2,4}}}{c_{{1,2}}}},$ $\ds A_1=-{\frac {c_{{1,3}}c_{{2,4}}}{2\,{c_{{1,2}}}^{3}}},$ $\ds
B_1={\frac {{c_{{2,4}}}^{2}}{{c_{{1,2}}}^{3}}},$  and $$
B_3=\frac{1}{8\,{c_{{1,2} }}^{5}{c_{{2,4}}}^{2}}
(8\,{c_{{2,4}}}^{3}{c_{{1,3}}}^{3}+4\,{B_{{2}}}^{2}{c_{{1,
2}}}^{8}-12\,{c_{{2,4}}}^{3}c_{{1,3}}c_{{1,4}}c_{{1,2}}+4\,{c_{{2,4}}}
^{3}c_{{1,5}}{c_{{1,2}}}^{2}+{c_{{2,4}}}^{4}{c_{{1,3}}}^{2}).$$
\end{proof}
\end{pr}

\newpage

\begin{pr}\emph{}
\begin{enumerate}
\item Two algebras $L(C)$ and $L(C')$ from $U_{8}^{20}$
are isomorphic, if and only if
 $$ {\frac { \left( 4\,c'_{{0,0}}c'_{{1,1}}-{c'}_{{0,1}}^{2} \right) ^{3}{c'}_{
{2,4}}^{10}}{{c'}_{{1,1}}^{10}}}= {\frac { \left(
4\,c_{{0,0}}c_{{1,1}}-c_{{0,1}}^{2} \right) ^{3}c_{
{2,4}}^{10}}{c_{{1,1}}^{10}}}
 ,\ \
 {\frac {c'_{{1,3}}}{c'_{{2,4}}}}={\frac {c_{{1,3}}}{c_{{2,4}}}},\ \ {\frac {{c'}_{{1,4}}^{3}}{{c'}_{{2,4}}^{2}c'_{{1,1}}}}
 ={\frac {c_{{1,4}}^{3}}{c_{{2,4}}^{2}c_{{1,1}}}}.$$
 \item For any $\lambda_1, \lambda_2, \lambda_3\in \mathbb{C}, $  there exists $L(C)\in
U_{8}^{20}:$  $${\frac { \left(
4\,c_{{0,0}}c_{{1,1}}-c_{{0,1}}^{2} \right) ^{3}c_{
{2,4}}^{10}}{c_{{1,1}}^{10}}}=\lambda_1, \ \ {\frac
{c_{{1,3}}}{c_{{2,4}}}}=\lambda_2, \ \ {\frac
{c_{{1,4}}^{3}}{c_{{2,4}}^{2}c_{{1,1}}}}=\lambda_3.$$

Then orbits in $U_{8}^{20}$ can be parameterized as
$L\left(\lambda_1,0,1,0 ,\lambda_2, \lambda_3,0,0,1, 0\right),$ $
\lambda_1, \lambda_2, \lambda_3 \in \mathbb{C}.$
\end{enumerate}
\begin{proof}Here,
 $$ A_1=-{\frac {A_{{0}}c_{{0,1}}}{2\,c_{{1,1}}}},\ \
B_1={\frac {{A_{{0}}}^{3}}{c_{{2,4}}}},\ {\rm{and}}\  B_3={\frac
{2\,{B_{{2}}}^{2}{c_{{2,4}}}^{3}c_{{1,1}}-3\,{A_{{0}}}^{
6}c_{{0,1}}{c_{{1,3}}}^{2}+2\,{A_{{0}}}^{6}c_{{1,5}}c_{{1,1}}+{A_{{0}}
}^{6}c_{{0,1}}c_{{1,3}}c_{{2,4}}}{4\,{A_{{0}}}^{3}{c_{{2,4}}}^{2}c_{{1,1}
}}} .$$
\end{proof}
\end{pr}

\begin{pr}\emph{}
\begin{enumerate}
\item Two algebras $L(C)$ and $L(C')$ from $U_{8}^{21}$
are isomorphic, if and only if \\
 $$
 {\frac {c'_{{1,3}}}{c'_{{2,4}}}}={\frac {c_{{1,3}}}{c_{{2,4}}}},\ \ {\frac {{c'}_{{1,4}}^{5}}{{c'}_{{2,4}}^{5}c'_{{0,1}}}}
 ={\frac {c_{{1,4}}^{5}}{c_{{2,4}}^{5}c_{{0,1}}}}.
 $$
 \item For any $\lambda_1, \lambda_2\in \mathbb{C} ,$  there exists $L(C)\in
U_{8}^{21}:$ \ $ \ds {\frac {c_{{1,3}}}{c_{{2,4}}}}=\lambda_1, \ \
{\frac {c_{{1,4}}^{5}}{c_{{2,4}}^{5}c_{{0,1}}}}=\lambda_2.$

Then orbits in $U_{8}^{21}$ can be parameterized as
$L\left(0,0,1,0 ,0,\lambda_1, \lambda_2,0,0,1, 0\right),$ $
\lambda_1, \lambda_2\in \mathbb{C}.$
\end{enumerate}
\begin{proof}We put,
 $$ A_1=-{\frac {A_{{0}}c_{{0,0}}}{c_{{0,1}}}},\ \
B_1={\frac {{A_{{0}}}^{3}}{c_{{2,4}}}},\ {\rm{and}}\ B_3={\frac
{{B_{{2}}}^{2}{c_{{2,4}}}^{3}c_{{0,1}}-3\,{A_{{0}}}^{6}c
_{{0,0}}{c_{{1,3}}}^{2}+{A_{{0}}}^{6}c_{{1,5}}c_{{0,1}}+{A_{{0}}}^{6}c
_{{0,0}}c_{{1,3}}c_{{2,4}}}{2\,{A_{{0}}}^{3}{c_{{2,4}}}^{2}c_{{0,1}}}}
.$$
\end{proof}
\end{pr}

\begin{pr}\emph{}
\begin{enumerate}
\item Two algebras $L(C)$ and $L(C')$ from $U_{8}^{22}$
are isomorphic, if and only if
 $$ {\frac {{c'}_{{2,4}}^{8}c'_{{0,0}}}{{c'}_{{1,4}}^{7}}}= {\frac {c_{{2,4}}^{8}c_{{0,0}}}{c_{{1,4}}^{7}}},\ \
 {\frac {c'_{{1,3}}}{c'_{{2,4}}}}={\frac {c_{{1,3}}}{c_{{2,4}}}}.
 $$
 \item For any $\lambda_1, \lambda_2\in \mathbb{C} $  there exists $L(C)\in
U_{8}^{22}:$ \ $ \ds {\frac
{c_{{2,4}}^{8}c_{{0,0}}}{c_{{1,4}}^{7}}} =\lambda_1, \ {\frac
{c_{{1,3}}}{c_{{2,4}}}}=\lambda_2.$

Then orbits in $U_{8}^{22}$ can be parameterized as
$L\left(\lambda_1,0,0,0, \lambda_2,1,0,0,1, 0\right),$ $
\lambda_1, \lambda_2\in \mathbb{C}.$
\end{enumerate}
\begin{proof}We put here,
 $$ A_0={\frac {c_{{1,4}}}{c_{{2,4}}}},\ \
B_1={\frac {{c_{{1,4}}}^{3}}{{c_{{2,4}}}^{4}}},\ {\rm{and}}\
B_3={\frac
{{B_{{2}}}^{2}{c_{{2,4}}}^{9}+3\,{c_{{1,4}}}^{5}A_{{1}}{c_
{{1,3}}}^{2}c_{{2,4}}+{c_{{1,4}}}^{6}c_{{1,5}}-{c_{{1,4}}}^{5}A_{{1}}c
_{{1,3}}{c_{{2,4}}}^{2}}{2\,{c_{{2,4}}}^{5}{c_{{1,4}}}^{3}}} .$$
\end{proof}
\end{pr}

\begin{pr}\emph{}
\begin{enumerate}
\item Two algebras $L(C)$ and $L(C')$ from $U_{8}^{23}$ are
isomorphic, if and only if\ \
 $ \ds   {\frac {c'_{{1,3}}}{c'_{{2,4}}}}={\frac {c_{{1,3}}}{c_{{2,4}}}}.
 $
 \item For any $\lambda\in \mathbb{C^*} , $  there exists $L(C)\in
U_{8}^{23}:$ \ \  $ \ds{\frac {c_{{1,3}}}{c_{{2,4}}}}=\lambda.$

Then orbits in $U_{8}^{23}$ can be parameterized as
$L\left(1,0,0,0,\lambda,0,0,0,1,0\right),$ $\lambda\in
\mathbb{C^*}.$
\end{enumerate}
\begin{proof}Put $$\ds B_1={\frac {{A_{{0}}}^{3}}{c_{{2,4}}}},\ {\rm{and}}\
B_3=\frac
{3\,{A_{{0}}}^{5}A_{{1}}{c_{{1,3}}}^{2}+{B_{{2}}}^{2}{c_{{
2,4}}}^{3}+{A_{{0}}}^{6}c_{{1,5}}-{A_{{0}}}^{5}A_{{1}}c_{{1,3}}c_{{2,4
}}}{2\,{A_{{0}}}^{3}{c_{{2,4}}}^{2}}.$$
\end{proof}
\end{pr}

\begin{pr}\emph{}
\begin{enumerate}
\item Two algebras $L(C)$ and $L(C')$ from $ U_{8}^{24}$ are
isomorphic, if and only if\ \
 $ \ds   {\frac {c'_{{1,3}}}{c'_{{2,4}}}}={\frac {c_{{1,3}}}{c_{{2,4}}}}.
 $
 \item For any $\lambda\in \mathbb{C^*} , $  there exists $L(C)\in
 U_{8}^{24}:$ \ \  $ \ds{\frac
{c_{{1,3}}}{c_{{2,4}}}}=\lambda.$

Then orbits in $U_{8}^{24}$ can be parameterized as
$L\left(0,0,0,0,\lambda,0,0,0,1,0\right),  \ \lambda\in
\mathbb{C^*}.$
\end{enumerate}
\begin{proof}Here, $$\ds B_1={\frac {{A_{{0}}}^{3}}{c_{{2,4}}}},\ {\rm{and}}\ \ds
B_3=\frac
{3\,{A_{{0}}}^{5}A_{{1}}{c_{{1,3}}}^{2}+{B_{{2}}}^{2}{c_{{
2,4}}}^{3}+{A_{{0}}}^{6}c_{{1,5}}-{A_{{0}}}^{5}A_{{1}}c_{{1,3}}c_{{2,4
}}}{2\,{A_{{0}}}^{3}{c_{{2,4}}}^{2}}.$$
\end{proof}
\end{pr}
\begin{pr}\emph{}
\begin{enumerate}
\item Two algebras $L(C)$ and $L(C')$ from $U_{8}^{25}$
are isomorphic, if and only if
 $$ {\frac { \left( 4\,c'_{{0,0}}{c'}_{{1,2}}^{4}-2\,c'_{{1,3}}c'_{{0,1}}
{c'}_{{1,2}}^{2}+{c'}_{{1,3}}^{2}{c'}_{{1,1}} \right)
{c'}_{{1,2}}^{3}}{
 \left( 4\,c'_{{1,4}}c'_{{1,2}}-5\,{c'}_{{1,3}}^{2} \right) ^{3}}}={\frac { \left( 4\,c_{{0,0}}c_{{1,2}}^{4}-2\,c_{{1,3}}c_{{0,1}}
c_{{1,2}}^{2}+c_{{1,3}}^{2}c_{{1,1}} \right) c_{{1,2}}^{3}}{
 \left( 4\,c_{{1,4}}c_{{1,2}}-5\,c_{{1,3}}^{2} \right) ^{3}}}
 ,$$ $$
 {\frac { \left( -c'_{{1,3}}c'_{{1,1}}+c'_{{0,1}}{c'}_{{1,2}}^{2} \right) ^
{2}{c'}_{{1,2}}^{6}}{ \left(
4\,c'_{{1,4}}c'_{{1,2}}-5\,{c'}_{{1,3}}^{2}
 \right) ^{5}}}={\frac { \left( -c_{{1,3}}c_{{1,1}}+c_{{0,1}}c_{{1,2}}^{2} \right) ^
{2}c_{{1,2}}^{6}}{ \left( 4\,c_{{1,4}}c_{{1,2}}-5\,c_{{1,3}}^{2}
 \right) ^{5}}}
,$$ $$ {\frac {{c'}_{{1,2}}^{3}c'_{{1,1}}}{ \left(
4\,c'_{{1,4}}c'_{{1,2}}-5 \,{c'}_{{1,3}}^{2} \right) ^{2}}}
={\frac {c_{{1,2}}^{3}c_{{1,1}}}{ \left(
4\,c_{{1,4}}c_{{1,2}}-5 \,c_{{1,3}}^{2} \right) ^{2}}} ,$$ $$
{\frac { \left(
2\,{c'}_{{1,3}}^{3}+c'_{{1,5}}{c'}_{{1,2}}^{2}-3\,c'_{
{1,3}}c'_{{1,4}}c'_{{1,2}} \right) ^{2}}{ \left(
4\,c'_{{1,4}}c'_{{1,2}}-5 \,{c'}_{{1,3}}^{2} \right) ^{3}}}={\frac
{ \left( 2\,c_{{1,3}}^{3}+c_{{1,5}}c_{{1,2}}^{2}-3\,c_{
{1,3}}c_{{1,4}}c_{{1,2}} \right) ^{2}}{ \left(
4\,c_{{1,4}}c_{{1,2}}-5 \,c_{{1,3}}^{2} \right) ^{3}}}
 .
 $$
 \item For any $\lambda_1, \lambda_2, \lambda_3, \lambda_4 \in \mathbb{C},$  there exists $L(C)\in
U_{8}^{25}:$  $$   {\frac { \left(
4\,c_{{0,0}}c_{{1,2}}^{4}-2\,c_{{1,3}}c_{{0,1}}
c_{{1,2}}^{2}+c_{{1,3}}^{2}c_{{1,1}} \right) c_{{1,2}}^{3}}{
 \left( 4\,c_{{1,4}}c_{{1,2}}-5\,c_{{1,3}}^{2} \right) ^{3}}} =\lambda_1, \ \  {\frac { \left( -c_{{1,3}}c_{{1,1}}+c_{{0,1}}c_{{1,2}}^{2} \right) ^
{2}c_{{1,2}}^{6}}{ \left( 4\,c_{{1,4}}c_{{1,2}}-5\,c_{{1,3}}^{2}
 \right) ^{5}}}=\lambda_2,$$ $$ {\frac {c_{{1,2}}^{3}c_{{1,1}}}{ \left(
4\,c_{{1,4}}c_{{1,2}}-5 \,c_{{1,3}}^{2} \right) ^{2}}}
=\lambda_3,\ \ \ \ \ \ \ \ \ \ \ \ \ \ \ \ {\frac { \left(
2\,c_{{1,3}}^{3}+c_{{1,5}}c_{{1,2}}^{2}-3\,c_{
{1,3}}c_{{1,4}}c_{{1,2}} \right) ^{2}}{ \left(
4\,c_{{1,4}}c_{{1,2}}-5 \,c_{{1,3}}^{2} \right) ^{3}}}=\lambda_4.$$

Then orbits in $U_{8}^{25}$ can be parameterized as
$L\left(\lambda_1, \lambda_2, \lambda_3,1,0,1,\lambda_4,
0,0,0\right),$ $\lambda_1, \lambda_2, \lambda_3, \lambda_4 \in
\mathbb{C}.$
\end{enumerate}
\begin{proof}Here,
$$ A_1=-{\frac {A_{{0}}c_{{1,3}}}{2\,{c_{{1,2}}}^{2}}},\ {\rm{and}}\
\ds B_1={\frac {{A_{{0}}}^{2}}{c_{{1,2}}}}.$$
\end{proof}
\end{pr}

\begin{pr}\emph{}
\begin{enumerate}
\item Two algebras $L(C)$ and $L(C')$ from $U_{8}^{26}$
are isomorphic, if and only if
 $$ {\frac { \left(
 4\,c'_{{0,0}}{c'}_{{1,2}}^{4}-2\,c'_{{1,3}}c'_{{0,1}}{c'}
_{{1,2}}^{2}+{c'}_{{1,3}}^{2}{c'}_{{1,1}} \right)
{c'}_{{1,2}}^{3}}{
 \left( -7\,{c'}_{{1,3}}^{3}+4\,{c'}_{{1,5}}{c'}_{{1,2}}^{2} \right) ^{2}}}={\frac { \left(
 4\,c_{{0,0}}c_{{1,2}}^{4}-2\,c_{{1,3}}c_{{0,1}}c
_{{1,2}}^{2}+c_{{1,3}}^{2}c_{{1,1}} \right) c_{{1,2}}^{3}}{
 \left( -7\,c_{{1,3}}^{3}+4\,c_{{1,5}}c_{{1,2}}^{2} \right) ^{2}}}
 ,$$ $$
{\frac { \left( -c'_{{1,3}}c'_{{1,1}}+c'_{{0,1}}{c'}_{{1,2}}^{2}
 \right) ^{3}{c'}_{{1,2}}^{9}}{ \left( -7\,{c'}_{{1,3}}^{3}+4\,c'_{{1,5}}
{c'}_{{1,2}}^{2} \right) ^{5}}} ={\frac { \left(
-c_{{1,3}}c_{{1,1}}+c_{{0,1}}c_{{1,2}}^{2}
 \right) ^{3}c_{{1,2}}^{9}}{ \left( -7\,c_{{1,3}}^{3}+4\,c_{{1,5}}
c_{{1,2}}^{2} \right) ^{5}}} ,$$ $$ {\frac
{{c'}_{{1,2}}^{9}{c'}_{{1,1}}^{3}}{ \left( -7\,{c'}_{{1,3}}^{
3}+4\,c'_{{1,5}}{c'}_{{1,2}}^{2} \right) ^{4}}}={\frac
{c_{{1,2}}^{9}c_{{1,1}}^{3}}{ \left( -7\,c_{{1,3}}^{
3}+4\,c_{{1,5}}c_{{1,2}}^{2} \right) ^{4}}}
 .
 $$
 \item For any $\lambda_1, \lambda_2, \lambda_3 \in \mathbb{C}, $  there exists $L(C)\in
U_{8}^{26}:$ \ $ \ds  {\frac { \left(
4\,c_{{0,0}}c_{{1,2}}^{4}-2\,c_{{1,3}}c_{{0,1}}c
_{{1,2}}^{2}+c_{{1,3}}^{2}c_{{1,1}} \right) c_{{1,2}}^{3}}{
 \left( -7\,c_{{1,3}}^{3}+4\,c_{{1,5}}c_{{1,2}}^{2} \right) ^{2}}} =\lambda_1,$ $$ {\frac { \left( -c_{{1,3}}c_{{1,1}}+c_{{0,1}}c_{{1,2}}^{2}
 \right) ^{3}c_{{1,2}}^{9}}{ \left( -7\,c_{{1,3}}^{3}+4\,c_{{1,5}}
c_{{1,2}}^{2} \right) ^{5}}}=\lambda_2, \ \ {\frac
{c_{{1,2}}^{9}c_{{1,1}}^{3}}{ \left( -7\,c_{{1,3}}^{
3}+4\,c_{{1,5}}c_{{1,2}}^{2} \right) ^{4}}}=\lambda_3.$$

Then orbits in $U_{8}^{26}$ can be parameterized as
$L\left(\lambda_1, \lambda_2, \lambda_3,1,0,0,1, 0,0,0\right),\ \
\lambda_1, \lambda_2, \lambda_3 \in \mathbb{C}.$
\end{enumerate}
\begin{proof}For this case, put in the base change (\ref{bc.dim08}),
$$ A_1=-{\frac {A_{{0}}c_{{1,3}}}{2\,{c_{{1,2}}}^{2}}},\ {\rm{and}}\
\ds B_1={\frac {{A_{{0}}}^{2}}{c_{{1,2}}}}.$$
\end{proof}
\end{pr}

\begin{pr}\emph{}
\begin{enumerate}
\item Two algebras $L(C)$ and $L(C')$ from $U_{8}^{27}$
are isomorphic, if and only if
 $${\frac { \left( 4\,c'_{{0,0}}{c'}_{{1,2}}^{4}-2\,c'_{{1,3}}c'_{{0,1}}{c'}_{{
1,2}}^{2}+{c'}_{{1,3}}^{2}c'_{{1,1}} \right)
^{2}}{{c'}_{{1,2}}^{3}{c'}_{{ 1,1}}^{3}}}={\frac { \left(
4\,c_{{0,0}}c_{{1,2}}^{4}-2\,c_{{1,3}}c_{{0,1}}c_{{
1,2}}^{2}+c_{{1,3}}^{2}c_{{1,1}} \right) ^{2}}{c_{{1,2}}^{3}c_{{
1,1}}^{3}}}
 ,$$ $$
{\frac { \left( -c'_{{1,3}}c'_{{1,1}}+c'_{{0,1}}{c'}_{{1,2}}^{2}
\right) ^ {4}}{{c'}_{{1,2}}^{3}{c'}_{{1,1}}^{5}}}={\frac { \left(
-c_{{1,3}}c_{{1,1}}+c_{{0,1}}c_{{1,2}}^{2} \right) ^
{4}}{c_{{1,2}}^{3}c_{{1,1}}^{5}}}
 .
 $$
 \item For any $\lambda_1, \lambda_2 \in \mathbb{C}, $  there exists $L(C)\in
U_{8}^{27}:$ \ $$ {\frac { \left(
4\,c_{{0,0}}c_{{1,2}}^{4}-2\,c_{{1,3}}c_{{0,1}}c_{{
1,2}}^{2}+c_{{1,3}}^{2}c_{{1,1}} \right) ^{2}}{c_{{1,2}}^{3}c_{{
1,1}}^{3}}} =\lambda_1, \ \ {\frac { \left(
-c_{{1,3}}c_{{1,1}}+c_{{0,1}}c_{{1,2}}^{2} \right) ^
{4}}{c_{{1,2}}^{3}c_{{1,1}}^{5}}}=\lambda_2.$$

Then orbits in $U_{8}^{27}$ can be parameterized as
$L\left(\lambda_1, \lambda_2, 1,1,0,0,0, 0,0,0\right),$ $
\lambda_1, \lambda_2 \in \mathbb{C}.$
\end{enumerate}
\begin{proof}Put in the base change (\ref{bc.dim08}),
$$ A_1=-{\frac {A_{{0}}c_{{1,3}}}{2\,{c_{{1,2}}}^{2}}},\ {\rm{and}}\
\ds B_1={\frac {{A_{{0}}}^{2}}{c_{{1,2}}}}.$$
\end{proof}
\end{pr}

\begin{pr}\emph{}
\begin{enumerate}
\item Two algebras $L(C)$ and $L(C')$ from $U_{8}^{28}$
are isomorphic, if and only if
 $$  {\frac { \left( 2\,c'_{{0,0}}{c'}_{{1,2}}^{2}-c'_{{1,3}}c'_{{0,1}}
 \right) ^{5}}{{c'}_{{1,2}}^{5}{c'}_{{0,1}}^{6}}}={\frac { \left( 2\,c_{{0,0}}c_{{1,2}}^{2}-c_{{1,3}}c_{{0,1}}
 \right) ^{5}}{c_{{1,2}}^{5}c_{{0,1}}^{6}}}.
 $$
 \item For any $\lambda_1 \in \mathbb{C} ,$  there exists $L(C)\in
U_{8}^{28}:$ \ $ \ds {\frac { \left(
2\,c_{{0,0}}c_{{1,2}}^{2}-c_{{1,3}}c_{{0,1}}
 \right) ^{5}}{c_{{1,2}}^{5}c_{{0,1}}^{6}}}=\lambda_1.$

Then orbits in $U_{8}^{28}$ can be parameterized as
$L\left(\lambda_1, 1, 0,1,0,0,0, 0,0,0\right),\ \ \lambda_1 \in
\mathbb{C}.$
\end{enumerate}
\begin{proof}For this case, we put
$$\ds A_1=-{\frac {A_{{0}}c_{{1,3}}}{2\,{c_{{1,2}}}^{2}}},\ {\rm{and}}\
\ds B_1={\frac {{A_{{0}}}^{2}}{c_{{1,2}}}}.$$
\end{proof}
\end{pr}

\begin{pr}\emph{}
\begin{enumerate}
\item Two algebras $L(C)$ and $L(C')$ from $U_{8}^{31}$
are isomorphic, if and only if\\
 $ \ds {\frac {{c'}_{{1,3}}^{4} \left( -9\,c'_{{0,0}}{c'}_{{1,3}}^{4}+3\,c'_{{1,5
}}c'_{{0,1}}{c'}_{{1,3}}^{2}-{c'}_{{1,5}}^{2}c'_{{1,1}} \right)
}{{c'}_{{1,4 }}^{7}}}={\frac {c_{{1,3}}^{4} \left(
-9\,c_{{0,0}}c_{{1,3}}^{4}+3\,c_{{1,5
}}c_{{0,1}}c_{{1,3}}^{2}-c_{{1,5}}^{2}c_{{1,1}} \right) }{c_{{1,4
}}^{7}}}
 ,\\
{\frac {{c'}_{{1,3}}^{3} \left(
-2\,c'_{{1,5}}c'_{{1,1}}+3\,c'_{{0,1}}{c'}_{ {1,3}}^{2} \right)
}{{c'}_{{1,4}}^{5}}} ={\frac {c_{{1,3}}^{3} \left(
-2\,c_{{1,5}}c_{{1,1}}+3\,c_{{0,1}}c_{ {1,3}}^{2} \right)
}{c_{{1,4}}^{5}}} ,\ \ {\frac
{{c'}_{{1,3}}^{2}c'_{{1,1}}}{{c'}_{{1,4}}^{3}}}={\frac
{c_{{1,3}}^{2}c_{{1,1}}}{c_{{1,4}}^{3}}}.
 $
 \item For any $\lambda_1, \lambda_2, \lambda_3 \in \mathbb{C}, $  there exists $L(C)\in
U_{8}^{31}:$ \ $ \ds {\frac {c_{{1,3}}^{4} \left(
-9\,c_{{0,0}}c_{{1,3}}^{4}+3\,c_{{1,5
}}c_{{0,1}}c_{{1,3}}^{2}-c_{{1,5}}^{2}c_{{1,1}} \right) }{c_{{1,4
}}^{7}}}=\lambda_1, \\ {\frac {c_{{1,3}}^{3} \left(
-2\,c_{{1,5}}c_{{1,1}}+3\,c_{{0,1}}c_{ {1,3}}^{2} \right)
}{c_{{1,4}}^{5}}}=\lambda_2, \ \ {\frac
{c_{{1,3}}^{2}c_{{1,1}}}{c_{{1,4}}^{3}}}=\lambda_3.$

Then orbits in $U_{8}^{31}$ can be parameterized as
$L\left(\lambda_1, \lambda_2, \lambda_3,0,1,1,0,0,0,0\right),\ \
\lambda_1, \lambda_2, \lambda_3 \in \mathbb{C}.$
\end{enumerate}
\begin{proof}Here, we put in the base change (\ref{bc.dim08}) the following
coefficients:
$$ A_0={\frac {c_{{1,4}}}{c_{{1,3}}}},\ \  A_1=-{\frac
{c_{{1,4}}c_{{1,5}}}{3\,{c_{{1,3}}}^{3}}},\ {\rm{and}} \
B_1={\frac {{c_{{1,4}}}^{3}}{{c_{{1,3}}}^{4}}}.$$
\end{proof}
\end{pr}

\begin{pr}\emph{}
\begin{enumerate}
\item Two algebras $L(C)$ and $L(C')$ from $U_{8}^{32}$ are
isomorphic, if and only if
 $$ {\frac {\left( -9\,c'_{{0,0}}{c'}_{{1,3}}^{4}+3\,c'_{{1,5
}}c'_{{0,1}}{c'}_{{1,3}}^{2}-{c'}_{{1,5}}^{2}c'_{{1,1}} \right)^3
}{{c'}_{{1,3 }}^{2}{c'}_{{1,1 }}^{7}}}={\frac {\left(
-9\,c_{{0,0}}c_{{1,3}}^{4}+3\,c_{{1,5
}}c_{{0,1}}c_{{1,3}}^{2}-c_{{1,5}}^{2}c_{{1,1}} \right)^3
}{c_{{1,3 }}^{2}c_{{1,1 }}^{7}}}
 ,$$ $$
{\frac {\left( -2\,c'_{{1,5}}c'_{{1,1}}+3\,c'_{{0,1}}{c'}_{
{1,3}}^{2} \right)^3 }{{c'}_{{1,3 }}{c'}_{{1,1 }}^{5}}} ={\frac
{\left( -2\,c_{{1,5}}c_{{1,1}}+3\,c_{{0,1}}c_{ {1,3}}^{2}
\right)^3 }{c_{{1,3 }}c_{{1,1 }}^{5}}}.
 $$
 \item For any $\lambda_1, \lambda_2\in \mathbb{C}, $  there exists $L(C)\in
U_{8}^{32}:$ $$ {\frac {\left(
-9\,c_{{0,0}}c_{{1,3}}^{4}+3\,c_{{1,5
}}c_{{0,1}}c_{{1,3}}^{2}-c_{{1,5}}^{2}c_{{1,1}} \right)^3
}{c_{{1,3 }}^{2}c_{{1,1 }}^{7}}}=\lambda_1, \\ {\frac {\left(
-2\,c_{{1,5}}c_{{1,1}}+3\,c_{{0,1}}c_{ {1,3}}^{2} \right)^3
}{c_{{1,3 }}c_{{1,1 }}^{5}}}=\lambda_2.$$

Then orbits in $U_{8}^{32}$ can be parameterized as
$L\left(\lambda_1, \lambda_2,1,0,1,0,0,0,0,0\right),$ $\lambda_1,
\lambda_2 \in \mathbb{C}.$
\end{enumerate}
\begin{proof} We put $$\ds A_1=-{\frac {A_{{0}}c_{{1,5}}}{3\,{c_{{1,3}}}^{2}}},\ {\rm{and}}\
\ds B_1={\frac {{A_{{0}}}^{3}}{c_{{1,3}}}}.$$
\end{proof}
\end{pr}

\begin{pr}\emph{}
\begin{enumerate}
\item Two algebras $L(C)$ and $L(C')$ from $U_{8}^{33}$ are
isomorphic, if and only if
 $$ {\frac { \left( 3\,c'_{{0,0}}{c'}_{{1,3}}^{2}-c'_{{1,5}}c'_{{0,1}}
 \right) ^{5}}{{c'}_{{1,3}}^{5}{c'}_{{0,1}}^{7}}}={\frac { \left( 3\,c_{{0,0}}c_{{1,3}}^{2}-c_{{1,5}}c_{{0,1}}
 \right) ^{5}}{c_{{1,3}}^{5}c_{{0,1}}^{7}}}.$$
 \item For any $\lambda_1\in \mathbb{C}, $  there exists $L(C)\in
U_{8}^{33}:$ \ $ \ds {\frac { \left(
3\,c_{{0,0}}c_{{1,3}}^{2}-c_{{1,5}}c_{{0,1}}
 \right) ^{5}}{c_{{1,3}}^{5}c_{{0,1}}^{7}}}=\lambda_1.$

Then orbits in $U_{8}^{33}$ can be parameterized as
$L\left(\lambda_1, 1,0,0,1,0,0,0,0,0\right),\ \ \lambda_1 \in
\mathbb{C}.$
\end{enumerate}
\begin{proof}Here,
$$\ds A_1=-{\frac {A_{{0}}c_{{1,5}}}{3\,{c_{{1,3}}}^{2}}},\ {\rm{and}}\
\ds B_1={\frac {{A_{{0}}}^{3}}{c_{{1,3}}}}.$$
\end{proof}
\end{pr}

\begin{pr}\emph{}
\begin{enumerate}
\item Two algebras $L(C)$ and $L(C')$ from $U_{8}^{36}$ are
isomorphic, if and only if
 $$ {\frac {{c'}_{{1,5}}^{10} \left( 4\,{c'}_{{0,0}}c'_{{1,1}}-{c'}_{{0,1}}^{2}
 \right) }{{c'}_{{1,1}}^{10}}}={\frac {c_{{1,5}}^{10} \left( 4\,c_{{0,0}}c_{{1,1}}-c_{{0,1}}^{2}
 \right) }{c_{{1,1}}^{10}}}
 ,\ \
{\frac {c'_{{1,1}}c'_{{1,4}}}{{c'}_{{1,5}}^{2}}}={\frac
{c_{{1,1}}c_{{1,4}}}{c_{{1,5}}^{2}}}.
 $$
 \item For any $\lambda_1, \lambda_2\in \mathbb{C}, $  there exists $L(C)\in
U_{8}^{36}:$ \ $ \ds {\frac {c_{{1,5}}^{10} \left(
4\,c_{{0,0}}c_{{1,1}}-c_{{0,1}}^{2}
 \right) }{c_{{1,1}}^{10}}}=\lambda_1, \ \ {\frac {c_{{1,1}}c_{{1,4}}}{c_{{1,5}}^{2}}}=\lambda_2.$

Then orbits in $U_{8}^{36}$ can be parameterized as
$L\left(\lambda_1,0,1,0,0, \lambda_2,1,0,0,0\right),\ \ \lambda_1,
\lambda_2 \in \mathbb{C}.$
\end{enumerate}
\begin{proof}Her, we take
$$\ds A_0={\frac {c_{{1,1}}}{c_{{1,5}}}},\ A_1=-{\frac
{c_{{0,1}}}{2\,c_{{1,5}}}},\ {\rm{and}}\ B_1={\frac
{{c_{{1,1}}}^{5}}{{c_{{1,5}}}^{6}}}.$$
\end{proof}
\end{pr}

\begin{pr}\emph{}
\begin{enumerate}
\item Two algebras $L(C)$ and $L(C')$ from $U_{8}^{37}$ are
isomorphic, if and only if
 $$ {\frac {c_{{1,4}}^{5} \left( 4\,c_{{0,0}}c_{{1,1}}-c_{{0,1}}^
{2} \right) }{c_{{1,1}}^{5}}}={\frac {c_{{1,4}}^{5} \left(
4\,c_{{0,0}}c_{{1,1}}-c_{{0,1}}^ {2} \right) }{c_{{1,1}}^{5}}} .
 $$
 \item For any $\lambda_1\in \mathbb{C} ,$  there exists $L(C)\in
U_{8}^{37}:$ \ $ \ds {\frac {c_{{1,4}}^{5} \left(
4\,c_{{0,0}}c_{{1,1}}-c_{{0,1}}^ {2} \right)
}{c_{{1,1}}^{5}}}=\lambda_1.$

Then orbits in $U_{8}^{37}$ can be parameterized as
$L\left(\lambda_1,0, 1,0,0,1,0,0,0,0\right),\ \ \lambda_1 \in
\mathbb{C}.$
\end{enumerate}
\begin{proof}We put in the base change
(\ref{bc.dim08}), $$A_0=\sqrt{\ds \frac { {c_{{1,1}}}}{
{c_{{1,4}}}}},\ A_1=-{\frac
{A_{{0}}c_{{0,1}}}{2\,c_{{1,1}}}},\ {\rm{and}}\ \ds B_1={\frac
{{c_{{1,1}}}^{5}}{{c_{{1,5}}}^{6}}}.$$
\end{proof}
\end{pr}

\begin{pr}\emph{}
\begin{enumerate}
\item Two algebras $L(C)$ and $L(C')$ from $U_{8}^{40}$ are
isomorphic, if and only if \ \
 $ \ds {\frac {c_{{0,1}}c_{{1,4}}^{5}}{c_{{1,5}}^{5}}}={\frac {c_{{0,1}}c_{{1,4}}^{5}}{c_{{1,5}}^{5}}} .
 $
 \item For any $\lambda_1\in \mathbb{C} $  there exists $L(C)\in
U_{8}^{40}:$ \ $ \ds{\frac
{c_{{0,1}}c_{{1,4}}^{5}}{c_{{1,5}}^{5}}}=\lambda_1.$

Then orbits in $U_{8}^{40}$ can be parameterized as
$L\left(0,1,0,0,0,\lambda_1,1,0,0,0\right),\ \ \lambda_1 \in
\mathbb{C}.$
\end{enumerate}
\begin{proof}Here,
$$\ds A_1=-{\frac {A_{{0}}c_{{0,0}}}{c_{{0,1}}}},\ {\rm{and}}\ \ds
B_1={\frac {c_{{0,1}}}{c_{{1,5}}}}.$$
\end{proof}
\end{pr}

\begin{pr}\emph{}
\begin{enumerate}
\item Two algebras $L(C)$ and $L(C')$ from $U_{8}^{43}$ are
isomorphic, if and only if \ \
 $ \ds {\frac {c_{{1,4}}^{9}c_{{0,0}}}{c_{{1,5}}^{8}}}=\ds {\frac {c_{{1,4}}^{9}c_{{0,0}}}{c_{{1,5}}^{8}}} .
 $
 \item For any $\lambda_1\in \mathbb{C} $  there exists $L(C)\in
U_{8}^{43}:$ \ $  \ds {\frac
{c_{{1,4}}^{9}c_{{0,0}}}{c_{{1,5}}^{8}}} =\lambda_1.$

Then orbits in $U_{8}^{43}$ can be parameterized as
$L\left(\lambda_1,0,0,0,0,1,1,0,0,0\right),\ \ \lambda_1 \in
\mathbb{C}.$
\end{enumerate}
\begin{proof}Finally, put
$$\ds A_0={\frac {c_{{1,5}}}{c_{{1,4}}}},\ {\rm{and}}\ \ds B_1={\frac
{{c_{{1,5}}}^{4}}{{c_{{1,4}}}^{5}}}.$$
\end{proof}
\end{pr}

\begin{pr}\emph{}

The subsets  $U_{8}^{5},\ U_{8}^{6},\ U_{8}^{7},\ U_{8}^{8},\
U_{8}^{9},\ U_{8}^{10},\ U_{8}^{17},\ U_{8}^{18},\ U_{8}^{29},\
U_{8}^{30},\ U_{8}^{34},\ U_{8}^{35},\ U_{8}^{38},\ U_{8}^{39},\
U_{8}^{41},\\ U_{8}^{42},\ U_{8}^{44},\ U_{8}^{45},\ U_{8}^{46},\
U_{8}^{47},\ U_{8}^{48}$ and $ U_{8}^{49}$\ are
single orbits with representatives \ $ L(1,0,0,0,1,0,0,0,0,1),\\
L(0,0,0,0,1,0,0,0,0,1),\ L(1,0,0,0,0,1,0,0,0,1),\
L(0,0,0,0,0,1,0,0,0,1),\ L(1,0,0,0,0,0,0,0,0,1),\\
L(0,0,0,0,0,0,0,0,0,1),\ L(1,0,0,0,0,0,0,1,0,0),\
L(0,0,0,0,0,0,0,1,0,0),\ L(1,0,0,1,0,0,0,0,0,0),\\
L(0,0,0,1,0,0,0,0,0,0),\ L(1,0,0,0,1,0,0,0,0,0),\
L(0,0,0,0,1,0,0,0,0,0),\ L(1,0,1,0,0,0,0,0,0,0),\\
L(0,0,1,0,0,0,0,0,0,0),\ L(0,1,0,0,0,1,0,0,0,0),\
L(0,1,0,0,0,0,0,0,0,0),\ L(1,0,0,0,0,1,0,0,0,0),\\
L(0,0,0,0,0,1,0,0,0,0),\ L(1,0,0,0,0,0,1,0,0,0),\
L(0,0,0,0,0,0,1,0,0,0),\ L(1,0,0,0,0,0,0,0,0,0)$\ and
$L(0,0,0,0,0,0,0,0,0,0),$ respectively.
\begin{proof}

To prove it, we give the appropriate values of
$A_0,A_1,A_2,A_3,B_1,B_2,B_3$ and $B_4$ in the base change
(\ref{bc.dim08})(as for other $A_i,\ i=4,...,7,$ and $ B_j,\
j=5,6,7$ they are any, except where specified otherwise).

\nt For $ U_{8}^{5}\ \rm{and}\ U_{8}^{6} :$ \\ $\ds B_{1}={\frac
{A_{{0}}+A_{{1}}c_{{3,4}}}{c_{{3,4}}}},\ \mathrm{and} \
B_3=\frac{-1}{2\,A_{{0}}{c_{{3,4}}}^{2} \left(
A_{{0}}+A_{{1}}c_{{3,4}} \right) }
({A_{{0}}}^{3}c_{{2,4}}+2\,{A_{{0}}}^{2}c_{{2,4}}A_{{1}}c
_{{3,4}}+A_{{0}}c_{{2,4}}{A_{{1}}}^{2}{c_{{3,4}}}^{2}-A_{{1}}c_{{1,3}}
{A_{{0}}}^{2}c_{{3,4}}-2\,{A_{{1}}}^{2}c_{{1,3}}A_{{0}}{c_{{3,4}}}^{2}
-{A_{{1}}}^{3}c_{{1,3}}{c_{{3,4}}}^{3}-A_{{0}}{B_{{2}}}^{2}{c_{{3,4}}}
^{3}) .$

\nt For $ U_{8}^{7}\ \rm{and}\ U_{8}^{8} :$  \\ $\ds B_{1}={\frac
{A_{{0}}+A_{{1}}c_{{3,4}}}{c_{{3,4}}}} \ \mathrm{and} \
B_3=\frac{-1}{2\,{c_{{3,4}}}^{2} \left( A_{{0}}+A_{{1}}c_{{3,4}}
\right) }
({A_{{0}}}^{2}c_{{2,4}}+2\,A_{{0}}c_{{2,4}}A_{{1}}c_{{3,4
}}+c_{{2,4}}{A_{{1}}}^{2}{c_{{3,4}}}^{2}-{B_{{2}}}^{2}{c_{{3,4}}}^{3})
 .$

\nt For $ U_{8}^{9}\ \rm{and}\ U_{8}^{10} :$   $$\ds B_{1}={\frac
{A_{{0}}+A_{{1}}c_{{3,4}}}{c_{{3,4}}}} \ \mathrm{and} \
B_3=-{\frac
{{A_{{0}}}^{2}c_{{2,4}}+2\,A_{{0}}c_{{2,4}}A_{{1}}c_{{3,4
}}+c_{{2,4}}{A_{{1}}}^{2}{c_{{3,4}}}^{2}-{B_{{2}}}^{2}{c_{{3,4}}}^{3}}
{2\,{c_{{3,4}}}^{2} \left( A_{{0}}+A_{{1}}c_{{3,4}} \right) }}
 .$$

\nt For $ U_{8}^{17}\ \rm{and}\ U_{8}^{18} :$   $$A_1={\frac
{A_{{0}}c_{{2,4}}}{3\,{c_{{2,3}}}^{2}}},\ \ B_{1}={\frac
{{A_{{0}}}^{2}}{c_{{2,3}}}},\ \ B_3={\frac
{{B_{{2}}}^{2}{c_{{2,3}}}^{3}+{A_{{0}}}^{4}c_{{1,4}}}{2\,{A_{
{0}}}^{2}{c_{{2,3}}}^{2}}},\ {\rm{and}} $$ $$ B_4=-{\frac
{-{B_{{2}}}^{3}{c_{{2,3}}}^{5}-2\,{A_{{0}}}^{6}c_{{1,5}}c
_{{2,3}}+2\,{A_{{0}}}^{6}c_{{2,4}}c_{{1,4}}-3\,B_{{2}}{A_{{0}}}^{4}c_{
{1,4}}{c_{{2,3}}}^{2}}{6\,{A_{{0}}}^{4}{c_{{2,3}}}^{3}}} .$$

\nt For $ U_{8}^{29}\ \rm{and}\ U_{8}^{30} :$   $$\ds A_1=-{\frac
{A_{{0}}c_{{1,3}}}{2\,{c_{{1,2}}}^{2}}},\ {\rm{and}}\ B_1={\frac
{{A_{{0}}}^{2}}{c_{{1,2}}}}.$$

\nt For $ U_{8}^{34}\ \rm{and}\ U_{8}^{35} :$   $$ A_1=-{\frac
{A_{{0}}c_{{1,5}}}{3\,{c_{{1,3}}}^{2}}},\ {\rm{and}}\ B_1={\frac
{{A_{{0}}}^{3}}{c_{{1,3}}}}.$$

\nt For $ U_{8}^{38}\ \rm{and}\ U_{8}^{39} :$   $$ A_1=-{\frac
{A_{{0}}c_{{0,1}}}{2\,c_{{1,1}}}},\ {\rm{and}}\ B_1={\frac
{{A_{{0}}}^{6}}{c_{{1,1}}}}.$$

\nt For $ U_{8}^{41}\ \rm{and}\ U_{8}^{42} :$   $$\ds A_1=-{\frac
{A_{{0}}c_{{0,0}}}{c_{{0,1}}}}.$$

\nt For $ U_{8}^{44}\ \rm{and}\ U_{8}^{45} :$   $$\ds B_1={\frac
{{A_{{0}}}^{4}}{c_{{1,4}}}}.$$

\nt For $ U_{8}^{46}\ \rm{and}\ U_{8}^{47} :$   $$\ds B_1={\frac
{{A_{{0}}}^{5}}{c_{{1,5}}}}.$$

\end{proof}
\end{pr}

\section{Conclusion}
\begin{enumerate}
\item In $TLeib_7$ we distinguished 26 isomorphism classes (12
parametric family and 14 concrete) of seven dimensional Leibniz
algebras and shown that they exhaust all possible cases.
 \item In
the case of $TLeib_8$ there are 49 isomorphism classes (27
parametric family and 22 concrete) and they exhaust all possible
cases.
\end{enumerate}

%
%

\end{document}